\documentclass[11pt, a4paper, twoside,leqno]{amsart}
\pdfoutput=1
\usepackage[centering, totalwidth = 400pt, totalheight = 600pt]{geometry}
\usepackage{amssymb, amsmath, amsthm}
\usepackage{enumerate, microtype, stmaryrd, url, mathtools}
\usepackage[latin1]{inputenc}
\usepackage{color}
\definecolor{darkgreen}{rgb}{0,0.45,0}
\usepackage[colorlinks,citecolor=darkgreen,linkcolor=darkgreen,hypertexnames=false]{hyperref}
\usepackage{autonum}
\usepackage[british]{babel}

\DeclareMathAlphabet{\mathbf}{OT1}{cmr}{b}{n}

\usepackage[arrow, matrix, tips, curve, graph, rotate]{xy}
\SelectTips{cm}{10}

\makeatletter
\def\matrixobject@{%
  \edef \next@{={\DirectionfromtheDirection@ }}%
  \expandafter \toks@ \next@ \plainxy@
  \let\xy@@ix@=\xyq@@toksix@
  \xyFN@ \OBJECT@}
\let\xy@entry@@norm=\entry@@norm
\def\entry@@norm@patched{%
  \let\object@=\matrixobject@
  \xy@entry@@norm }
\AtBeginDocument{\let\entry@@norm\entry@@norm@patched}
\makeatother

\newcommand{\twocong}[2][0.5]{\ar@{}[#2] \save ?(#1)*{\cong}\restore}
\newcommand{\twoeq}[2][0.5]{\ar@{}[#2] \save ?(#1)*{=}\restore}
\newcommand{\rtwocell}[3][0.5]{\ar@{}[#2] \ar@{=>}?(#1)+/l 0.2cm/;?(#1)+/r 0.2cm/^{#3}}
\newcommand{\ltwocell}[3][0.5]{\ar@{}[#2] \ar@{=>}?(#1)+/r 0.2cm/;?(#1)+/l 0.2cm/^{#3}}
\newcommand{\ltwocello}[3][0.5]{\ar@{}[#2] \ar@{=>}?(#1)+/r 0.2cm/;?(#1)+/l 0.2cm/_{#3}}
\newcommand{\dtwocell}[3][0.5]{\ar@{}[#2] \ar@{=>}?(#1)+/u  0.2cm/;?(#1)+/d 0.2cm/^{#3}}
\newcommand{\dltwocell}[3][0.5]{\ar@{}[#2] \ar@{=>}?(#1)+/ur  0.2cm/;?(#1)+/dl 0.2cm/^{#3}}
\newcommand{\drtwocell}[3][0.5]{\ar@{}[#2] \ar@{=>}?(#1)+/ul  0.2cm/;?(#1)+/dr 0.2cm/^{#3}}
\newcommand{\dthreecell}[3][0.5]{\ar@{}[#2] \ar@3{->}?(#1)+/u  0.2cm/;?(#1)+/d 0.2cm/^{#3}}
\newcommand{\utwocell}[3][0.5]{\ar@{}[#2] \ar@{=>}?(#1)+/d 0.2cm/;?(#1)+/u 0.2cm/_{#3}}
\newcommand{\dtwocelltarg}[3][0.5]{\ar@{}#2 \ar@{=>}?(#1)+/u  0.2cm/;?(#1)+/d 0.2cm/^{#3}}
\newcommand{\utwocelltarg}[3][0.5]{\ar@{}#2 \ar@{=>}?(#1)+/d  0.2cm/;?(#1)+/u 0.2cm/_{#3}}

\newdir{(}{{}*!<0em,-.14em>-\cir<.14em>{l^r}}
\newdir{ (}{{}*!/-5pt/\dir{(}}
\newdir{ >}{{}*!/-5pt/\dir{>}}
\newcommand{\sh}[2]{**{!/#1 -#2/}}

\DeclareMathOperator{\ob}{ob}

\newcommand{\cat}[1]{\mathbf{#1}}
\newcommand{\op}{\mathrm{op}}

\newcommand{\id}{\mathrm{id}}
\newcommand{\thg}{{\mathord{\text{--}}}}
\newcommand{\Lan}{\mathrm{Lan}}

\newcommand{\dbr}[1]{\left\llbracket{#1}\right\rrbracket}

\newcommand{\spn}[1]{{\langle{#1}\rangle}}

\newcommand{\defeq}{\mathrel{\mathop:}=}
\newcommand{\cd}[2][]{\vcenter{\hbox{\xymatrix#1{#2}}}}

\renewcommand{\phi}{\varphi}
\newcommand{\A}{{\mathcal A}}
\newcommand{\B}{{\mathcal B}}
\newcommand{\C}{{\mathcal C}}
\newcommand{\D}{{\mathcal D}}
\newcommand{\E}{{\mathcal E}}
\newcommand{\F}{{\mathcal F}}
\newcommand{\G}{{\mathcal G}}

\newcommand{\M}{{\mathcal M}}
\newcommand{\N}{{\mathcal N}}
\renewcommand{\O}{{\mathcal O}}
\renewcommand{\P}{{\mathcal P}}

\newcommand{\Ss}{{\mathcal S}}
\newcommand{\T}{{\mathcal T}}

\newcommand{\V}{{\mathcal V}}
\newcommand{\W}{{\mathcal W}}

\newcommand{\Z}{{\mathcal Z}}

\newcommand{\xtor}[1]{\cdl[@1]{{} \ar[r]|-{\object@{|}}^{#1} & {}}}
\newcommand{\tor}{\ensuremath{\relbar\joinrel\mapstochar\joinrel\rightarrow}}

\makeatletter

\def\hookleftarrowfill@{\arrowfill@\leftarrow\relbar{\relbar\joinrel\rhook}}
\def\twoheadleftarrowfill@{\arrowfill@\twoheadleftarrow\relbar\relbar}
\def\leftbararrowfill@{\arrowdoublefill@{\leftarrow\mkern-5mu}\relbar\mapstochar\relbar\relbar}
\def\Leftbararrowfill@{\arrowdoublefill@{\Leftarrow\mkern-2mu}\Relbar\Mapstochar\Relbar\Relbar}
\def\leftringarrowfill@{\arrowdoublefill@{\leftarrow\mkern-3mu}\relbar{\mkern-3mu\circ\mkern-2mu}\relbar\relbar}
\def\lefttriarrowfill@{\arrowfill@{\mathrel\triangleleft\mkern0.5mu\joinrel\relbar}\relbar\relbar}
\def\Lefttriarrowfill@{\arrowfill@{\mathrel\triangleleft\mkern1mu\joinrel\Relbar}\Relbar\Relbar}

\def\hookrightarrowfill@{\arrowfill@{\lhook\joinrel\relbar}\relbar\rightarrow}
\def\twoheadrightarrowfill@{\arrowfill@\relbar\relbar\twoheadrightarrow}
\def\rightbararrowfill@{\arrowdoublefill@{\relbar\mkern-0.5mu}\relbar\mapstochar\relbar\rightarrow}
\def\Rightbararrowfill@{\arrowdoublefill@{\Relbar\mkern-2mu}\Relbar\Mapstochar\Relbar\Rightarrow}
\def\rightringarrowfill@{\arrowdoublefill@\relbar\relbar{\mkern-2mu\circ\mkern-3mu}\relbar{\mkern-3mu\rightarrow}}
\def\righttriarrowfill@{\arrowfill@\relbar\relbar{\relbar\joinrel\mkern0.5mu\mathrel\triangleright}}
\def\Righttriarrowfill@{\arrowfill@\Relbar\Relbar{\Relbar\joinrel\mkern1mu\mathrel\triangleright}}

\def\leftrightarrowfill@{\arrowfill@\leftarrow\relbar\rightarrow}
\def\mapstofill@{\arrowfill@{\mapstochar\relbar}\relbar\rightarrow}

\renewcommand*\xleftarrow[2][]{\ext@arrow 20{20}0\leftarrowfill@{#1}{#2}}
\providecommand*\xLeftarrow[2][]{\ext@arrow 60{22}0{\Leftarrowfill@}{#1}{#2}}
\providecommand*\xhookleftarrow[2][]{\ext@arrow 10{20}0\hookleftarrowfill@{#1}{#2}}
\providecommand*\xtwoheadleftarrow[2][]{\ext@arrow 60{20}0\twoheadleftarrowfill@{#1}{#2}}
\providecommand*\xleftbararrow[2][]{\ext@arrow 10{22}0\leftbararrowfill@{#1}{#2}}
\providecommand*\xLeftbararrow[2][]{\ext@arrow 50{24}0\Leftbararrowfill@{#1}{#2}}
\providecommand*\xleftringarrow[2][]{\ext@arrow 10{26}0\leftringarrowfill@{#1}{#2}}
\providecommand*\xlefttriarrow[2][]{\ext@arrow 80{24}0\lefttriarrowfill@{#1}{#2}}
\providecommand*\xLefttriarrow[2][]{\ext@arrow 80{24}0\Lefttriarrowfill@{#1}{#2}}

\renewcommand*\xrightarrow[2][]{\ext@arrow 01{20}0\rightarrowfill@{#1}{#2}}
\providecommand*\xRightarrow[2][]{\ext@arrow 04{22}0{\Rightarrowfill@}{#1}{#2}}
\providecommand*\xhookrightarrow[2][]{\ext@arrow 00{20}0\hookrightarrowfill@{#1}{#2}}
\providecommand*\xtwoheadrightarrow[2][]{\ext@arrow 03{20}0\twoheadrightarrowfill@{#1}{#2}}
\providecommand*\xrightbararrow[2][]{\ext@arrow 01{22}0\rightbararrowfill@{#1}{#2}}
\providecommand*\xRightbararrow[2][]{\ext@arrow 04{24}0\Rightbararrowfill@{#1}{#2}}
\providecommand*\xrightringarrow[2][]{\ext@arrow 01{26}0\rightringarrowfill@{#1}{#2}}
\providecommand*\xrighttriarrow[2][]{\ext@arrow 07{24}0\righttriarrowfill@{#1}{#2}}
\providecommand*\xRighttriarrow[2][]{\ext@arrow 07{24}0\Righttriarrowfill@{#1}{#2}}

\providecommand*\xmapsto[2][]{\ext@arrow 01{20}0\mapstofill@{#1}{#2}}
\providecommand*\xleftrightarrow[2][]{\ext@arrow 10{22}0\leftrightarrowfill@{#1}{#2}}
\providecommand*\xLeftrightarrow[2][]{\ext@arrow 10{27}0{\Leftrightarrowfill@}{#1}{#2}}

\makeatother

\numberwithin{equation}{section}

\theoremstyle{plain}
\newtheorem{Thm}{Theorem}
\newtheorem{Prop}[Thm]{Proposition}
\newtheorem{Cor}[Thm]{Corollary}
\newtheorem{Lemma}[Thm]{Lemma}

\theoremstyle{definition}
\newtheorem{Defn}[Thm]{Definition}

\newtheorem{Ex}[Thm]{Example}

\newtheorem{Rk}[Thm]{Remark}

\newtheorem*{Defn*}{Definition}

\DeclareFontFamily{U}{mathb}{\hyphenchar\font45}
\DeclareFontShape{U}{mathb}{m}{n}{
      <5> <6> <7> <8> <9> <10> gen * mathb
      <10.95> mathb10 <12> <14.4> <17.28> <20.74> <24.88> mathb12
      }{}
\DeclareSymbolFont{mathb}{U}{mathb}{m}{n}

\DeclareMathSymbol{\curvearrowleft}{3}{mathb}{"F0}
\DeclareMathSymbol{\curvearrowright}{3}{mathb}{"F1}
\DeclareMathSymbol{\curvearrowleftright}{3}{mathb}{"F2}

\newcommand{\comp}{\circ}
\newcommand{\comm}{\odot}
\newcommand{\I}{I}
\newcommand{\conv}{\ast}
\newcommand{\J}{J}
\newcommand{\fun}{\mathop{\square}}
\newcommand{\Th}{\mathrm{Th}}
\newcommand{\esp}{\Ss p}

\begin{document}
\leftmargini=2em 
\title{Commutativity}
\author{Richard Garner} 
\address{Department of Mathematics, Macquarie University, NSW 2109, Australia} 
\email{richard.garner@mq.edu.au}
\thanks{The first author gratefully acknowledges the support of
  Australian Research Council Discovery Projects DP110102360 and
  DP130101969.}

\author{Ignacio L\'opez Franco} 
\address{Department of Pure Mathematics and Mathematical Statistics,
  University of Cambridge, Cambridge CB3 0WB, UK} 
\email{ill20@cam.ac.uk}
\thanks{The second author gratefully acknowledges the support of a Research Fellowship of
  Gonville and Caius College, Cambridge, and of the Australian Research Council
  Discovery Project DP1094883.}


\begin{abstract}
  We describe a general framework for notions of commutativity based
  on enriched category theory. We extend Eilenberg and Kelly's tensor
  product for categories enriched over a symmetric monoidal base to a
  tensor product for categories enriched over a normal duoidal
  category; using this, we re-find notions such as the commutativity
  of a finitary algebraic theory or a strong monad, the commuting
  tensor product of two theories, and the Boardman--Vogt tensor
  product of symmetric operads.
\end{abstract}
\maketitle

\section{Introduction}

This article is a category-theoretic investigation into the notion of
\emph{commutativity}. We first meet commutativity in elementary
algebra: two elements $a,b$ of a monoid $M$ are said to commute if
$ab = ba$, while $M$ itself is called commutative if all its elements
commute pairwise. This immediately yields other notions of
commutativity: for groups (on forgetting the inverses), for rings (on
forgetting the additive structure) and for Lie algebras (on passing to
the universal enveloping algebra). 

Later on, we encounter more sophisticated forms of commutativity not
directly reducible to that for monoids. For example, a pair of
operations $f, g$ of arities $m, n$ in an algebraic theory $\T$ are
said to \emph{commute} if the two $mn$-ary operations
\begin{align}
  & f(g(x_{11}, \dots, x_{1n}), \dots, g(x_{m1}, \dots, x_{mn})) \\ 
  \text{and} \qquad \ 
  & g(f(x_{11}, \dots, x_{m1}), \dots, f(x_{1n}, \dots, x_{mn}))
\end{align}
are equal, while $\T$ itself is \emph{commutative} when all of its
operations commute pairwise; typical commutative theories are those
for join-semilattices, for commutative monoids and for modules over a
commutative ring $R$. An important related notion in this context is
the \emph{commuting tensor product} $\Ss \comm \T$ of theories $\Ss$
and $\T$; this has the property that $\Ss \comm \T$-models in a
category $\E$ correspond with $\Ss$-models in the category of
$\T$-models in $\E$, and also with $\T$-models in the category of
$\Ss$-models in $\E$. There is a corresponding notion of commutativity
for operations in \emph{symmetric operads} in the sense
of~\cite{May1972The-geometry}, and the analogue of the commuting
tensor product in this context is the \emph{Boardman--Vogt} tensor
product of~\cite{Boardman1973Homotopy}.

Yet another kind of generalised commutativity arises in the context of
the \emph{sesquicategories} of~\cite{Street1996Categorical}; these may
be defined succinctly as comprising a category $\C$ together with a
lifting of
$\cat{Hom}_\C \colon \C^\op \times \C \rightarrow \cat{Set}$ through
the forgetful functor $\cat{Cat} \rightarrow \cat{Set}$. To give such
a lifting is to equip $\C$ with $2$-cells that admit vertical
composition and whiskering on each side with $1$-cells, but which
need not satisfy the \emph{interchange} axiom, which requires that for
any pair of $2$-cells in the configuration
\begin{equation}\label{eq:47}
  \cd{
    A \ar@/^1em/[r]^-{f} \ar@/_1em/[r]_-{g} \dtwocell{r}{\alpha} &
    B \ar@/^1em/[r]^-{h} \ar@/_1em/[r]_-{k} \dtwocell{r}{\beta} &
    C
  }
\end{equation}
we should have
$\beta g \circ h\alpha = k\alpha \circ \beta f \colon hf \Rightarrow
kg$.
If we declare the pair $(\alpha, \beta)$ to commute just when they
\emph{do} satisfy interchange, then a sesquicategory will be
commutative, in the sense of all of its composable pairs commuting,
precisely when it is a $2$-category. A related example involves the
\emph{premonoidal categories} of~\cite{Power1997Premonoidal}, which
bear the same relation to (non-strict) monoidal categories as
sesquicategories do to $2$-categories.

The objective of this paper is to describe an abstract framework for
commutativity that encompasses each of the examples given above, and
others besides. As a starting point, we observe that each of our
examples is concerned with a kind of structure---monoids, algebraic
theories, operads, sesquicategories---that can be viewed as a monoid
in a particular monoidal category; an ordinary monoid, is, of course,
a monoid in the cartesian monoidal $\cat{Set}$, while a finitary
algebraic theory can be seen as a monoid in the substitution monoidal
category $[\mathbb F, \cat{Set}]$, where $\mathbb F$ is the category
of functions between finite cardinals. Consequently, the key notion
of our abstract theory will be the definition, for a suitable monoidal
category $(\V, \comp, \I)$ and a monoid $C$ therein, of what it means
for a pair of generalised elements
\begin{equation}\label{eq:1}
\cd[@-0.4em]{
  {A}  \ar[dr]_-{f} & &
  {B} \ar[dl]^-{g} \\ &
  {C}
}  
\end{equation}
of $C$ to commute. As explained in~\cite{Janelidze2009Cover}, other
aspects of the theory flow easily once this definition is made: for
example, $C$ itself is commutative just when the generalised element
$1_C$ commutes with itself, while the commuting tensor product of
monoids $A$ and $B$ is the universal monoid $A \comm B$ in which $A$
and $B$ commute; other notions such as centralizers and centres also
admit expression in this generality.

When $(\V, \comp, \I)$ is a \emph{braided} monoidal category, it is easy to say
when~\eqref{eq:1} should be a commuting cospan---namely, just when the
left-hand diagram in:
\begin{equation}\label{eq:15}
  \cd[@!C@C-1.5em@R-0.7em]{
    & \sh{l}{0.5em} A \comp B \ar[r]^-{f \comp g} &
    \sh{r}{0.5em}C \comp C \ar[dr]^m & & 
        & \sh{l}{0.5em} A \comp B \ar[r]^-{f \comp g} &
    \sh{r}{0.5em} C \comp C \ar[dr]^m \\ 
    A \comp B \ar[ur]^-{1} \ar[dr]_-{c} & & & 
    C &
    A \conv B \ar[ur]^-{\sigma} \ar[dr]_-{\tau} & & & 
    C \\
    &\sh{l}{0.5em} B \comp A \ar[r]_-{g \comp f} &
    \sh{r}{0.5em}C \comp C \ar[ur]_m & & 
    &\sh{l}{0.5em} B \comp A \ar[r]_-{g \comp f} &
    \sh{r}{0.5em}C \comp C \ar[ur]_m
  }
\end{equation}
commutes in $\V$; here $c$ is the braiding of $\V$ and $m$ is the
multiplication of the monoid $C$. With $\V = (\cat{Set}, \times, 1)$,
this recovers the case of classical monoids; but it does not account
for examples---such as finitary algebraic theories---wherein $\V$ is
\emph{not} braided monoidal. The key novelty of our treatment is in
how we extend the basic commutativity notion in~\eqref{eq:1} to cases
such as these. Rather than a braiding, we assume that $\V$ is equipped
with a second monoidal structure whose tensor
$\conv \colon \V \times \V \to \V$, unit $\J \colon 1 \to \V$ and
associated coherences are normal opmonoidal with respect to the first;
this makes $(\V, \conv, \J, \comp, \I)$ into a \emph{normal
  duoidal}~\cite{Batanin2012Centers} or \emph{$2$-fold
  monoidal}~\cite{Balteanu2003Iterated} category. As we recall in
Section~\ref{sec:normality} below, a normal duoidal structure gives
rise to natural families of maps
\begin{equation}
  \sigma \colon A \conv B \rightarrow A \comp B \qquad \text{and} \qquad \tau \colon A \conv B \rightarrow A \comp B
\end{equation}
which in suitable circumstances can serve as a surrogate for a
braiding; in particular, we may generalise the notion of commuting
cospan~\eqref{eq:1} from the braided to the normal duoidal context by
replacing $1$ and $c$ as to the left of~\eqref{eq:15} by $\sigma$ and
$\tau$ as to the right. Note that this is a true generalisation: any
braided monoidal $\V$ bears a canonical normal duoidal structure with
$\conv = \comp$ and $\J = \I$ for which $\sigma$ and $\tau$ reduce
exactly to $1$ and $c$.

This more general framework for commutativity is sufficient to capture
all of the leading examples. For example, the substitution monoidal
category $([\mathbb F, \cat{Set}], \comp, \I)$, wherein monoids are
finitary algebraic theories, becomes normal duoidal when equipped with
the second monoidal structure $(\conv, \J)$ given by \emph{Day
  convolution}~\cite{Day1970Construction} with respect to product; as
we will see in Section~\ref{sec:exampl-algebr-theor}, the resultant
theory of commutativity for finitary algebraic theories is precisely
the classical one outlined above. One of the basic contentions of this
paper is consequently that \emph{normal duoidal categories are an
  appropriate environment for describing a theory of commutativity}.

In fact, this theory becomes more perspicuous if we adopt a broader
perspective. A monoid in a monoidal category $\V$ is equally well a
one-object \emph{$\V$-enriched category}~\cite{Kelly1982Basic}, and
the meaning of our theory of commutativity may be clarified by
generalising it from $\V$-monoids to $\V$-categories; the basic notion
of commuting cospan~\eqref{eq:1} is then replaced by a notion of
\emph{bifunctor} between $\V$-categories---which we now explain.

Consider first the case where $\V = \cat{Set}$; here a $\V$-category
is just an ordinary (locally small) category, and we are familiar with
the fact that a \emph{bifunctor} from $\A, \B$ to $\C$ is simply a
functor $T \colon \A \times \B \rightarrow \C$. An important basic
exercise (see~\cite[\S II.3, Proposition 1]{Mac-Lane1971Categories})
shows that giving the data of a bifunctor is equivalent to giving
families of functors $T(a, \thg) \colon \B \rightarrow \C$ for
$a \in \A$ and $T(\thg, b) \colon \A \rightarrow \C$ for $b \in \B$
that agree on objects---so $T(\thg,b)(a) = T(a, \thg)(b) = T(a,b)$,
say---and which on arrows satisfy the \emph{commutativity} condition
that, for all $f \colon a \rightarrow a'$ in $\A$ and
$g \colon b \rightarrow b'$ in $\B$, the following square should
commute in $\C$:
\begin{equation}\label{eq:16}
  \cd[@-0.3em]{
    {T(a,b)} \ar[r]^-{T(f,b)} \ar[d]_{T(a,g)} &
    {T(a',b)} \ar[d]^{T(a',g)} \\
    {T(a,b')} \ar[r]_-{T(f,b')} &
    {T(a',b')\rlap{ .}}
  }
\end{equation}

More generally, for any braided monoidal $\V$, there is a
well-established notion of $\V$-bifunctor; as explained
in~\cite[Chapter~III, \S4]{Eilenberg1966Closed}, it may again be
described in two ways. On the one hand, the existence of the braiding
means that there is an easily-defined notion of \emph{tensor product} for
$\V$-categories~\cite[\S 1.4]{Kelly1982Basic}, and a $\V$-bifunctor
from $\A, \B$ to $\C$ may now be defined simply as a $\V$-functor
$T \colon \A \otimes \B \rightarrow \C$. On the other hand, we may
once again specify a bifunctor in terms of families of $\V$-functors
$T(a, \thg) \colon \B \rightarrow \C$ and
$T(\thg, b) \colon \A \rightarrow \C$ which match up on objects, and
which satisfy a commutativity condition like~\eqref{eq:16}, but now
expressed in terms of commuting diagrams of hom-objects in $\V$:
\begin{equation}\label{eq:17}
  \cd[@!C@C-9.7em@R-0.5em]{
    & \sh{l}{2.5em}\A(a, a') \comp \B(b, b') \ar[rr]^-{T(\thg, b') \comp T(a, \thg)} & &
    \sh{r}{3em}\C(Tab', Ta'b') \comp \C(Tab, Tab') \ar[dr]^m \\ 
    \A(a, a') \comp \B(b, b') \ar[ur]^-{1} \ar[dr]_-{c} & & & &
    \C(Tab, Ta'b') \\
    &\sh{l}{2.5em}\B(b, b') \comp \A(a, a') \ar[rr]_-{T(a', \thg) \comp T(\thg, b)} & &
    \sh{r}{3em}\C(Ta'b, Ta'b') \comp \C(Tab, Ta'b)\rlap. \ar[ur]_m
  }
\end{equation}
Comparing~\eqref{eq:17} to the left-hand diagram of~\eqref{eq:15},
we see that the latter is simply the one-object case of the former, so
that, for braided monoidal categories, the commuting cospans
of~\eqref{eq:1} are \emph{the same} as $\V$-bifunctors between
one-object $\V$-categories. 

Given this, there is now an obvious way of generalising the notion of
$\V$-bifunctor from the braided monoidal to the normal duoidal
context: we define a bifunctor from $\A, \B$ to $\C$ in terms of
families of one-variable $\V$-functors $T(\thg, b)$ and $T(a, \thg)$
which agree on objects, and which satisfy the commutativity condition
obtained from~\eqref{eq:17} by replacing $1$ and $c$ therein with
$\sigma$ and $\tau$---just as we did in~\eqref{eq:15}. Under
reasonable hypotheses on $\V$, such $\V$-bifunctors
$\A, \B \rightarrow \C$ can be represented by $\V$-functors
$\A \comm \B \rightarrow \C$---and this tensor producd $\A \comm \B$
of $\V$-categories is now the many-object version of the
\emph{commuting tensor product} discussed above. Under further
reasonable hypotheses on $\V$, these commuting tensor products can be
made into part of a monoidal biclosed structure on the $2$-category of
small $\V$-categories, generalising the one existing in the braided
monoidal case.

The many-object perspective clarifies not only the basic commutativity
notion and the commuting tensor product, but also various further
aspects of the theory. For example, as we will see in
Section~\ref{sec:comm-tens-prod}, we may exploit the internal homs
$[\thg, \thg]$ of the commuting tensor product to construct
\emph{centralizers} of monoid maps and \emph{centres} of monoids.
Similarly, we will see in Section~\ref{sec:catend_fc-enrich-cat} that in the
case $\V = [\mathbb F, \cat{Set}]$, we may realise the category of
models of a theory $\T$ in a category $\E$ with finite powers as an
internal hom $[\T, \E]$ in $\V\text-\cat{Cat}$; the correspondence
between $\Ss \comm \T$-models in $\E$, $\Ss$-models in $\T$-models
in $\E$, and $\T$-models in $\Ss$-models in $\E$, is then a direct
consequence of the isomorphisms
$[\Ss \comm \T, \E] \cong [\Ss, [\T, \E]] \cong [\T, [\Ss, \E]]$
associated to the symmetric monoidal closed structure on
$\V\text-\cat{Cat}$.

It may also be useful to note what we do \emph{not} do in this paper.
First, we will say nothing about notions of commutativity in the
context of \emph{semi-abelian} categories~\cite{Borceux2004Malcev}.
This is because, as far as we have been able to tell, examples from
this sphere simply do not fit into our framework. The intersection of
our theory and the semi-abelian theory is essentially the content
of~\cite{Janelidze2009Cover}, which shows how to define notions such
as commuting tensor product and commutative object in a category
\emph{given}, as in~\eqref{eq:1}, a suitable commutation relation on
cospans. We exploit some of these results in
Section~\ref{sec:one-object-case} below, but note that this only
relates to the one-object case of our theory; for the many-object case
we must argue from scratch.

A second point we do not touch on, purely for reasons of space, is the
generalisation of our theory from enrichment over monoidal categories
to enrichment over \emph{bicategories}~\cite{Walters1981Sheaves}.
While it may sound esoteric, such a generalisation would allow us, for
example, to exploit the work of~\cite{Fiore2008The-cartesian} in order
to describe not only the Boardman--Vogt tensor product of symmetric
operads but also that of \emph{symmetric
  multicategories}~\cite{Lambek1969Deductive}.

The final point we do not deal with, again for reasons of space, is
the generalisation of our theory from one-dimensional to
two-dimensional enrichment. The basic example is the category of small
$2$-categories, which as well as its commuting tensor product (=
cartesian product) also admits a ``pseudo-commuting'' tensor product
known as the \emph{Gray tensor product}~\cite{Gordon1995Coherence},
together with lax and oplax variants thereof~\cite{Gray1974Formal}.
Likewise, when we generalise from algebraic theories to
\emph{two-dimensional} algebraic
theories~\cite{Blackwell1989Two-dimensional}, we have not just a
commuting tensor product but also pseudo, lax, and oplax variants. We
hope to deal with both this generalisation and the preceding one in
future work.

The remainder of this paper is laid out as follows. We begin in
Section~\ref{sec:duoidal-categories} by gathering together the
necessary background on duoidal categories.
Section~\ref{sec:comm-gener-theory} then introduces the key notions of
our theory of commutativity: the notions of sesquifunctor and
bifunctor for categories enriched over a normal duoidal $\V$, and a
description of the monoidal closed structure this induces on
$\V\text-\cat{Cat}$ for a well-behaved $\V$.
Section~\ref{sec:one-object-case} then goes on to indicate how these
results specialise to the important case of one-object
$\V$-categories. This concludes the abstract theory; the remainder of
the paper is devoted to examples.

Section~\ref{sec:exampl-algebr-theor} considers the case of
\emph{finitary algebraic theories}, showing that our framework
suffices to re-find the classical notions of commutativity described
above; Section~\ref{sec:exampl-symm-oper} then considers
\emph{symmetric operads}, in particular showing how the Boardman--Vogt
tensor product mentioned above falls out of our theory.
Section~\ref{sec:norm-duoid-categ} breaks off from our main
development to describe a process by which arbitrary duoidal
categories can be \emph{normalized} into normal duoidal ones; we then
make use of this construction in giving our final three examples. In
Section~\ref{sec:exampl-comm-monads}, we study \emph{strong monads} on
a monoidal category~\cite{Kock1970Monads}; in
Section~\ref{sec:bistrong-promonads}, we generalise this to the
\emph{Freyd-categories}
of~\cite{Power1997Premonoidal,Levy2003Modelling}; while finally in
Section~\ref{sec:exampl-sesq}, we incorporate the example of
sesquicategories into our general framework.

\section{Background on duoidal categories}
\label{sec:duoidal-categories}

\subsection{Duoidal categories}
\label{sec:duoidal-categories-2}
As explained in the introduction, the ambient setting for our theory
of commutativity is that of a \emph{duoidal category}. These were
introduced in a slightly degenerate form
in~\cite{Balteanu2003Iterated} under the name \emph{$2$-fold monoidal
  categories}; the fully general definition may be found, for example,
in~\cite{Aguiar2010Monoidal} under the name \emph{$2$-monoidal
  category}. The term ``duoidal'' is due to~\cite{Batanin2012Centers}.

\begin{Defn}
  \label{def:1}
  A \emph{duoidal category} is a monoidale (= pseudomonoid) in the
  monoidal $2$-category of monoidal categories, oplax monoidal
  functors and oplax monoidal natural transformations.
\end{Defn}

A duoidal structure on a category $\V$ thus comprises two monoidal
structures $(\comp, \I)$ and $(\conv, \J)$ (whose unit and
associativity constraints we leave unnamed) such that the functors
$\conv \colon \V \times \V \to \V$ and $\J \colon \cat 1 \to \V$ and
the associated coherence transformations are oplax monoidal with
respect to $\comp$. The oplax monoidal constraint data of $\conv$ and
$\J$ comprise a natural family of \emph{interchange maps}
\begin{equation}\label{eq:5}
\xi \colon (X \comp Y) \conv (Z \comp W) \to (X \conv Z) \comp (Y
\conv W)
\end{equation}
together with arrows $ \mu \colon \I \conv \I \to \I$,
$\upsilon \colon \J \to \I$ and $ \gamma \colon \J \to \J \comp \J$
satisfying axioms which, among other things, make $(\I, \upsilon, \mu)$
into a $\conv$-monoid and $(\J, \upsilon, \gamma)$ into a
$\comp$-comonoid. These data and axioms are equally those
required to make $\comp$ and $\I$ and the associated constraints
\emph{lax} monoidal with respect to $(\conv, \J)$, so that a duoidal
category is alternatively a monoidale in the $2$-category
of monoidal categories and lax monoidal
functors.

\begin{Ex}
  \label{ex:1}
  Any braided monoidal category $(\V, \otimes, I)$ can be made into a
  duoidal category by taking $\mathord \comp = \mathord \conv =
  \mathord \otimes$ and
  $\I = \J$, with $\upsilon$ taken to be the identity, $\mu$ and
  $\gamma$ given by unit constraints, and the interchange maps $\xi$
  constructed from associativities and the braiding. Conversely, if
  the duoidal $\V$ has all its constraint maps invertible, then the
  monoidal structures $\comp$ and $\conv$ are isomorphic and braided;
  see~\cite[Remark~5.1]{Joyal1993Braided}
  or~\cite[Proposition~6.11]{Aguiar2010Monoidal}.
\end{Ex}

We will give more examples relevant to our theory from
Section~\ref{sec:exampl-algebr-theor} onwards. In these examples, the
two monoidal structures $\comp$ and $\conv$ are generally thought of
as \emph{composition} and \emph{convolution} respectively; it is
almost always the case that the convolution tensor has associated
internal homs, and will often be the case that it is braided in a
manner compatible with $\comp$. The following definition formalises
these concepts.

\begin{Defn}
  \label{def:6}
Let $\V$ be a duoidal category. We say that:
\begin{enumerate}[(i)]
\item $\V$ is \emph{$\conv$-biclosed} if each functor $(\thg) \conv X$
  and $X \conv (\thg) \colon \V \rightarrow \V$ has a right adjoint,
  written as $[X, \thg]_\ell$ and $[X, \thg]_r$ and called \emph{left}
  and \emph{right} hom, respectively.\vskip0.25\baselineskip
\item $\V$ is \emph{$\conv$-braided} if the $\conv$-monoidal structure
  is given a braiding $c$, with respect to which the $\comp$-monoidal
  structure maps are braided monoidal; we may say
  \emph{$\conv$-symmetric} if the given braiding is in fact a symmetry.
\end{enumerate}
\end{Defn}
Spelling out (ii), the compatibility of the $\conv$-braiding and the
$\comp$-monoidal structure amounts to the requirement that $(\I, \upsilon, \mu)$
be a commutative $\conv$-monoid in $\V$, and that each diagram of the
following form should commute:
\begin{equation}\label{eq:2}
  \cd{
    {(X \comp Y) \conv (Z \comp W)} \ar[r]^-{\xi} \ar[d]_{c} &
    {(X \conv Z) \comp (Y
      \conv W)} \ar[d]^{c \,\comp\, c} \\
    {(Z \comp W) \conv (X \comp Y)} \ar[r]^-{\xi} &
    {(Z \conv X) \comp (W \conv Y)\rlap{ .}}
  }
\end{equation}
For instance, if $(\V, \otimes)$ is a braided monoidal category,
then the associated duoidal category $(\V, \otimes, \otimes)$ is
$\conv$-braided if and only if the braiding on $\otimes$ is a symmetry;
this is proven in~\cite[Proposition~6.13]{Aguiar2010Monoidal}.

\subsection{Normality}
\label{sec:normality}
It turns out that not every duoidal category will be appropriate for
our theory; we must assume essentially the same degeneracy with
respect to units as appears in~\cite{Balteanu2003Iterated}. Recall
that an opmonoidal functor $F$ is called \emph{normal} if the unit
comparison map $FI \to I$ is invertible.

\begin{Defn}
  \label{def:2}
  A duoidal category $\V$ is called \emph{normal} if the opmonoidal
  functors $\conv \colon \V \times \V \to \V$ and
  $\J \colon \cat{1} \to \V$ are normal.
\end{Defn}

In elementary terms, the normality of the duoidal $\V$ amounts to the
requirements that the morphisms $\upsilon \colon \J \to \I$ and
$\mu \colon \I \conv \I \to \I$ be invertible; in fact, invertibility of
$\upsilon$ easily implies that of $\mu$ and also of $\gamma$. In a
normal duoidal category, there are maps
$\sigma \colon X \conv Y \to X \comp Y$ and
$\tau \colon X \conv Y \to Y \comp X$ given by
\begin{equation}\label{eq:21}
  \begin{aligned}
    \sigma &= X \conv Y \xrightarrow{\cong} (X \comp \I) \conv (\I
    \comp Y) \xrightarrow{\xi} (X \conv \I) \comp (\I \conv Y)
    \xrightarrow{\cong} X \comp Y\\
    \tau &= X \conv Y \xrightarrow{\cong} (\I \comp
    X) \conv (Y \comp \I) \xrightarrow{\xi} (\I \conv Y) \comp (X
    \conv \I) \xrightarrow{\cong} Y \comp X\rlap{ ,}
  \end{aligned}
\end{equation}
where the unnamed isomorphisms are built from unit constraints for
$\comp$ and $\conv$ and the inverse of $\upsilon \colon \J \to \I$.
These maps play a central role in the theory that follows. For the
canonical normal duoidal structure on a braided monoidal category,
$\sigma$ and $\tau$ are the identity map and the braiding
respectively.

More generally, any normal duoidal $\V$ possesses the following
families of maps, which are the \emph{linear distributivities} of
\cite{Cockett1997Weakly} (there called ``weak distributivities'').
While the statements of our main definitions and results will not make
use of these, the proofs will.
\begin{align}
  & \delta^\ell_{\ell} \colon X \conv (Y \comp Z) \xrightarrow{\cong} (X \comp \I) \conv (Y \comp Z)
    \xrightarrow{\xi} (X \conv Y) \comp (\I \conv Z) \xrightarrow{\cong}
    (X \conv Y) \comp Z\\
  & \delta^\ell_{r} \colon X \conv (Y \comp Z) \xrightarrow{\cong} (\I \comp X) \conv (Y \comp Z)
    \xrightarrow{\xi} (\I \conv Y) \comp (X \conv Z) \xrightarrow{\cong}
    Y \comp (X \conv Z)\\
  & \delta^r_{\ell} \colon (X \comp Y) \conv Z \xrightarrow{\cong} (X \comp Y) \conv (Z \comp \I)
    \xrightarrow{\xi} (X \conv Z) \comp (Y \conv \I) \xrightarrow{\cong}
    (X \conv Z) \comp Y\\
  & \delta^r_{r} \colon (X \comp Y) \conv Z \xrightarrow{\cong} (X \comp Y) \conv (\I \comp Z)
    \xrightarrow{\xi} (X \conv \I) \comp (Y \conv Z) \xrightarrow{\cong}
    X \comp (Y \conv Z)\rlap{ .}
\end{align}

Note that when $\V$ is $\conv$-braided, the maps $\delta^\ell_\ell$ and
$\delta^{\ell}_{r}$ above, and similarly $\delta^r_\ell$ and
$\delta^{r}_{r}$, may be derived from each other using the braiding. A
particular case of this is that the maps $\sigma$ and $\tau$ are
related through the braiding $c$ by a commuting diagram
\begin{equation}\label{eq:4}
  \cd{
    & 
    X \conv Y \ar[dl]_-{\sigma} \ar[dr]^-{\tau} \ar[d]_-{c} \\
    X \comp Y &
    Y \conv X \ar[l]_-{\tau} \ar[r]^-{\sigma} &
    Y \comp X\rlap{ .}
  }
\end{equation}

\subsection{Bimonoids and duoids}
\label{sec:bimonoids-duoids}
Some kinds of structure definable using a braiding on a monoidal
category can be defined more generally using the interchange maps of a
duoidal structure; two examples relevant for us are the notions of
\emph{bialgebra} and of \emph{commutative monoid}. The key to the
generalisation is the fact that, since the $\comp$-monoidal structure
of a duoidal $\V$ is lax $\conv$-monoidal, it lifts to a
$\comp$-monoidal structure on the category of $\conv$-monoids in $\V$.
\begin{Defn}
  \label{def:10}\cite[Definitions~6.25~and~6.28]{Aguiar2010Monoidal}
Let $\V$ be a duoidal category.
  \begin{enumerate}[(i)]
  \item A \emph{bimonoid} in $\V$ is a $\comp$-comonoid in the
    category of $\conv$-monoids in $\V$. The category of bimonoids
    $\cat{Bimon}(\V)$ is the category
    $\cat{Comon}_\comp(\cat{Mon}_\conv(\V))$.
\vskip0.25\baselineskip
  \item A \emph{duoid} in $\V$ is a $\comp$-monoid in the category of
    $\conv$-monoids in $\V$. The category of duoids $\cat{Duoid}(\V)$
    is the category $\cat{Mon}_\comp(\cat{Mon}_\conv(\V))$.
  \end{enumerate}
  If $\V$ is $\conv$-braided, then we declare a bimonoid or duoid to be
  \emph{$\conv$-commutative} if it its underlying $\conv$-monoid is commutative.
\end{Defn}
Spelling these definitions out in more detail, a bimonoid is thus an
object $A$ equipped with $\conv$-monoid and $\comp$-comonoid
structures $e \colon \J \rightarrow A \leftarrow A \conv A \colon m$
and $u \colon \I \leftarrow A \rightarrow A \comp A \colon d$ which
are such that $e \colon \J \rightarrow A$ is a map of
$\comp$-comonoids, $u \colon A \rightarrow \I$ is a map of
$\conv$-monoids, and the following \emph{bialgebra} diagram commutes:
\begin{equation}
  \cd{
    A \conv A \ar[d]_-{d \conv d} \ar[r]^-{m} &
    A \ar[r]^-{d} & A \comp A \\
    (A \comp A) \conv (A \comp A) \ar[rr]^-{\xi} & &
    (A \conv A) \comp (A \conv A) \ar[u]_-{m \comp m}\rlap{ .}
  }
\end{equation}

On the other hand, a duoid is an object $A$ equipped with
$\comp$-monoid and $\conv$-monoid structures
$e \colon \I \rightarrow A \leftarrow A \comp A \colon m$ and
$\iota \colon \J \rightarrow A \leftarrow A \conv A \colon \nu$, which
are such that $e \colon \I \rightarrow A$ is a map of $\conv$-monoids,
$\iota \colon \J \rightarrow A$ is a map of $\comp$-monoids, 
and the following \emph{duoid} diagram commutes:
\begin{equation}\label{eq:7}
  \cd{
    (A \comp A) \conv (A \comp A) \ar[r]^-{\xi} \ar[d]_-{m \conv m} &
    (A \conv A) \comp (A \conv A) \ar[r]^-{\nu \comp \nu} &
    A \comp A \ar[d]^-{m} \\
    A \conv A \ar[rr]_-{\nu} & & A\rlap{ .}
  }
\end{equation}
When the duoidal structure on $\V$ is induced by a braided monoidal
structure, bimonoids are bialgebras in $\V$ in the usual sense, while
duoids reduce by the Eckmann--Hilton argument to commutative monoids.

\section{Commutativity: the general theory}
\label{sec:comm-gener-theory}
In this section, we introduce our abstract framework for
commutativity. As explained in the introduction, the central notion is
that of a \emph{bifunctor} between categories enriched over a normal
duoidal category $\V$; we introduce this, and describe circumstances
under which bifunctors between small $\V$-categories are represented
by a monoidal closed structure on the $2$-category of $\V$-categories.

\subsection{Sesquifunctors}
\label{sec:sesquifunctors}
To start with we assume only that $(\V, \comp, I)$ is a monoidal
category; shortly, we will add a second monoidal structure making $\V$
normal duoidal, but even then our convention will be that a
\emph{$\V$-category} is one enriched in $(\V, \comp, \I)$. We write
$\V\text-\cat{Cat}$ and $\V\text-\cat{CAT}$ for the $2$-categories of
small and large $\V$-categories, together with the
$\V$-functors and $\V$-natural transformations between them;
see~\cite[\S1.2]{Kelly1982Basic} for the full definitions.

The notion of $\V$-bifunctor we introduce is expressed in terms of
families of one-variable $\V$-functors satisfying a commutativity or
bifunctoriality condition. While the commutativity condition requires
a normal duoidal structure, the rest of the definition does not; we
begin, therefore, with this.
\begin{Defn}
  \label{def:15}
  Let $\A$, $\B$ and $\C$ be $\V$-categories.
  \begin{itemize}
  \item A \emph{sesquifunctor} $T \colon \A, \B \rightarrow \C$
    comprises families
    $(T(a, \thg) \colon \B \rightarrow \C)_{a \in \A}$ and
    $(T(\thg, b) \colon \A \rightarrow \C)_{b \in \B}$ of
    $\V$-functors such that 
    $T(a, \thg)(b) = T(\thg, b)(a) = Tab$, say.\vskip0.25\baselineskip
  \item A \emph{sesquitransformation}
    $\alpha \colon S \Rightarrow T \colon \A, \B \rightarrow \C$
    comprises families of $\V$-natural transformations
    $\alpha_{a\thg} \colon S(a, \thg) \Rightarrow T(a, \thg)$ and
    $\alpha_{\thg b} \colon S(\thg, b) \Rightarrow T(\thg, b)$ such
    that
    $(\alpha_{a\thg})_b = (\alpha_{\thg b})_a = \alpha_{ab} \colon I
    \rightarrow \C(Sab, Tab)$.
\end{itemize}
\end{Defn}
Given a sesquifunctor $T \colon \A, \B \rightarrow \C$ and
$\V$-functors $F \colon \A' \rightarrow \A$ and
$G \colon \B' \rightarrow \B$ and $H \colon \C \rightarrow \C'$, there
is a composite sesquifunctor
$HT(F,G) \colon \A', \B' \rightarrow \C'$ with
components $HT(Fa', G\thg) \colon \B' \rightarrow \C'$ and
$HT(F\thg, Gb') \colon \A' \rightarrow \C'$; with the obvious
extension of this composition to transformations, we obtain a $2$-functor
\begin{equation}
  \cat{SESQ}(\thg, \thg; \thg) \colon \V\text-\cat{CAT}^\op \times \V\text-\cat{CAT}^\op \times \V\text-\cat{CAT} \rightarrow \cat{CAT}\rlap{ .}
\end{equation}

There are also higher arity analogues of sesquifunctors and
sesquitransformations; the general pattern may be deduced if we describe
the next simplest case:
\begin{Defn}
  \label{def:16}
  A \emph{ternary sesquifunctor} $T \colon \A, \B, \C \rightarrow \D$
  between $\V$-categories comprises families of sesquifunctors
  $T(a, \thg, \thg) \colon \B, \C \rightarrow \D$ and
  $T(\thg, b, \thg) \colon \A, \C \rightarrow \D$ and
  $T(\thg, \thg, c) \colon \A, \B \rightarrow \D$ which are compatible
  on objects, in the sense that the $\V$-functors $T(a,b,\thg)$ and
  $T(a, \thg, c)$ and $T(\thg, b, c)$ are unambiguously defined.
\end{Defn}
These higher arity sesquifunctors and transformations also compose;
for example, given sesquifunctors $T \colon \C, \D \rightarrow \E$ and
$S \colon \A, \B \rightarrow \C$, the ternary sesquifunctor
$T(S,1) \colon \A, \B, \D \rightarrow \E$ has components
$T(S(a,\thg), 1)$ and $T(S(\thg, b), 1)$ and $T(S\thg, d)$. Taken
together, these compositions make the totality of $\V$-categories,
sesquifunctors and sesquitransformations into a \emph{symmetric
  $2$-multicategory} $\cat{SESQ}$. This structure on sesquifunctors is studied in some
detail, and in a somewhat more general context, in~\cite{Weber2013Free}; what is relevant here is that the
$2$-multicategory $\cat{Sesq}$ of sesquifunctors between small
$\V$-categories is often represented by a monoidal structure on
$\V\text-\cat{Cat}$. The following result summarises the key points;
in (c), we write $\mathcal I$ for the $\V$-category with a single
object $\ast$ and $\mathcal I(\ast, \ast) = \I$.

\begin{Prop}
  \label{prop:13}
  Let  $\A$, $\B$ and $\C$ be
  $\V$-categories, with $\A$ and $\B$ small.
  \begin{enumerate}[(a)]
  \item\label{item:1} If $\V\text-\cat{Cat}$ has conical colimits, preserved by the
    inclusion into $\V\text-\cat{CAT}$, then the $2$-functor
    $\cat{SESQ}(\A, \B; \thg) \colon \V\text-\cat{CAT} \rightarrow
    \cat{CAT}$
    admits a small representation
    $\A \fun \B$.\vskip0.25\baselineskip
  \item If $\V$ is complete, then the $2$-functors $\cat{SESQ}(\B, \thg\,; \C) \colon \V\text-\cat{CAT}^\op \rightarrow
    \cat{CAT}$ and
    $\cat{SESQ}(\thg, \B; \C) \colon \V\text-\cat{CAT}^\op \rightarrow
    \cat{CAT}$ admit a common representing object $\dbr{\B, \C}$,
    which is small whenever $\C$ is so.\vskip0.25\baselineskip
  \item If (a) and (b) hold, then
    $(\V\text-\cat{Cat},\, \fun,\, \mathcal I)$ is a symmetric
    monoidal closed $2$-category.
  \end{enumerate}
\end{Prop}
\begin{proof}
  For (a), the universal sesquifunctor
  $V \colon \A, \B \rightarrow \A \fun \B$ is  the
  pushout:
\begin{equation}\label{eq:18}
  \cd[@C-2.7em@-0.6em@!C]{
    & \sum_{a,b \in \A \times \B} \mathcal I \ar[dl]_-{\ell} \ar[dr]^-{r} \\
    \sum_{a \in \A} \B \ar[dr]_-{\spn{V(a, \thg)}_{a \in \A}} & & \sum_{b\in \B} \A \ar[dl]^-{\spn{V(\thg, b)}_{b \in \B}} \\ &    \A \fun \B
  }
\end{equation}
in $\V\text-\cat{Cat}$. Here, the $(a,b)$-components of $\ell$ and $r$
pick out the object $b$ in the $a$th copy of $\B$, respectively the
object $a$ in the $b$th copy of $\A$. 

For (b), we take $\dbr{\B, \C}$ to have $\V$-functors
$F \colon \B \to \C$ as objects, and hom-objects given by
$\dbr{\B,\C}(F,G) = \prod_{a \in \B}\C(Fa,Ga)$ with componentwise
composition. Clearly $\dbr{\B, \C}$ is small if $\C$ is so. The universal
sesquifunctor $\varepsilon \colon \dbr{\B, \C}, \B \rightarrow \C$ has
$\varepsilon(F, \thg) = F \colon \B \rightarrow \C$ and
$\varepsilon(\thg, b) = \mathrm{ev}_b \colon \dbr{\B, \C} \rightarrow
\C$; the universal $\B, \dbr{\B, \C} \rightarrow \C$ is dual.

For (c), we obtain the unit constraints $\lambda$ and $\rho$ of the
desired monoidal structure directly on taking $\A = \mathcal I$ or
$\B = \mathcal I$ in~\eqref{eq:18}, while the symmetry constraints are
immediate from \eqref{eq:18}'s symmetry in $\A$ and $\B$. As for
associativity, let $\A$, $\B$ and $\C$ be small $\V$-categories; it is
easy to see that composition with the universal map
$\varepsilon \colon \dbr{\C, \D}, \C \rightarrow \D$ induces
bijections, $2$-natural in $\D$, of the form:
\begin{equation}\label{eq:24}
  \cat{SESQ}(\A, \B; \dbr{\C, \D}) \xrightarrow{\cong} \cat{SESQ}(\A, \B, \C; \D)\rlap{ ,}
\end{equation}
where on the right we have the category of ternary sesquifunctors and
sesquitransformations. Since
$\cat{SESQ}(\A, \B; \dbr{\C, \D}) \cong \V\text-\cat{CAT}(\A \fun \B,
\dbr{\C, \D}) \cong \V\text-\cat{CAT}((A \fun \B) \fun \C, \D)$,
we see that $(\A \fun \B) \fun \C$ classifies ternary sesquifunctors;
by symmetry, so too does $\A \fun (\B \fun \C)$, whence there is a
unique isomorphism
$\alpha \colon (\A \fun \B) \fun \C \cong \A \fun (\B \fun \C)$
commuting with the universal maps. The triangle axiom is now easy to
verify, while the pentagon axiom follows by arguing that each vertex
of the pentagon represents quaternary sesquifunctors, and each edge
commutes with the universal maps.
\end{proof}
\begin{Rk}
  \label{rk:1}
  The observation in (c) above that the maps~\eqref{eq:24} are
  bijective is part of the fact that $\cat{SESQ}$ is a \emph{closed
    $2$-multicategory}; as explained in~\cite{Manzyuk2012Closed},
  this, together with the \emph{weak representability} of multimaps
  exhibited in (a), allows the associativity of the monoidal structure
  to be derived in a purely formal manner. We will use a similar
  argument in the proof of Proposition~\ref{prop:17} below.
\end{Rk}

When $\V = \cat{Set}$, the symmetric monoidal closed structure on
$\cat{Cat}$ this proposition yields is the ``funny tensor
product''~\cite{Street1996Categorical}, whose internal hom
$\dbr{\B, \C}$ is the category of functors and
not-neccessarily-natural transformations $\B \rightarrow \C$. This
and the cartesian structure are in fact the only symmetric
monoidal closed structures on $\cat{Cat}$; see~\cite{Foltz1980171}.

\subsection{Bifunctors}
\label{sec:bifunctors}
We are now ready to state the central definition of our theory: that
of a commuting sesquifunctor, or \emph{bifunctor}. In order to do so,
we henceforth assume that $\V$ is a normal duoidal category
$(\V, \conv, \J, \comp, \I)$; we reiterate that, in this context,
``$\V$-category'' will mean ``category enriched in
$(\V, \comp, \I)$''.

\begin{Defn}
  \label{def:17}
  Let $\V$ be a normal duoidal category, and let $\A,\B,\C$ be
  $\V$-categories. A sesquifunctor $T \colon \A,\B \rightarrow \C$ is
  said to \emph{commute}, or to be a \emph{bifunctor}, if for each
  $a,a' \in \A$ and $b,b' \in \B$, the diagram
\begin{equation}\label{eq:19}
  \cd[@!C@C-9.9em@R-0.8em]{
    & \sh{l}{2.5em}\A(a, a') \comp \B(b, b') \ar[rr]^-{T(\thg, b') \comp T(a, \thg)} & &
    \sh{r}{3em}\C(Tab', Ta'b') \comp \C(Tab, Tab') \ar[dr]^m \\ 
    \A(a, a') \conv \B(b, b') \ar[ur]^-{\sigma} \ar[dr]_-{\tau} & & & &
    \C(Tab, Ta'b') \\
    &\sh{l}{2.5em}\B(b, b') \comp \A(a, a') \ar[rr]_-{T(a', \thg) \comp T(\thg, b)} & &
    \sh{r}{3em}\C(Ta'b, Ta'b') \comp \C(Tab, Ta'b) \ar[ur]_m
  }
\end{equation}
commutes in $\V$; here $\sigma$ and $\tau$ are the maps
defined in~\eqref{eq:21}.
\end{Defn}
When the normal duoidal structure on $\V$ comes from a braided
monoidal structure, the diagram~\eqref{eq:19} reduces to
the~\eqref{eq:17} of the introduction, and so our bifunctors reduce to
those of~\cite[Chapter~III, \S4]{Eilenberg1966Closed}; in particular,
when $\V = \cat{Set}$, the bifunctoriality of
$T \colon \A, \B \rightarrow \C$ is the familiar requirement that
each~\eqref{eq:16} should commute in $\C$.

\begin{Prop}
  \label{prop:3}
  Let $\V$ be a normal duoidal category, and suppose given
  $\V$-categories, $\V$-functors and $\V$-sesquifunctors as in:
  \begin{equation}
    F \colon \A' \rightarrow \A \qquad G \colon \B' \rightarrow \B \qquad T \colon \A, \B \rightarrow \C \qquad H \colon \C \rightarrow \C'\rlap{ .}
  \end{equation}
  \begin{enumerate}[(i)]
  \item If $T$ commutes, then so does $HT(F,G)$.
  \item If $HT$ commutes and $H$ is faithful, then
    $T$ commutes.
  \item If $\V$ is $\conv$-braided, then $T$ commutes if and only if
    $T^c \colon \B, \A \rightarrow \C$ does so.
  \end{enumerate}
\end{Prop}
In (ii), we call a functor \emph{faithful} if the
morphisms in $\V$ expressing its action on hom-objects are 
monomorphic. In (iii), we write $T^c$ for the sesquifunctor with
components $T^c(\thg, a) = T(a, \thg)$ and $T^c(b, \thg) = T(\thg, b)$.
\begin{proof}
  Parts (i) and (ii) are immediate on observing that pre- and
  postcomposing~\eqref{eq:19} for $T$ by $F_{aa'} \conv G_{bb'}$ and
  $H_{Tab,Ta'b'}$ yields~\eqref{eq:19} for $HT(F,G)$. Part (iii) is
  verified by precomposing~\eqref{eq:19} with
  $c \colon \B(b,b') \conv \A(a,a') \rightarrow \A(a,a') \conv
  \B(b,b')$ and using~\eqref{eq:4}.
\end{proof}
It follows from this proposition that we have a $2$-functor
\begin{equation}
  \cat{BIFUN}(\thg, \thg; \thg) \colon \V\text-\cat{CAT}^\op \times \V\text-\cat{CAT}^\op \times \V\text-\cat{CAT} \longrightarrow \cat{CAT}
\end{equation}
sending $\A, \B, \C$ to the category of bifunctors and
sesquitransformations ${\A, \B \rightarrow \C}$. Clearly, this is a
locally full sub-$2$-functor of $\cat{SESQ}(\thg, \thg; \thg)$; and
more generally, we can show that the symmetric $2$-multicategory
$\cat{SESQ}$ of $\V$-categories and $n$-ary sesquifunctors has a
locally full sub-$2$-multicategory $\cat{BIFUN}$ of $\V$-categories
and $n$-ary bifunctors; here, for example, a \emph{trifunctor}
$T \colon \A, \B, \C \rightarrow \D$ is a ternary sesquifunctor such
that each of $T(a, \thg, \thg)$ and $T(\thg, b, \thg)$ and
$T(\thg, \thg, c)$ are bifunctors. When $\V$ is an $\ast$-braided
duoidal category, the sub-$2$-multicategory $\cat{BIFUN}$ of
$\cat{SESQ}$ is symmetric by Proposition~\ref{prop:3}(iii);
Example~\ref{ex:2} below shows that it is not so in general.

\subsection{Commuting tensor product of $\V$-categories}
\label{sec:tensor-product-v}
Just as with $\cat{Sesq}$, the $2$-multi\-category $\cat{Bifun}$ of
bifunctors between small $\V$-categories is often represented by a
monoidal structure on $\V\text-\cat{Cat}$, whose tensor product we
call the \emph{commuting tensor product}. Eilenberg and Kelly show
in~\cite[Chapter~III, \S4]{Eilenberg1966Closed} how to construct this
monoidal structure when $\V$ is a symmetric monoidal category:
the commuting tensor product $\A \comm \B$ of $\V$-categories $\A$ and $\B$ has
 object-set $\ob \A \times \ob \B$, hom-objects
$(\A \comm \B)((a,b), (a',b')) = \A(a,a') \otimes \B(b,b')$, and
composition maps defined using the composition in $\A$ and $\B$ and
the symmetry isomorphisms.

This construction still works in the braided monoidal
case~\cite[Remark~5.2]{Joyal1993Braided}, but when the normal duoidal
structure on $\V$ does not come from a braided monoidal one, the
construction of the commuting tensor product is completely different,
and will require stronger assumptions on $\V$. Recall that a
\emph{$\V$-graph} $A$ comprises a set $\ob A$ together with a family
$A(a,a')_{a,a' \in \ob A}$ of objects of $\V$, while a map
$f \colon A \rightarrow B$ of $\V$-graphs is given by a function
$f \colon \ob \A \rightarrow \ob \B$ together with a family of maps
$f_{aa'} \colon A(a,a') \rightarrow B(fa,fa')$ in $\V$. We write
$\V\text-\cat{Gph}$ for the category of \emph{small}
$\V$-graphs---those with small object set---and
$U \colon \V\text-\cat{Cat} \rightarrow \V\text-\cat{Gph}$ for the
obvious forgetful functor. We will say that \emph{free
  $\V$-categories exist} if this functor $U$ has a left adjoint $F$.

\begin{Prop}
  \label{prop:10}
  Let $\V$ be normal duoidal and let $\A, \B$ be small
  $\V$-categories. If $\V\text-\cat{Cat}$ has conical colimits,
  preserved by the inclusion into $\V\text-\cat{CAT}$, and free
  $\V$-categories exist, then
  $\cat{BIFUN}(\A, \B; \thg) \colon \V\text-\cat{CAT} \rightarrow
  \cat{CAT}$ has a small representation $\A \comm \B$.
\end{Prop}
\begin{proof}
  Consider the small $\V$-graph $U\A \conv U\B$ with object-set
  $\ob \A \times \ob \B$ and with homs
  $(U\A \conv U\B)((a,b),(a',b')) = \A(a,a') \conv \B(b,b')$. For any
  sesquifunctor $T \colon \A, \B \rightarrow \C$, we have a parallel
  pair of $\V$-graph morphisms $U\A \ast U\B \rightrightarrows U\C$
  which on objects both send $(a,b)$ to $Tab$, and on homs have their
  respective actions given by the upper and lower paths
  around~\eqref{eq:19}. Applying this to the universal sesquifunctor
  $V \colon \A, \B \rightarrow \A \fun \B$ whose existence is
  guaranteed by Proposition~\ref{prop:13}\eqref{item:1}, we obtain a
  pair of $\V$-graph morphisms
  $U\A \conv U\B \rightrightarrows U(\A \fun \B)$, corresponding under
  adjunction to a parallel pair
  $H, K \colon F(U\A \conv U\B) \rightrightarrows \A \fun \B$ of
  $\V$-functors. Let
  \begin{equation}
    \cd{
      F(U\A \conv U\B) \ar@<3pt>[r]^-{H} \ar@<-3pt>[r]_-{K} &
      \A \fun \B \ar@{->>}[r]^-{Q} & \A \comm \B
    }
  \end{equation}
  be their coequalizer in $\V\text-\cat{Cat}$; we claim that
  $QV \colon \A, \B \rightarrow \A \comm \B$ is the required
  universal bifunctor. This is to say that, for each
  $\C \in \V\text-\cat{CAT}$, the functor
  \begin{equation}\label{eq:20}
    \V\text-\cat{CAT}(\A \comm \B, \C) \xrightarrow{(\thg) \circ Q} \V\text-\cat{CAT}(A \fun \B, \C) \xrightarrow{(\thg) \circ V} \cat{SESQ}(\A, \B; \C)
  \end{equation}
  is injective on objects and fully faithful, and has as its image
  precisely the bifunctors $\A, \B \rightarrow \C$. Now $Q$ is the
  coequalizer of two functors $H, K$ which agree on objects; thus if
  $F, G \colon \A \comm \B \rightarrow \C$ and
  $\alpha \colon FQ \Rightarrow GQ$, then necessarily
  $\alpha H = \alpha K$, and so there is a unique
  $\bar \alpha \colon F \Rightarrow G$ with $\alpha = \bar \alpha Q$.
  So the first arrow in~\eqref{eq:20} is fully faithful; the second is
  too, being an isomorphism, and so~\eqref{eq:20} is itself fully
  faithful. On the other hand, a $\V$-functor
  $F \colon \A \fun \B \rightarrow \C$ factors through
  $Q$ just when $FH = FK$; transposing, this is equally to ask that
  the two composite morphisms
  $U\A \conv U\B \rightrightarrows U(\A \fun \B)
  \rightarrow U\C$
  are equal; but since $F$ is a functor, these two composites are
  precisely the two sides of~\eqref{eq:19} for $T = FV$. It follows
  that $F$ is in the image of the full embedding $(\thg) \circ Q$ just
  when $FV$ is a bifunctor, as required.
\end{proof}

\subsection{Functor $\V$-categories}
\label{sec:internal-hom-v}
We now turn our attention to the internal homs associated to the
commuting tensor product of $\V$-categories. It is easy to see that
the unit $\V$-category $\mathcal I$ is a unit for this tensor, from
which it follows that the underlying ordinary category of either
internal hom $[\B, \C]_\ell$ or $[\B, \C]_r$ must be the ordinary
category of $\V$-functors and $\V$-natural transformations
$\B \rightarrow \C$; which justifies our calling these internal homs
\emph{functor $\V$-categories}. The key to constructing these is a
notion of enriched \emph{end}, which generalises from the symmetric
monoidal case the definition of~\cite[\S 2.1]{Kelly1982Basic}. Before
giving this, let us recall some necessary background on profunctors.
\begin{Defn}
  \label{def:5}
  Let $(\V, \comp, \I)$ be a monoidal category.
  \begin{enumerate}[(i)]
  \item If $\A$ and $\B$ are
  $\V$-categories, then a \emph{$\V$-profunctor} $M \colon \A \tor \B$
  comprises a family of objects $M(b,a) \in \V$ together with actions
  $m \colon \A(a,a') \comp M(b,a) \rightarrow M(b,a')$ and
  $m \colon M(b,a) \comp \B(b',b) \rightarrow M(b',a)$ satisfying
  the usual associativity and unitality laws.\vskip0.25\baselineskip
\item Given profunctors $M \colon \A \tor \B$ and
  $N \colon \B \tor \C$ and $L \colon \A \tor \C$, a family of maps
  $f_{abc} \colon M(b,a) \comp N(c,b) \rightarrow L(c,a)$ is
  \emph{bilinear} if it is equivariant with respect to the left
  $\A$-action and right $\C$-action and satisfies the $\B$-bilinearity
  axiom
  $f(m\comp 1) = f(1 \comp m) \colon M(b',a) \comp \B(b,b')
  \comp N(c,b) \rightrightarrows L(c,a)$.\vskip0.25\baselineskip
\item Given $\V$-functors $F \colon \A \rightarrow \C$ and
  $G \colon \B \rightarrow \C$, we write $\C(F,G) \colon \B \tor \A$
  for the profunctor whose $(a,b)$-component is $\C(Fa,Gb)$ and whose
  actions are obtained using composition in $\C$ and the actions of
  $F$ and $G$ on homs. If we have a further $\V$-functor
  $H \colon \D \rightarrow \C$, then it is easy to see that the family
  of composition maps
  $\C(Gb, Hd) \comp \C(Fa, Gb) \rightarrow \C(Fa, Hd)$ is
  $\B$-bilinear.
\end{enumerate}
As our notation suggests, when $\V$ is a normal duoidal
category, we interpret these notions with respect to the
$\comp$-monoidal structure.
\end{Defn}
\begin{Defn}
  \label{def:14}
  Let $\V$ be a normal duoidal category 
  and $M \colon \A \tor \A$ a $\V$-profunctor. A family of maps
  $p_a \colon K \rightarrow M(a,a)$ in $\V$ is called
  \emph{left extranatural} or \emph{right extranatural} if each
  diagram to the left, respectively right, in:
\begin{equation}\label{eq:23}
    \cd[@!C@C-7.7em@R-0.5em]{
    & \sh{l}{1.7em}K \comp \A(a, b) \ar[rr]^-{p_{b} \comp 1} & &
    \sh{r}{2em}M(b,b) \comp \A(a, b) \ar[dr]^m \\ 
    \sh{r}{0.5em} K \conv \A(a, b) \ar[ur]^-{\sigma} \ar[dr]_-{\tau} & & & &
    M(a,b) \\
    &\sh{l}{1.7em}\A(a, b) \comp K \ar[rr]_-{1 \comp p_{a}} & &
    \sh{r}{2em}\A(a,b) \comp M(a,a) \ar[ur]_m
  }\ \ \ 
  \cd[@!C@C-7.7em@R-0.5em]{
  & \sh{l}{1.7em}\A(a, b) \comp K \ar[rr]^-{1 \comp q_a} & &
  \sh{r}{2em}\A(a, b) \comp M(a,a) \ar[dr]^m \\ 
  \sh{r}{0.5em}\A(a, b) \conv K \ar[ur]^-{\sigma} \ar[dr]_-{\tau} & & & &
  M(a,b)\\
  &\sh{l}{1.7em}K \comp \A(a, b) \ar[rr]_-{q_{b} \comp 1} & &
  \sh{r}{2em}M(b,b) \comp \A(a,b) \ar[ur]_m
}
\end{equation}
commutes in $\V$. A \emph{left end} for the profunctor $M$ is a universal left extranatural family
$u_a \colon U \rightarrow M(a,a)$; this means that each extranatural
family $p_a \colon K \rightarrow M(a,a)$ is of the form
$u_a \circ \bar p$ for a unique
$\bar p \colon K \rightarrow U$. A \emph{right end} for $M$ is defined correspondingly.
\end{Defn}
\begin{Prop}
  \label{prop:15}
  If $\V$ is complete and $\ast$-biclosed, and $\A$ is a small
  $\V$-category, then any profunctor
  $M \colon \A \tor \A$ has both a left and right end.
\end{Prop}
\begin{proof}
  The left end is given by an equalizer
  $s, t \colon \Pi_a M(a,a) \rightrightarrows \Pi_{a,b} [\A(a,b),
  M(a,b)]_\ell$,
  where the $(a,b)$-components of $s$ and $t$ are the transposes of
  the two sides of the left hexagon in~\eqref{eq:23}, interpreted for
  the family $\pi_a \colon \Pi_a M(a,a) \rightarrow M(a,a)$. The right
  end is given by a similar equalizer, but with the right hom
  $[\thg, \thg]_r$ replacing the left $[\thg, \thg]_\ell$, and with
  the right hexagon replacing the left in~\eqref{eq:23}.
\end{proof}
We will exploit the notions of left and right end to construct the
hom-objects of the functor $\V$-categories $[\B, \C]_\ell$ and
$[\B, \C]_r$ associated to the commuting tensor product. The following
lemma is the crucial step in describing the composition law.
\begin{Lemma}
  \label{lem:4}
  Let $\V$ be normal duoidal, let $\A$ be a $\V$-category, and let
  $M,N,P \colon \A \tor \A$ be profunctors. If
  $p_a \colon K \rightarrow M(a,a)$ and
  $q_a \colon L \rightarrow N(a,a)$ are left
  extranatural families, and
  $r_{abc} \colon M(b,a) \comp N(c,b) \rightarrow P(c,a)$ is a
  bilinear one, then the composite family
  \begin{equation}\label{eq:25}
    K \comp L \xrightarrow{p_a \comp q_a} M(a,a) \comp N(a,a) \xrightarrow{r_{aaa}} P(a,a)
  \end{equation}
  is left extranatural. The corresponding result holds for right extranaturality.
\end{Lemma}
\begin{proof}
  We consider the following diagram, wherein for typographical
  convenience we temporarily write $\A(a,b)$ as $\A_{ab}$, and so on:
  \begin{equation}
  \cd[@C+0.2em@R-0.3em]{
    (K \comp L) \conv \A_{ab} \ar[r]^{\delta^r_r} \ar[d]_{\delta^r_\ell} &
    K \comp (L \conv \A_{ab}) \ar[r]^-{1 \comp \sigma} \ar[d]^-{1 \comp \tau} &
    K \comp L \comp \A_{ab} \ar[r]^-{1 \comp q_b \comp 1} &
    K \comp N_{bb} \comp \A_{ab} \ar[d]^-{1 \comp m}\\
    (K \conv \A_{ab}) \comp L \ar[r]^-{\sigma \comp 1} \ar[d]_-{\tau \comp 1} &
    K \comp \A_{ab} \comp L \ar[r]^-{1 \comp 1 \comp q_a} \ar[d]_-{p_b \comp 1 \comp 1} &
    K \comp \A_{ab} \comp N_{aa} \ar[r]^-{1 \comp m} \ar[d]^-{p_b \comp 1 \comp 1} &
    K \comp N_{ab} \ar[d]^-{p_b \comp 1} \\
    \A_{ab} \comp K \comp L \ar[d]_-{1 \comp p_a \comp 1} &
    M_{bb} \comp \A_{ab} \comp L \ar[r]^-{1 \comp 1 \comp q_a} \ar[d]_-{m \comp 1} &
    M_{bb} \comp \A_{ab} \comp N_{aa} \ar[r]^-{1 \comp m} \ar[d]^-{m \comp 1} &
    M_{bb} \comp N_{ab} \ar[d]^-{r_{abb}} \\
    \A_{ab} \comp M_{aa} \comp L \ar[r]^-{m \comp 1} &
    M_{ab} \comp L \ar[r]^-{1 \comp q_a} &
    M_{ab} \comp N_{aa} \ar[r]^-{r_{aab}} &
    P_{ab} \rlap{ .}
  }
\end{equation}
The rectangular regions commute by left extranaturality of $p$ and $q$, and
the bottom right square by bilinearity of $r$. The top left
square, in which we make use of the linear distributivities of
Section~\ref{sec:normality}, commutes as an easy consequence of the
duoidal category axioms; while the remaining three squares commute by
(ordinary) bifunctoriality of $\comp$. We conclude that the outside
rectangle commutes, and the duoidal axioms and bifunctoriality of
$\comp$ now show that this rectangle is the extranaturality
hexagon~\eqref{eq:23} for~\eqref{eq:25}. The case of right
extranaturality is dual.
\end{proof}
\begin{Prop}
  \label{prop:16}
  Let $\V$ be a complete $\ast$-biclosed normal duoidal category. If
  $\B$ and $\C$ are $\V$-categories with $\B$ small, then the
  $2$-functors
  $\cat{BIFUN}(\thg, \B; \C) \colon \V\text-\cat{CAT}^\op \rightarrow
  \cat{CAT}$
  and
  $\cat{BIFUN}(\B, \thg; \C) \colon \V\text-\cat{CAT}^\op \rightarrow
  \cat{CAT}$
  admit representations $[\B, \C]_\ell$ and $[\B, \C]_r$, which are
  sub-$\V$-categories of $\dbr{\B,\C}$, and are small whenever $\C$ is
  so.
\end{Prop}
\begin{proof}
  The objects of $[\B, \C]_\ell$ are $\V$-functors
  $\B \rightarrow \C$, and the hom-object $[\B, \C]_\ell(F,G)$ is the
  left end of $\C(F,G)$, whose existence is guaranteed by
  Proposition~\ref{prop:15}. Clearly $[\B, \C]_\ell$ will be a small
  $\V$-category whenever $\C$ is so. We write the universal
  extranatural families associated to the homs as
  $p_{F,G,a} \colon [\B, \C]_\ell(F,G) \rightarrow \C(Fa, Ga)$. 

  Now, by universality, maps $\I \rightarrow [\B, \C]_\ell(F,G)$
  correspond to left extranatural families $\I \rightarrow \C(Fa,Ga)$,
  and these are easily seen to correspond to $\V$-natural
  transformations in the usual sense. In particular, for each
  $F \in [\B, \C]_\ell$, the identity $\V$-natural transformation on
  $F$ yields a unique map $j_F$ rendering commutative each square as
  to the left in:
  \begin{equation}
    \cd{
      {\I} \ar[r]^-{j_F} \ar[d]_{1} &
      {[\B, \C]_\ell(F,F)} \ar[d]^{p_{F,F,a}} \\
      {\I} \ar[r]_-{j_{Fa}} &
      {\C(Fa,Fa)}
    } \qquad 
    \cd{
      {[\B,\C]_\ell(G,H) \comp [\B,\C]_\ell(F,G)} \ar[r]^-{m_{FGH}} \ar[d]_{p_{G,H,a} \comp p_{F,G,a}} &
      {[\B, \C]_\ell(F,H)} \ar[d]^{p_{F,H,a}} \\
      {\C(Ga,Ha) \comp \C(Fa,Ga)} \ar[r]_-{m_{Fa,Ga,Ha}} &
      {\C(Fa,Ha)\rlap{ .}}
    }
  \end{equation}

  On the other hand, given $F,G,H \colon \B \rightarrow \C$ and
  $a \in A$, we may form the lower composite around the square right
  above, and by Lemma~\ref{lem:4}, these composites constitute a left
  extranatural family; so by universality we induce a unique $m_{FGH}$
  as displayed making the square commute for all $a \in \B$. This
  defines the unit and composition maps of $[\B,\C]_\ell$; the
  associativity and unitality axioms follow immediately from those of
  $\C$ and the universal property of end. Note also that, by the
  definition of left end, the universal families $p_{F,G,b}$ assemble
  to give monomorphisms
  $[\B,\C]_\ell(F,G)\rightarrowtail{}\prod_b\C(Fb,Gb)=\dbr{\B,\C}(F,G)$,
  and commutativity of the above diagrams immediately implies that
  these are part of an identity-on-objects $\V$-functor
  $[\B,\C]_\ell\to\dbr{\B,\C}$.

  For a $\V$-bifunctor $T \colon \A, \B \rightarrow
  \C$, the corresponding $\bar T \colon \A \rightarrow [\B,
  \C]_\ell$ is given on objects by $\bar T(a) = T(a,
  \thg)$; on homs, the map $\bar T_{aa'} \colon \A(a,a') \rightarrow [\B,
  \C]_\ell(T(a,\thg), T(a',\thg))$ corresponds to the family
  $T(\thg,b)_{aa'} \colon \A(a,a') \rightarrow \C(Tab, Tab')$,
  whose left extranaturality in $b$ is expressed precisely by the
  bifunctoriality hexagon~\eqref{eq:19}. The functoriality of $\bar T$
  follows from that of each $T(\thg,b)$ and the universal
  property of end, and the assignation $T \mapsto \bar T$ is easily seen
  to be bijective, and $2$-natural in $\A$. 

  Suppose now we are given a sesquitransformation
  $\alpha \colon T \Rightarrow S \colon \A, \B \rightarrow \C$. The
  corresponding
  $\bar \alpha \colon \bar T \Rightarrow \bar S \colon \A \rightarrow
  [\B, \C]_\ell$
  has components
  $\bar \alpha_a \colon I \rightarrow [\B, \C]_\ell(T(a,\thg),
  S(a,\thg))$
  corresponding to the family
  $\alpha_{ab} \colon I \rightarrow \C(Tab, Sab)$ whose left
  extranaturality in $b$ is precisely its $\V$-naturality in $b$. The
  $\V$-naturality of $\bar \alpha$ itself is correspondingly the
  $\V$-naturality of the components $\alpha_{ab}$ in $a$. Once again,
  the assignation $\alpha \mapsto \bar \alpha$ is easily seen to be
  bijective, and $2$-natural in $\A$. This shows that $[\B, \C]_\ell$
  represents $\cat{BIFUN}(\thg, \B; \C)$ as required; the case of
  $[\B, \C]_r$ is dual.
\end{proof}

\subsection{The monoidal $2$-category $\V\text-\cat{Cat}$}
\label{sec:2-monoidal-structure}
By assembling the constructions of the preceding two sections, we are
now able to give sufficient conditions for the commuting tensor
product of $\V$-categories to make $\V\text-\cat{Cat}$ into a monoidal
closed $2$-category.

\begin{Prop}
  \label{prop:17}
  Let $\V$ be a complete $\conv$-biclosed normal duoidal category. If
  for all $\C, \D \in \V\text-\cat{Cat}$, the commuting tensor product
  $\C \comm \D$ exists, then it forms part of a biclosed monoidal
  $2$-category structure on $\V\text-\cat{Cat}$ with unit the $\V$-category
  $\mathcal \I$. This monoidal structure is symmetric whenever
  $\V$ is $\ast$-braided.
\end{Prop}
\begin{proof}
  The universal property of the commuting tensor product gives a $2$-functor
  $\comm \colon \V\text-\cat{Cat} \times \V\text-\cat{Cat} \to
  \V\text-\cat{Cat}$
  and $2$-natural transformation
  $q \colon \mathord \square \Rightarrow \mathord \comm$ from the
  tensor product of Proposition~\ref{prop:13} which is pointwise
  epimorphic and co-fully faithful. It is easy to see that any
  sesquifunctor $\mathcal I, \A \rightarrow \B$ or
  $\A, \mathcal I \rightarrow \B$ is a bifunctor, so that the unit
  coherence maps for $\fun$ descend (uniquely) to $\comm$, as to the
  left in:
  \begin{equation}
    \cd{
      {\mathcal \I \fun \A} \ar[r]^-{\lambda_{\fun}} \ar@{->>}[d]_{q} &
      {\A} \ar@{->>}[d]^{\id} \\
      {\mathcal \I \comm \A} \ar@{-->}[r]^-{\lambda_\comm} &
      {\A}
    }
    \qquad 
    \cd{
      {\A \fun \mathcal \I} \ar[r]^-{\rho_{\fun}} \ar@{->>}[d]_{q} &
      {\A} \ar@{->>}[d]^{\id} \\
      {\A \comm \mathcal \I} \ar@{-->}[r]^-{\rho_\comm} &
      {\A}
    } 
    \qquad 
    \cd{
      {(\A \fun \B) \fun \C} \ar[r]^-{\alpha_{\fun}} \ar@{->>}[d]_{q(q \fun \id)} &
      {\A \fun (\B \fun \C)} \ar@{->>}[d]^{q(\id \fun q)} \\
      {(\A \comm \B) \comm \C} \ar@{-->}[r]^-{\alpha_\comm} &
      {\A \comm (\B \comm \C)}\rlap{ .}
    }
  \end{equation}
  Since $\fun$ preserves colimits in each variable, the vertical maps
  in the square right above are also epimorphic and co-fully faithful;
  so to obtain the associativity $2$-natural transformation it
  suffices to show that the $\alpha_{\fun}$'s also descend as
  indicated. The key point is to show that for any $\V$-category $\D$,
  composition with the universal
  $\varepsilon \colon [\C, \D]_\ell, \C \rightarrow \D$ induces
  $2$-natural bijections
  \begin{equation}\label{eq:28}
    \cat{BIFUN}(\A, \B; [\C, \D]_\ell) \rightarrow \cat{BIFUN}(\A, \B, \C; \D)\rlap{ .}
  \end{equation}
  Given this, the $2$-natural bijections
  $\cat{BIFUN}(\A, \B; [\C, \D]_\ell) \cong \V\text-\cat{CAT}((\A
  \comm \B) \comm \C, \D)$ then imply that $(A \comm \B) \comm \C$
  represents trifunctors; a dual argument shows that $\A \comm (\B
  \comm \C)$ does so too, so allowing us to conclude that $\alpha_{\fun}$
  descends as required.


  To show that~\eqref{eq:28} is invertible, consider a trifunctor
  $T \colon \A, \B, \C \rightarrow \D$. For each object $a \in \A$, the
  bifunctor $T(a, \thg, \thg) \colon \B, \C \rightarrow \D$ is
  classified by a $\V$-functor
  $T'(a, \thg) \colon \B \rightarrow [\C, \D]_\ell$, while for each
  $b \in \B$, the bifunctor
  $T(\thg, b, \thg) \colon \A, \C \rightarrow \D$ is classified by a
  $\V$-functor $T'(\thg, b) \colon \A \rightarrow [\C, \D]_\ell$;
  these now assemble to give
  a
  sesquifunctor $T' \colon \A, \B \rightarrow [\C, \D]_\ell$, which we
  claim is commutative. The evident evaluation $\V$-functors
  $\mathrm{ev}_c \colon [\C, \D]_\ell \rightarrow \D$ are jointly
  faithful, and so by Proposition~\ref{prop:3}(ii), it suffices to
  show that each $\mathrm{ev}_c \circ T' \colon \A, \B \rightarrow \C$
  is a bifunctor; which is so since
  $\mathrm{ev}_c \circ T' = T(\thg, \thg, c)$.

  Given this, we conclude by the above argument that both
  $(\A \comm \B) \comm \C$ and $\A \comm (\B \comm \C)$ classify
  trifunctors, whence the associativity constraints for $\fun$ descend
  to $\comm$. The pentagon and triangle axioms for $\comm$ now follow
  easily from those for $\fun$, and so
  $(\V\text-\cat{Cat}, \comm, \mathcal I)$ becomes a biclosed monoidal
  $2$-category with internal homs $[\A, \B]_\ell$ and $[\A, \B]_r$.
  Finally, when $\V$ is $\ast$-braided, we see by
  Proposition~\ref{prop:3}(iii) that $\A \comm \B$ represents
  bifunctors $\B, \A \rightarrow \C$ as well as ones
  $\A, \B \rightarrow \C$; whence the symmetry isomorphisms of $\fun$
  descend to $\comm$, which is thus symmetric monoidal.
\end{proof}

\subsection{Commuting graph morphisms}
\label{sec:comm-graph-morph}
A critical examination of Definition~\ref{def:17} reveals that the
definition of $\V$-bifunctor $\A, \B \rightarrow \C$ makes no use of
the compositional structure of $\A$ or $\B$. Guided by this, we may
define a more general notion of bimorphism whose codomain is still a
$\V$-category but whose domain is given by a pair of mere $\V$-graphs:

\begin{Defn}
  \label{def:18}
  Let $\V$ be a normal duoidal category, let $A$ and $B$ be
  $\V$-graphs and let $\C$ be a $\V$-category. A \emph{graph
    sesquimorphism} $T \colon A, B \rightarrow U\C$ comprises families
  $(T(a, \thg) \colon B \rightarrow U\C)_{a \in A}$ and
  $(T(\thg, b) \colon A \rightarrow U\C)_{b \in B}$ of graph
  morphisms such that $T(a, \thg)(b) = T(\thg, b)(a) = Tab$, say. $T$
  is said to \emph{commute}, or to be a \emph{graph bimorphism}, if
  each instance of~\eqref{eq:19}---with $A$ and $B$
  replacing $\A$ and $\B$---is commutative in $\V$.
\end{Defn}

In fact, when $\V$ is $\ast$-biclosed and complete and free
$\V$-categories exist, the extra generality this provides is only
apparent: graph bimorphisms $A, B \rightarrow U\C$ turn out to
coincide with bifunctors $FA, FB \rightarrow \C$ from the
corresponding free $\V$-categories. The key to proving this is the
construction of a more general kind of functor $\V$-category, whose
domain is a small $\V$-\emph{graph} rather than a $\V$-category.

\begin{Prop}
  \label{prop:6}
  Let $\V$ be a complete $\ast$-biclosed normal duoidal category, let
  $B$ be a small $\V$-graph and let $\C$ be a $\V$-category. There is
  a $\V$-category $[B, \C]_\ell$ and graph bimorphism
  $\varepsilon \colon U[B, \C]_\ell, B \rightarrow U\C$,
  $\V$-functorial in its first variable, that provides a
  representation for the functor
  $\cat{BIMOR}(\thg, B; U\C) \colon \V\text-\cat{GPH}^\op \rightarrow
  \cat{SET}$.
\end{Prop}
Of course, there is a dual construction of $[B, \C]_r$, with
correspondingly dual properties, which we do not trouble to state.
\begin{proof}
First, a \emph{profunctor} between $\V$-graphs $M \colon A \tor B$
comprises components $M(b,a)$ with left $A$-actions and right
$B$-actions, but satisfying no associativity or unit axioms. Next, if
$M \colon A \tor A$ is a profunctor between graphs, then a \emph{left}
or \emph{right extranatural} family $K \rightarrow M(a,a)$ is defined
just as in Definition~\ref{def:14}; the corresponding notions of
\emph{left} and \emph{right end}, together with their construction in
Proposition~\ref{prop:15}, carry over directly. We may thus define
$[B, \C]_\ell$ to be the $\V$-category with:
\begin{itemize}
\item \emph{Objects} being graph morphisms $F \colon B \rightarrow U\C$;
\item \emph{Hom-object} $[B, \C]_\ell(F,G)$ being the left end of
  $\C(F,G) \colon B \tor B$;
\item \emph{Composition} derived as in Proposition~\ref{prop:16},
  using the graph analogue of Lemma~\ref{lem:4}.
\end{itemize}
The universal property of $[B, \C]_\ell$ follows exactly as in
the proof of Proposition~\ref{prop:16}. 
\end{proof}

\begin{Prop}
  \label{prop:18}
  Let $\V$ be complete $\ast$-biclosed normal duoidal, let
  $A$ and $B$ be $\V$-graphs, and let
  $\eta_A \colon A \rightarrow UFA$ and
  $\eta_B \colon B \rightarrow UFB$ exhibit $FA$ and $FB$ as the free
  $\V$-categories on $A$ and $B$. For any $\V$-category $\C$,
  composing with $(\eta_A, \eta_B)$ establishes a bijection
  between $\V$-bifunctors $FA, FB \rightarrow \C$ and graph
  bimorphisms $A,B \rightarrow U\C$.
\end{Prop}
\begin{proof}
  If $T \colon FA, FB \rightarrow \C$ is a $\V$-bifunctor, then
  $UT \colon UFA, UFB \rightarrow U\C$ is a graph bimorphism, whence
  by (the graph analogue of) Proposition~\ref{prop:3}(i), so too is
  the composite
  $UT \circ (\eta_A, \eta_B) \colon A, B \rightarrow U\C$. Suppose
  conversely that $S \colon A, B \rightarrow U\C$ is a graph
  bimorphism. This corresponds by Proposition~\ref{prop:6} to a graph
  morphism $A \rightarrow U[B, \C]_\ell$, and so to a $\V$-functor
  $S' \colon FA \rightarrow [B, \C]_\ell$. The composite
  $\varepsilon \cdot (US', 1) \colon UFA, B \rightarrow U\C$ is now a
  graph bimorphism, $\V$-functorial in its first variable, whose
  precomposition with $(\eta_A, 1)$ is $S$. Repeating the same
  argument using $[UFA, \C]_r$ in place of $[B, \C]_\ell$ yields a graph
  bimorphism $UFA, UFB \rightarrow U\C$, $\V$-functorial in each
  variable---thus, a $\V$-bifunctor $FA, FB \rightarrow \C$---whose
  precomposition with $(\eta_A, \eta_B)$ is $S$, as required. 
\end{proof}

\subsection{Change of base}
\label{sec:change-base}
In the final part of this section, we briefly explore the interaction
of commuting tensor products with \emph{change of base}. Recall that,
if $F \colon \V \rightarrow \W$ is a (lax) monoidal functor between
monoidal categories, then there is an induced $2$-functor
$F_\ast \colon \V\text-\cat{CAT} \rightarrow \W\text-\cat{CAT}$, which
sends a $\V$-category $\A$ to the $\W$-category $F_\ast \A$ with the
same set of objects, and with homs $(F_\ast\A)(X,Y) = F(\A(X,Y))$.

If $\V$ and $\W$ are braided monoidal and $F$ is a braided monoidal
functor, then $F_\ast$ easily becomes a monoidal $2$-functor with
respect to the tensor products on $\V\text-\cat{CAT}$ and
$\W\text-\cat{CAT}$, one which is \emph{strong} monoidal whenever $F$
is so. We wish to extend this fact from the braided monoidal to the normal
duoidal context. To this end, given two duoidal categories
$(\V, \conv, \comp)$ and $(\W, \conv, \comp)$, we define a
\emph{duoidal functor} $F\colon\V\to\W$ to be a functor $F$ which is
monoidal with respect to both $\conv$ and $\comp$ in such a way
that the constraint maps $FX\comp FY\to F(X\comp Y)$ and $I\to F(I)$
for the $\comp$-monoidal structure are monoidal natural
transformations with respect to $\conv$. We call $F$
\emph{strong} if both underlying monoidal functors are so.

\begin{Ex}
  \label{ex:3}
  If $\V$ is any duoidal category, then
  $\V(\J, \thg) \colon \V \rightarrow \cat{Set}$ is a duoidal functor
  when $\cat{Set}$ is seen as cartesian duoidal as in
  Example~\ref{ex:1}.
\end{Ex}

\begin{Prop}
  \label{prop:14}
  Let $F \colon \V \rightarrow \W$ be a duoidal functor between normal
  duoidal categories. The induced
  $F_\ast \colon \V\text-\cat{CAT} \rightarrow \W\text-\cat{CAT}$ is
  lax monoidal with respect to the commuting tensor products on
  $\V\text-\cat{CAT}$ and $\W\text-\cat{CAT}$, insofar as these are
  defined; if moreover $\V$, $\W$ and $F$ are $\ast$-braided, then $F_\ast$
  is symmetric monoidal.
\end{Prop}
\begin{proof}
  Since sesquifunctors are defined by families of partial functors, it
  is clear that the action of $F_\ast$ on morphisms can be extended to
  an action
\begin{equation}
  F_\ast \colon \cat{SESQ}_\V(\A,\B;\C)\longrightarrow  \cat{SESQ}_\W(F_\ast\A,F_\ast\B;F_\ast\C)
  \label{eq:58}
\end{equation}
which is 2-natural in $\A$, $\B$ and $\C$. Given a sesqui-$\V$-functor
$T \colon \A, \B \rightarrow \C$, applying $F$ to the bifunctor
diagram~\eqref{eq:19} for $T$ and precomposing with the comparison map
$F\A(a,a') \conv F\B(b,b') \rightarrow F(\A(a,a') \conv \B(b,b'))$
yields the corresponding bifunctor diagram for
$F_\ast(T)$; in particular, commutativity of the former implies
commutativity of the latter, and so  the functors~\eqref{eq:58} restrict to ones
\begin{equation}
  \label{eq:59}
  F_\ast\colon\cat{BIFUN}_\V(\A,\B;\C)\longrightarrow
  \cat{BIFUN}_\W(F_\ast\A,F_\ast\B;F_\ast\C)\rlap{ .}
\end{equation}
More generally, we may show that $F_\ast$ extends to a morphism of
$2$-multicategories $\cat{BIFUN}_\V \rightarrow \cat{BIFUN}_\W$, and
as a direct consequence of this, $F_\ast$ is lax monoidal with respect
to the representing monoidal structures $(\comm, \mathcal I)$ on
$\V\text-\cat{CAT}$ and $\W\text-\cat{CAT}$, insofar as these are
defined. Finally, if $\V$, $\W$ and $F$ are $\ast$-braided, then
$F_\ast$ is a morphism of \emph{symmetric} $2$-multicategories, and so
is \emph{symmetric} monoidal wherever the tensor product is defined.
\end{proof}
In order for change of base along a duoidal functor
$F \colon \V \rightarrow \W$ to be \emph{strong} monoidal with respect
to $\comm$, we will have to assume more than merely that $F$ is strong
duoidal. The reason is that the construction of the commuting tensor
product of $\V$-categories involves free $\V$-categories and colimits
of $\V$-categories, and there is no reason why $F$ being strong
duoidal should force change of base to preserve these. The simplest
way of ensuring such preservation is to ask for $F$ to be a left
adjoint.

More precisely, by a \emph{duoidal adjunction} between duoidal
categories, we mean an adjunction
$F \dashv G \colon \W \rightarrow \V$ for which $F$ and $G$ are
duoidal functors and the unit and counit $1 \Rightarrow GF$ and
$FG \Rightarrow 1$ are monoidal natural transformations with respect
to both $\comp$ and $\conv$. By the considerations
of~\cite{Kelly1974Doctrinal}, the $F$ in this situation must be
\emph{strong} duoidal, and the duoidal constraint cells for $G$
determined as the mates under adjunction of those for $F$.
\begin{Ex}
  \label{ex:4}
  If $\V$ is a normal duoidal category admitting copowers of the unit
  $\J$, then the functor
  $\V(\J, \thg) \colon \V \rightarrow \cat{Set}$ has a left adjoint
  $(\thg) \cdot \J$. If both monoidal structures $\comp$ and $\conv$
  on $\V$ are closed on at least one side, then this left adjoint is
  easily seen to be strong duoidal, and so part of a duoidal
  adjunction
  $(\thg) \cdot J \dashv \V(\J, \thg) \colon \V \rightarrow
  \cat{Set}$.
\end{Ex}
\begin{Prop}
  \label{prop:4}
  If $F\dashv G\colon\W\to\V$ is a duoidal adjunction, then $F_\ast$
  preserves any commuting tensor product of $\V$-categories that 
  exists; hence, under the hypotheses of
  Proposition~\ref{prop:17}, $F_\ast$ becomes a  strong monoidal 2-functor
  $\V\text-\cat{Cat} \rightarrow \W\text-\cat{Cat}$.
\end{Prop}
\begin{proof}
  The adjunction $F \dashv G$ induces a $2$-adjunction
  $F_\ast \dashv G_\ast \colon \W\text-\cat{Cat} \rightarrow \V\text-\cat{Cat}$, and so $2$-natural isomorphisms
  \begin{equation}
    \label{eq:46}
    \cat{SESQ}_\W(F_\ast\A,F_\ast\B;\C)\cong\cat{SESQ}_\V(\A,\B;G_\ast\C)\rlap{ .}
  \end{equation}
  We have, moreover, a bijection between maps
  $FA \conv FB \rightarrow C$ in $\W$ and ones
  $A \conv B \rightarrow GC$ in $\V$; under this bijection, the
  diagram~\eqref{eq:19} for a sesquifunctor
  $T \colon F_\ast \A, F_\ast \B \rightarrow \C$ transposes to the
  corresponding diagram for
  $\tilde T \colon \A, \B \rightarrow G_\ast \C$, and so the 
  isomorphisms~\eqref{eq:46} restrict to ones
  \begin{equation}
    \label{eq:60}
    \cat{BIFUN}_\W(F_\ast\A,F_\ast\B;\C)\cong\cat{BIFUN}_\V(\A,\B;G_\ast\C)\rlap{ .}
  \end{equation}
  Thus if $\A, \B$ in $\W\text-\cat{CAT}$ admit the tensor product
  $\A \comm \B$, then
  $\W\text-\cat{CAT}(F_\ast(\A\comm\B),\thg)$ is isomorphic to
  $\cat{BIFUN}(F_\ast\A,F_\ast\B;\thg)$, and so $F_\ast(\A \comm \B)$
  is a tensor product of $F_\ast(\A)$ and $F_\ast(\B)$. The second
  assertion can be easily deduced from the first and the construction
  of the associativity and unit contraints for $\comm$ as in
  Section~\ref{sec:2-monoidal-structure}.
\end{proof}
Note that, in the situation of this proposition, the adjunction
$F_\ast \dashv G_\ast \colon \W\text-\cat{Cat} \rightarrow
\V\text-\cat{Cat}$
becomes an adjunction of monoidal $2$-categories. In fact, an
alternative way of proving the last clause of the preceding
proposition would have been to run this argument backwards: by
exploiting the $2$-functoriality of the assignation
$\V \mapsto \V\text-\cat{Cat}$ from normal duoidal categories to
monoidal $2$-categories, we could deduce that any duoidal adjunction
is sent to an adjunction of monoidal $2$-categories, which
by~\cite{Kelly1974Doctrinal} must have a strong left adjoint.

\section{Commutativity: the one-object case}
\label{sec:one-object-case}
In this section, we indicate how our theory of commutativity
specialises to the case of one-object $\V$-categories---that is, of
$\comp$-monoids in $\V$.

\subsection{The basic notions}
\label{sec:basic-notions}
Given a $\comp$-monoid $A$ in the normal duoidal $\V$, we will write
$\Sigma A$ for the corresponding one-object $\V$-category. With this
convention, we see that a sesquifunctor
$T \colon \Sigma A, \Sigma B \rightarrow \Sigma C$ is given simply by
a cospan $f \colon A \rightarrow C \leftarrow B \colon g$ of
$\comp$-monoid morphisms. As anticipated in the introduction, such a
cospan will be said to \emph{commute}, in the sense of corresponding to a
bifunctor, just when the hexagon
\begin{equation}\label{eq:22}
    \cd[@!C@C-0.5em@R-0.7em]{
    & \sh{l}{0.5em} A \comp B \ar[r]^-{f \comp g} &
    \sh{r}{0.5em} C \comp C \ar[dr]^m \\ 
    A \conv B \ar[ur]^-{\sigma} \ar[dr]_-{\tau} & & & 
    C \\
    &\sh{l}{0.5em} B \comp A \ar[r]_-{g \comp f} &
    \sh{r}{0.5em}C \comp C \ar[ur]_m
  }
\end{equation}
commutes in $\V$. Correspondingly, an $n$-ary sesquifunctor
$\Sigma A_1, \dots, \Sigma A_n \rightarrow \Sigma B$ amounts to an
$n$-tuple of $\comp$-monoid morphisms $f_i \colon A_i \rightarrow B$,
and such an $n$-tuple commutes just when the cospan $(f_i, f_j)$ does
so for each $1 \leqslant i < j \leqslant n$.

More generally, the one-object case of a graph sesquimorphism
$A, B \rightarrow U \C$ in the sense of Definition~\ref{def:18} is
that of a cospan $f \colon A \rightarrow UC \leftarrow B \colon g$,
where $C$ is a $\comp$-monoid and $A, B$ are mere objects in $\V$. The
commutativity of such a cospan is the requirement that the
same~\eqref{eq:22} should commute in $\V$. In this context,
Proposition~\ref{prop:18} expresses that a cospan of this kind is
commutative precisely when the transpose cospan
$FA \rightarrow C \leftarrow FB$ of $\comp$-monoids is so.

Note that, if $T \colon \Sigma A, \Sigma B \rightarrow \Sigma C$
corresponds to a cospan
$f \colon A \rightarrow C \leftarrow B \colon g$, then
$T^c \colon \Sigma B, \Sigma A \rightarrow \Sigma C$ corresponds to
the reverse cospan $g \colon B \rightarrow C \leftarrow A \colon f$;
using this fact, we can fulfil a promise made in
Section~\ref{sec:bifunctors} above by exhibiting a $\V$ for which the
notion of bifunctor is not symmetric in its input arguments.

\begin{Ex}
  \label{ex:2}
  Let $\V$ be the category of positively graded $k$-vector spaces,
  equipped with the tensor product
\begin{equation}
  (X \otimes Y)_n = \sum_{n = r + s} X_r \otimes Y_s \qquad \text{and} \qquad 
  I_n =
  \begin{cases}
    k & \text{if $n = 0$;} \\ 0 & \text{otherwise.}
  \end{cases}
\end{equation}
If we fix $q \neq 0$ in $k$, there is as
in~\cite[Section~2.3]{Aguiar2010Monoidal} a braiding
$c_q \colon X \otimes Y \rightarrow Y \otimes X$ given on homogeneous
elements by $c_q(x \otimes y) = q^{rs}(y \otimes x)$ for $x \in X_r$
and $y \in Y_s$; using this, we may see $\V$ as a normal duoidal
category in the canonical way.

In this situation, a $\comp$-monoid is a graded $k$-algebra. If $C$ is
any graded $k$-algebra, and $A$ and $B$ are the graded $k$-algebras
freely generated by elements $a \in A_r$ and $b \in B_s$, then a
cospan of $\comp$-monoid maps
$f \colon A \rightarrow C \leftarrow B \colon g$, corresponding to a
sesquifunctor $T$, will commute just when $f(a)g(b) = q^{rs}g(b)f(a)$.
On the other hand, the cospan
$g \colon B \rightarrow C \leftarrow A \colon f$, corresponding to the
transpose sequifunctor $T^c$, will commute just when
$g(b)f(a) = q^{rs}f(a)g(b)$. In particular, if $q \neq \pm 1$, then
these two commutativities are, in general, distinct.
\end{Ex}

\subsection{Commuting tensor product and centralizers}
\label{sec:comm-tens-prod}
Whenever $\V$ has a terminal object, the set-of-objects functor
$\V\text-\cat{Cat} \rightarrow \cat{Set}$ has a right adjoint, and so
preserves colimits and sends $\fun$ to $\times$; it follows that
one-object $\V$-categories are closed under the commuting tensor
product as constructed in Section~\ref{sec:tensor-product-v}.

In fact, constructing the commuting tensor product of $\comp$-monoids
$A$ and $B$ does not require the full strength of the hypotheses of
Proposition~\ref{prop:10}: it suffices that free $\comp$-monoids
should exist and that $\cat{Mon}_\comp(\V)$ should admit finite
colimits. Indeed, one first forms the coproduct
$\iota_1 \colon A \rightarrow A+B \leftarrow B \colon \iota_2$ in
$\cat{Mon}_\comp(\V)$; then the parallel pair
$A \conv B \rightrightarrows U(A+B)$ in $\V$ given by the two sides of
the hexagon~\eqref{eq:22} for the cospan $(\iota_1, \iota_2)$; and
finally obtains the commuting tensor product of $A$ and $B$ as the
coequalizer in $\cat{Mon}_\comp(\V)$ of the transposed maps
$F(A \conv B) \rightrightarrows A+B$.

By contrast to the above, one-object $\V$-categories are not closed
under taking functor $\V$-categories; indeed, objects of either
internal hom $[\Sigma A, \Sigma B]_\ell$ or $[\Sigma A, \Sigma B]_r$
are arbitrary $\comp$-monoid morphisms $A \rightarrow B$. Nonetheless,
each endo-hom-object of $[\Sigma A, \Sigma B]_\ell$ or
$[\Sigma A, \Sigma B]_r$ is a $\comp$-monoid in $\V$ and in fact a
sub-$\comp$-monoid of $B$---and may be seen as providing a general
notion of \emph{centralizer}:

\begin{Defn}
  \label{def:19}
  Let $\V$ be a complete $\ast$-biclosed normal duoidal category. If
  $f \colon A \rightarrow B$ is a $\comp$-monoid morphism in $\V$,
  then the \emph{left centralizer} of $f$ is the sub-$\comp$-monoid
  $C_\ell(f) \defeq [\Sigma A, \Sigma B]_\ell(f,f)$ of $B$; the
  \emph{right centralizer} $C_r(f)$ is defined dually. The \emph{left}
  or \emph{right centre} of a $\comp$-monoid $A$ is the left or
  right centralizer of $1_A$.
\end{Defn}
Note that, when $\V$ is $\conv$-braided, the two notions of
centralizer and centre coincide by Proposition~\ref{prop:3}(iii), and
in this case, we drop the modifiers ``left'' and ``right''. Unwinding
the proof of Proposition~\ref{prop:16}, we find that the left
centralizer of $f \colon A \rightarrow B$ may be constructed as
follows. We take $g = 1_B$ in~\eqref{eq:22}, transpose both paths
under the adjunction $A \conv (\thg) \dashv [A, \thg]_\ell$ to obtain
a parallel pair $B \rightrightarrows [A, B]_\ell$, and take the
equalizer of this pair to obtain $C_\ell(f)$. The construction of
$C_r(f)$ is dual.

When $\V = \cat{Set}$, the centralizer of a monoid morphism
$f \colon N \rightarrow M$ is, as expected, the set
$\{m \in M : mn = nm \text{ for all $n \in N$}\}$. For a general $\V$,
our nomenclature is justified by the following result, which is an
immediate consequence of the universal characterisation of functor
$\V$-categories.
\begin{Prop}
  \label{prop:21}
  Let $\V$ be a complete $\conv$-biclosed normal duoidal category and
  let $f \colon A \rightarrow C \leftarrow B \colon g$ be a cospan of
  $\comp$-monoid morphisms. The following are equivalent:
  \begin{enumerate}[(i)]
  \item $(f,g)$ is a commuting cospan;
  \item $f$ factors through the left centralizer $C_\ell(g)
    \rightarrowtail C$;
  \item $g$ factors through the right centralizer $C_r(f)
    \rightarrowtail C$.
  \end{enumerate}
\end{Prop}
By adapting the proof of Proposition~\ref{prop:17}, we see that, if
$\V$ is complete $\conv$-biclosed normal duoidal, and all commuting
tensor products of $\comp$-monoids exist, then we obtain a monoidal
structure $\comm$ on the category $\cat{Mon}_\comp(\V)$ of
$\comp$-monoids in $\V$. This monoidal structure has the two properties
that (i) its unit object $I$ is initial; (ii) for each $A, B \in
\cat{Mon}_\comp(\V)$, the two maps
\begin{equation}
\label{eq:26}  \cd{A \ar[r]^-{\cong} & A \comm I \ar[r]^-{A \comm !} & A \comm B & I \comm B \ar[l]_-{{!} \comm B} & B \ar[l]_-{\cong}}
\end{equation}
are jointly epimorphic. Monoidal structures with these two
properties were studied in detail in~\cite[\S 2]{Janelidze2009Cover},
and characterised in terms of properties of the ``generalised commutation
relation'' on cospans that induces them.

\subsection{Commutative $\comp$-monoids}
\label{sec:commutativity-duoids}
In the one-object case, there is a natural definition of
\emph{commutative} $\comp$-monoid; we now give this together with a
number of alternative characterisations of the notion.
\begin{Defn}
  \label{def:20}
  Let $A$ be a $\comp$-monoid in the normal duoidal $\V$. We say that
  $A$ is \emph{commutative} if $1_A \colon A \rightarrow A
  \leftarrow A \colon 1_A$ is a commuting cospan; equivalently, if the
  following square commutes in $\V$:
  \begin{equation}
    \cd[@!C@-1em@C-1em]{
      & A \comp A \ar[dr]^-{m} \\
      A \conv A \ar[ur]^-{\sigma} \ar[dr]_-{\tau}  & & A\rlap{ .} \\
      & A \comp A \ar[ur]_-{m}
    }
  \end{equation}
\end{Defn}

Comparing this definition with Section~\ref{sec:bimonoids-duoids}, we
see that, in a normal duoidal category, both \emph{duoids} and
\emph{commutative $\comp$-monoids} are generalisations of commutative
monoids in a braided monoidal category. The following proposition
shows that, in fact, they are the same generalisation.
\begin{Prop}
  \label{prop:11}
  Let $\V$ be a normal duoidal category. The forgetful functor from
  duoids to $\comp$-monoids 
  $U \colon \cat{Duoid}(\V) \rightarrow \cat{Mon}_\comp(\V)$ is
  injective on objects and fully faithful, and its image comprises the
  commutative $\comp$-monoids. Moreover, if $\V$ is $\conv$-braided,
  then every duoid in $\V$ is $\conv$-commutative.
\end{Prop}
\begin{proof}
  Let the maps
  $e \colon \I \rightarrow A \leftarrow A \comp A \colon m$ and
  $\iota \colon \J \rightarrow A \leftarrow A \conv A \colon \nu$
  exhibit $A$ as a duoid. Since $e \colon \I \rightarrow A$ is a map
  of $\conv$-monoids, we have
  $\iota = e\upsilon \colon \J \rightarrow \I \rightarrow A$ and so
  $\iota$ is determined by $e$. Moreover, by precomposing the
  axiom~\eqref{eq:7} by the map
  \begin{equation}\label{eq:6}
    A \conv A \cong  (A \comp \I) \conv (\I \comp A) \xrightarrow{(A \comp e) \conv (e \comp
      A)} (A \comp A) \conv (A \comp A)\rlap{ ,}
  \end{equation}
  the lower and upper sides become
  $\nu \colon A \conv A \rightarrow A$ and
  $m\sigma \colon A \conv A \rightarrow A \comp A \rightarrow A$
  respectively, so that $\nu$ is determined by $m$. Thus $U$ is
  injective on objects, and the formulae $\iota = e \upsilon$ and
  $\nu = m \sigma$ now imply easily that it is fully faithful too.
  
  Replacing $(A \comp e) \conv (e \comp A)$ by
  $(e \comp A) \conv (A \comp e)$ in~\eqref{eq:6}, we see that any
  duoid $A$ verifies $\nu = m \tau$ as well as
  $\nu = m \sigma$; so $m\tau = m\sigma$, which is the condition for
  the underlying $\comp$-monoid to be commutative. Moreover, if $\V$
  is $\conv$-braided, then by~\eqref{eq:4} we have
  $\nu = m \sigma = m \tau c = \nu c \colon A \conv A \rightarrow A$,
  so that the duoid $A$ is necessarily $\conv$-commutative.

  All that remains is to show that that \emph{every} commutative
  $\comp$-monoid $A$  is in the image of
  $U$. Of course, the $\conv$-monoid structure on $A$ must be
  given by $e \upsilon \colon \J \rightarrow A$ and
  $m \sigma = m \tau \colon A \conv A \rightarrow A$; it is now direct
  that this is an $\conv$-monoid, and that $e \colon \I \rightarrow A$
  is a map of $\conv$-monoids. To verify~\eqref{eq:7}, we first verify
  that the following diagrams commute in any normal duoidal category:
  \begin{equation}
    \cd[@C+0.5em]{
      (A \comp B) \conv (C \comp D) \ar[d]_-{\xi} \ar[r]^-{\delta^{\ell}_{\ell}} &
      ((A \comp B) \conv C) \comp D \ar[r]^-{\delta^{r}_{r} \comp D} & 
      (A \comp (B \conv C)) \comp D \ar[d]^-{(A \comp \tau) \comp D} \\
      (A \conv C) \comp (B \conv D) \ar[r]^-{\sigma \comp \sigma} &
      (A \comp C) \comp (B \comp D) \ar[r]^-{\cong} &
      (A \comp (C \comp B)) \comp D
    }
  \end{equation}
  \begin{equation}
    \cd[@C+0.5em]{
      (A \comp B) \conv (C \comp D) \ar@{=}[d] \ar[r]^-{\delta^{\ell}_{\ell}} &
      ((A \comp B) \conv C) \comp D \ar[r]^-{\delta^{r}_{r} \comp D} & 
      (A \comp (B \conv C)) \comp D \ar[d]^-{(A \comp \sigma) \comp D} \\
      (A \comp B) \conv (C \comp D) \ar[r]^-{\sigma} &
      (A \comp B) \comp (C \comp D) \ar[r]^-{\cong}&
      (A \comp (B \comp C)) \comp D\rlap{ .}
    }
  \end{equation}
  On taking $A=B=C=D$ and postcomposing with the quaternary
  multiplication map
  $m.(m \comp A).((A \comp m) \comp A) \colon (A \comp (A \comp A))
  \comp A \rightarrow A$,
  commutativity of $A$ implies that the two upper paths become
  equal; whence the two lower paths do too. But these two paths are easily
  seen to be the two sides of~\eqref{eq:7}, which thus commutes.
\end{proof}
In the situation where we have the commuting monoidal structure on
$\comp$-monoids, we may give a further characterisation of the
commutative $\comp$-monoids.

\begin{Prop}
\label{prop:8}
Let $\V$ be a normal duoidal category for which $\cat{Mon}_\comp(\V)$
admits the commuting monoidal structure $\comm$. The forgetful functor
$U \colon \cat{Mon}_\comm(\cat{Mon}_\comp(\V)) \rightarrow
\cat{Mon}_\comp(\V)$
is injective on objects and fully faithful, and its image comprises
the commutative $\comp$-monoids. If $\V$ is $\conv$-braided, then
every $\comm$-monoid in $\cat{Mon}_\comp(\V)$ is commutative.
\end{Prop}
\begin{proof}
  This relies solely on the two properties of the $\comm$-monoidal
  structure on $\cat{Mon}_\comp(\V)$ noted in
  Section~\ref{sec:comm-tens-prod} above: that the unit is initial,
  and that the maps~\eqref{eq:26} are jointly epimorphic. A full proof
  of the result from these assumptions is given
  in~\cite[Theorem~2.8.2]{Janelidze2009Cover}; we reproduce it here
  for completeness. Of course, the unit of any $\comm$-monoid $A$ in
  $\cat{Mon}_\comp(\V)$ is necessarily the unique map
  $\eta \colon I \rightarrow A$ from the initial object, and by the
  epimorphicity of~\eqref{eq:26}, the multiplication $\mu$ must be the
  unique map fitting into a diagram
  \begin{equation}\label{eq:3}
    \cd{
      A \ar[d]_-{\cong} \ar[r]^-{1} & A & A \ar[d]^-{\cong} \ar[l]_-{1}  \\
      I \comm A \ar[r]_-{\eta \comm A} & A \comm A \ar@{-->}[u]^-{\mu} & A \comm I \ar[l]^-{A \comm \eta}\rlap{ .}
    }
  \end{equation}
  So $U$ is injective on objects, and clearly faithful; for fullness,
  we observe that any map $f \colon A \rightarrow B$ of $\comp$-monoids between
  $\comm$-monoids must satisfy
  $\mu.(f \comm f) = f.\mu \colon A \comm A \rightarrow B$, since both
  sides precompose with the jointly epimorphic cospan
  $A \rightarrow A \comm A \leftarrow A$ to yield the cospan
  $f \colon A \rightarrow B \leftarrow A \colon f$. Commutativity
  of~\eqref{eq:3}, together with the universal property of $\comm$,
  shows that any $\comm$-monoid in $\cat{Mon}_\comp(\V)$ has a
  commutative underlying $\comp$-monoid; it remains to show that every
  commutative $\comp$-monoid lies in the image of $U$. The
  commutativity and the universal property of $\comm$ yields the
  existence of a map $\mu$ as in~\eqref{eq:3} for which the unique map
  $\eta \colon I \rightarrow A$ is a unit; to show that
  $(A, \eta, \mu)$ is a $\comm$-monoid, it thus remains to check the
  associativity axiom
  $\mu . (\mu \comm A) = \mu . (A \comm \mu) \colon A \comm A \comm A
  \rightarrow A$;
  which is clear on precomposition with the three jointly epimorphic
  maps $A \rightarrow A \comm A \comm A$.

  Finally, when $\V$ is $\conv$-braided, the $\comm$-tensor product on
  $\cat{Mon}_\comp(\V)$ admits a symmetry $c$; now for any
  $\comm$-monoid $(A, \eta, \mu)$, we have that $\mu = \mu c \colon A
  \comm A \rightarrow A$, since both sides precompose with the jointly
  epimorphic cospan $A \rightarrow A \comm A \leftarrow A$ to yield
  the identity cospan $1_A \colon A \rightarrow A \leftarrow A \colon 1_A$.
\end{proof}

\section{Example: algebraic theories}
\label{sec:exampl-algebr-theor}
This concludes our development of the general theory of commutativity;
in the remainder of the paper, we apply it to a range of examples,
starting in this section with (finitary) \emph{algebraic theories}. In
this context, there are well-known notions of \emph{commuting tensor
  product} of theories and of \emph{commutative} algebraic
theory~\cite{Freyd1966Algebra}; our objective is to see
how these arise as particular instances of our general notions.

\subsection{Algebraic theories and commutativity}
\label{sec:algebr-theor-comm}
There are many ways of presenting algebraic
theories---see~\cite{Hyland2007The-Category} for an overview---but for
our purposes, we take the following perspective. We write $\F$ for the
presheaf category $[\mathbb F, \cat{Set}]$, where $\mathbb F$ is the
category of finite cardinals and mappings; now restriction and left
Kan extension along the inclusion
$I \colon \mathbb F \rightarrow \cat{Set}$ exhibits $\F$ as equivalent
to the category $\cat{End}_\omega(\cat{Set})$ of
filtered-colimit-preserving endofunctors of $\cat{Set}$, and under
this equivalence, the composition monoidal structure on
$\cat{End}_\omega(\cat{Set})$ transports to the 
\emph{substitution} monoidal structure on $\F$, whose unit object is the
inclusion $I$, and whose binary tensor is given by
\begin{equation}
  \label{eq:29}
  (A \comp B)(n) = \textstyle\int^{m \in \cat F} Am \times (Bn)^m\rlap{ .}
\end{equation}

By an \emph{algebraic theory}, we mean a $\comp$-monoid
$(T, \eta, \mu)$ in $\F$. We call elements $\alpha \in Tn$ the
\emph{$n$-ary operations} of $T$, write
 $\pi_0, \dots, \pi_{n-1} \in Tn$ for the elements in the
image of $\eta_n \colon n \rightarrow Tn$, and, given $f \in Tm$ and
$g_1, \dots, g_n \in Tn$, write $f (g_1, \dots, g_n)$ for the image
under
$\mu_n \colon \int^{m \in \mathbb F} Tm \times (Tn)^m \rightarrow Tn$
of the element $(f, g_1, \dots, g_n)$.


\begin{Defn}
  \label{def:9}
  Let $\E$ be a category with finite powers, and let $T$ be an
  algebraic theory. A \emph{model} of $T$ in $\E$ is an
  object $X \in \E$ together with functions
  $\dbr{\alpha} \colon X^n \rightarrow X$ for each $\alpha \in Tn$
  such that $\dbr{\pi_i} = \pi_i \colon X^n \rightarrow X$ for each $0
  \leqslant i < n$, and such that
  \begin{equation}
    \dbr{f(g_1, \dots, g_n)} = X^m \xrightarrow{(\dbr{g_1}, \dots, \dbr{g_n})} X^n \xrightarrow{\dbr{f}} X
  \end{equation}
  for all $f \in Tn$ and $g_1, \dots, g_n \in Tm$. A
  \emph{homomorphism of models} from $X$ to $Y$ is a map $f \colon X
  \rightarrow Y$ in $\E$ such that $f.\dbr{\alpha}_X =
  \dbr{\alpha}_Y.f^n$ for all $\alpha \in Tn$. We write $\cat{Mod}(T;
  \E)$ for the category of $T$-models and model homomorphisms in $\E$.
\end{Defn}
For any $\E$ with finite powers, it is easy to see that the category
$\cat{Mod}(T; \E)$ of the preceding definition again has finite
powers, created by the evident forgetful functor to $\E$. Any
finite-power-preserving functor $\E \rightarrow \G$ induces a
finite-power-preserving functor
$\cat{Mod}(T; \E) \rightarrow \cat{Mod}(T; \G)$; similarly, any
morphism $S \rightarrow T$ of finitary monads induces a
finite-power-preserving functor
$\cat{Mod}(T; \E) \rightarrow \cat{Mod}(S; \E)$ commuting with the
forgetful functors to $\E$.

\begin{Defn}
  \label{def:21}
  Let $S$ and $T$ be algebraic theories and let $\E$ be a
  category with finite powers. By a \emph{commuting $S$-$T$-model} in
  $\E$, we mean a $T$-model in $\cat{Mod}(S; \E)$, or equivalently, an
  $S$-model in $\cat{Mod}(T; \E)$. We write $\cat{Mod}(S, T; \E)$ for
  either of the isomorphic categories $\cat{Mod}(S; \cat{Mod}(T; \E)) \cong \cat{Mod}(T;
  \cat{Mod}(S; \E))$.
\end{Defn}
With the above definitions in place, we can now define the (commuting)
tensor product of theories, and the notion of
commutative algebraic theory.

\begin{Defn}
  \label{def:22}
  The \emph{tensor product} of algebraic theories $S$ and $T$ is a
  theory $U$ equipped with an isomorphism
  $\cat{Mod}(U; \cat{Set}) \cong \cat{Mod}(S, T; \cat{Set})$ over
  $\cat{Set}$. An algebraic theory $T$ is \emph{commutative} if the
  diagonal
  $\cat{Mod}(T; \cat{Set}) \rightarrow \cat{Mod}(T; \cat{Set})
  \times_\cat{Set} \cat{Mod}(T; \cat{Set})$
  factors through $\cat{Mod}(T, T; \cat{Set})$.
\end{Defn}

\subsection{The duoidal category $\F$}
\label{sec:duoid-categ-finit}
Our goal now is to make the substitution monoidal structure on $\F$
into part of a normal duoidal one in such a way that our general
commutativity notions reduce to those of Definition~\ref{def:22}; to
obtain the second tensor product, we make use of \emph{Day
  convolution}~\cite{Day1970Construction}.

\begin{Defn}
  \label{def:11}
  Suppose that $\A$ and $\V$ are monoidal categories.
  \begin{enumerate}[(i)]
  \item A \emph{convolution tensor} of
  $F, G \in [\A, \V]$ is a
  functor $F \conv G \in [\A, \V]$ equipped with a universal natural family of
  maps
  $u_{AB} \colon FA \otimes GB \rightarrow (F \conv G)(A \otimes B)$;
  this means that any natural family
  $k_{AB} \colon FA \otimes GB \rightarrow H(A \otimes B)$ is of the
  form $k_{AB} = \bar k_{A \otimes B}.u_{AB}$ for a unique 
  map $\bar k \colon F \conv G \rightarrow H$ in $[\A, \V]$.\vskip0.25\baselineskip 
\item A \emph{nullary convolution tensor} is a functor
  $\J \in [\A, \V]$ together with a universal map
  $j \colon \I \rightarrow J(I)$; this means that each map
  $k \colon \I \rightarrow H(I)$ in $\V$ is of the form $\bar k_\I . j$ for a
  unique $\bar k \colon \J \rightarrow H$ in $[\A, \V]$.
\end{enumerate}
\end{Defn}
When $\A$ is small, and $\V$ is biclosed, complete and cocomplete, all
convolution tensors exist, and underlie a biclosed monoidal structure
on $[\A, \V]$, which will be braided or symmetric whenever the
monoidal structures on  $\A$ and $\V$ are so.
The unit of this monoidal structure is the copower
$\J = \A(\I, \thg) \cdot \I$, while binary tensors and internal homs are
given by the formulae:
\begin{equation}\label{eq:9}
  \begin{gathered}
  F \conv G = \textstyle\int^{A,B \in \A} \A(A \otimes B, \thg) \cdot FA \otimes GB\\ 
  [F,G]_\ell = \textstyle\int_{A\in \A} [FA, G(A \otimes \thg)]_\ell \quad \text{and}\quad 
  [F,G]_r = \textstyle\int_{B \in \A} [FB, G(\thg \otimes B)]_r \rlap{ .}
  \end{gathered}
\end{equation}

In the case of $\F = [\mathbb F, \cat{Set}]$, applying Day convolution
with respect to product in both $\mathbb F$ and $\cat{Set}$ yields a symmetric monoidal
closed structure $(\conv, \J)$ on $\F$.
\begin{Prop}
  \label{prop:9}
  $(\F, \conv, \J, \comp, \I)$ is a
  normal duoidal
  category.
\end{Prop}
\begin{proof}
  Since $\J = \cat{Set}(1, \thg) \cdot 1$ and $\I$ is the inclusion
  functor $\mathbb F \rightarrow \cat{Set}$, the evaluation maps
  $\cat{Set}(1, X) \cdot 1 \rightarrow X$ provide an isomorphism
  $\upsilon \colon \J \rightarrow \I$ of unit objects; the other unit
  structure maps $\mu \colon \I \conv \I \rightarrow \I$ and
  $\gamma \colon J \rightarrow J \comp J$ are determined uniquely from
  this, and it remains only to give the interchange maps
  $\xi \colon (F \comp G) \conv (H \comp K) \rightarrow (F \conv H)
  \comp (G \conv K)$.
  
  We have already observed that the $\comp$-monoidal structure on
  $\F$ transports to give composition in
  $\cat{End}_\omega(\cat{Set})$; on the other hand, for any
  $F, G, H \in \cat{End}_\omega(\cat{Set})$, each natural family
  $FA \times GB \rightarrow H(A \times B)$ is uniquely determined by
  its components at finite cardinals $A$ and $B$, and it follows that
  the $\conv$-tensor product on $\F$ transports to
  convolution in
  $\cat{End}_\omega(\cat{Set}) \subset [\cat{Set}, \cat{Set}]$.
  Consequently, to give the maps $\xi$ it suffices to give natural
  families of maps
  $FG(A) \times HK(B) \rightarrow (F \conv H)(G \conv K)(A \times B)$
  for all finitary endofunctors $F, G, H, K$ of $\cat{Set}$. We obtain
  these as the composites
  \begin{equation}
    FG(A) \times HK(B) \xrightarrow{\!\!u_{GA, KB}\!\!} 
    (F \conv H)(GA \times KB) \xrightarrow{\!\!(F\conv H)(u_{A,B})\!\!} 
    (F \conv H)(G \conv K)(A \times B)\rlap{ .}
  \end{equation}
  The axioms for a duoidal category are now easily verified by
  exploiting the universal property of convolution.
\end{proof}
Applying our general framework to the normal duoidal $\F$ thus yields
a theory of commutativity for algebraic theories (=~$\comp$-monoids in
$\F$); in particular, we have notions of commuting tensor product of
theories, and of commutative algebraic theory. In order to identify
these notions with those of Definition~\ref{def:22}, we will first
identify categories of models in our framework as certain \emph{functor
  $\F$-categories}.

\subsection{$\F$-enriched category theory}
\label{sec:catend_fc-enrich-cat}
The theory of categories enriched over $(\F, \comp, \I)$ was the
object of study of~\cite{Garner2014Lawvere}; one of its main results
identifies categories $\E$ with finite powers as $\F$-enriched
categories admitting certain enriched \emph{absolute
  colimits}~\cite{Street1983Absolute}. For our purposes, the salient
points of this result are summarised by:
\begin{Prop}
  \label{prop:1}
  To each category $\E$ with finite powers we may associate an
  $\F$-category $\underline \E$ with the
  same objects as $\E$ and hom-objects given by:
  \begin{equation}
    \underline \E(X, Y)(n) = \E(X^n, Y)\rlap{ .}
  \end{equation}
  The assignation $\E \mapsto \underline \E$ underlies a
  $2$-fully faithful embedding
  $\cat{CAT}_{fp} \rightarrow \F\text-\cat{CAT}$
  of the $2$-category of categories with finite powers and
  power-preserving functors into the $2$-category of
  $\F$-enriched categories.
\end{Prop}
\begin{proof}
  The composition morphisms
  $\underline \E(Y, Z) \comp \underline \E(X,Y) \rightarrow \underline
  \E(X,Z)$
  of $\underline \E$ are induced by the dinatural family of maps
  \begin{equation}
    \begin{aligned}
      \E(Y^m, Z) \times \E(X^n, Y)^m & \rightarrow \E(X^n, Z) \\
      (f, g_1, \dots, g_m) & \mapsto f(g_1, \dots, g_m)
    \end{aligned}
  \end{equation}
  while the $n$-component of the identities map
  $I \rightarrow \underline \E(X,X)$ is the map
  $n \rightarrow \E(X^n, X)$ picking out the $n$ projection maps
  $\pi_0, \dots, \pi_{n-1}$. For the remaining details, we refer the
  reader to~\cite[Proposition~3.8]{Garner2014Lawvere}.
\end{proof}
The additional aspect of the theory enabled by our framework for
commutativity is the existence of a symmetric monoidal closed
structure on $\F\text-\cat{Cat}$. Indeed, the normal duoidal $\F$ is
$\ast$-symmetric and $\ast$-closed; moreover, its underlying category
$\F$ is locally presentable (in particular complete) and the
$\comp$-tensor product thereon is accessible, whence by
Proposition~\ref{prop:22}, $\F\text-\cat{Cat}$ is cocomplete and free
$\F$-categories exist. So Proposition~\ref{prop:17} applies to show
that $\F\text-\cat{Cat}$ admits a symmetric monoidal closed structure
$(\comm, \mathcal I, [\thg, \thg])$ given by the commuting tensor
product of $\F$-categories.

\begin{Prop}
  \label{prop:2}
  For any category $\E$ with finite powers and any algebraic theory
  $T$, we have an isomorphism of $\F$-categories
  \begin{equation}
    \underline{\cat{Mod}(T; \E)} \cong [\Sigma T, \underline \E]\rlap{ .}
  \end{equation}
\end{Prop}
\begin{proof}
  To give an $\F$-functor
  $F \colon \Sigma T \rightarrow \underline \E$ is equally to give an
  object $X \in \E$ together with a map of $\comp$-monoids
  $f \colon T \rightarrow \underline \E(X,X)$ in $\F$; now using the definition
  of $\underline \E$ in Proposition~\ref{prop:1}, we see that this is
  precisely to give a $T$-model in $\E$. Given another $\F$-functor
  $G \colon \Sigma T \rightarrow \underline \E$, corresponding to a
  $\comp$-monoid map $g \colon T \rightarrow \underline \E(Y,Y)$, say, we may
  trace through the construction of Proposition~\ref{prop:15} to find
  that the hom-object $[\Sigma T, \underline \E](F, G)$ is obtained
  by forming the hexagon:  \begin{equation}
    \cd[@!C@C-7.7em@R-0.5em]{
    & \sh{l}{1.7em}\underline \E(X,Y) \comp T \ar[rr]^-{1 \comp f} & &
    \sh{r}{2em}\underline \E(X,Y) \comp \underline \E(X,X) \ar[dr]^m \\ 
    \sh{r}{0.5em} \underline \E(X,Y) \conv T \ar[ur]^-{\sigma} \ar[dr]_-{\tau} & & & &
    \underline \E(X,Y)\rlap{ ,} \\
    &\sh{l}{1.7em}T \comp \underline \E(X,Y) \ar[rr]_-{g \comp 1} & &
    \sh{r}{2em}\underline \E(Y,Y) \comp \underline \E(X,Y) \ar[ur]_m
  }
  \end{equation}
  and then taking the equalizer of the transposed parallel pair
  $\underline \E(X,Y) \rightrightarrows [T, \underline \E(X,Y)]$. The
  two sides of the preceding hexagon send a pair
  $(f \in \E(X^n, Y), \alpha \in Tm)$ to the respective composites
  \begin{equation}
    X^{nm} \cong (X^m)^n \xrightarrow{\dbr{\alpha}_X^n} X^n \xrightarrow{f} Y \qquad \text{and} \qquad X^{nm} \cong (X^n)^m \xrightarrow{f^m} Y^m \xrightarrow{\dbr{\alpha}_Y} Y
  \end{equation}
  from which it follows that $[\Sigma T, \underline \E](F, G)(n)$ is
  given by the set of all $f \in \E(X^n, Y)$ such that
  $f.\dbr{\alpha}_X^n = \dbr{\alpha}_Y.f^m$ for all $\alpha \in Tm$;
  that is, by the set of $T$-model homomorphisms from $X^n$ to $Y$. It
  is easy to see that this bijection respects composition, and so we have
  $\underline{\cat{Mod}(T; \E)} \cong [\Sigma T,
  \underline \E]$ as required.
\end{proof}
Using this result, we may identify the notions of our general theory
with those given by Definition~\ref{def:22}.
\begin{Cor}
  \label{cor:4}
  An algebraic theory $U$ is the tensor product of theories $S$ and
  $T$ if and only if $\Sigma U \cong \Sigma S \comm \Sigma T$ as $\F$-categories.
\end{Cor}
\begin{proof}
  As is
  well-known~\cite[Thm~III.1.1~\&~III.1.2]{Lawvere1963Functorial}, the
  functor $\cat{AlgTh}^\op \rightarrow \cat{CAT} / \cat{Set}$ sending each
  theory $T$ to the forgetful functor
  $\cat{Mod}(T; \cat{Set}) \rightarrow \cat{Set}$ is fully faithful;
  and so $\Sigma U \cong \Sigma S \comm \Sigma T$ if and only if
  $\cat{Mod}(U; \cat{Set}) \cong \cat{Mod}(S \comm T; \cat{Set})$ over
  $\cat{Set}$. But by the preceding result, we have that
  \begin{equation}\label{eq:27}
    \begin{aligned}
    \underline{\cat{Mod}(S \comm T; \cat{Set})} &\cong [\Sigma S \comm \Sigma T, \underline{\cat{Set}}] \cong
[\Sigma S, [\Sigma T, \underline{\cat{Set}}]] \cong
[\Sigma S, \underline{\cat{Mod}(T; \cat{Set})}] \\ &\cong
\underline{\cat{Mod}(S; \cat{Mod}(T; \cat{Set}))} = \underline{\cat{Mod}(S, T; \cat{Set})}
    \end{aligned}
  \end{equation}
  as categories over $\cat{Set}$; whence
  $\Sigma U \cong \Sigma S \comm \Sigma T$ if and only if $U$ is the
  tensor product of $S$ and $T$ in the sense of
  Definition~\ref{def:22}.
\end{proof}
\begin{Cor}
  \label{cor:2}
  An algebraic theory $T$ is commutative in the sense of
  Definition~\ref{def:22} just when it is a commutative $\comp$-monoid
  in the sense of Definition~\ref{def:20}.
\end{Cor}
\begin{proof}
$T$ is commutative just when there is a factorisation as in the diagram
\begin{equation}
  \cd[@C-1.5em@R-0.5em]{
    & \cat{Mod}(T; \cat{Set}) \ar@/_1em/[ddl]_-{1} \ar@/^1em/[ddr]^-{1} \ar@{-->}[d] \\
    & \cat{Mod}(T, T; \cat{Set}) \ar[dl]^-{\pi_1} \ar[dr]_-{\pi_2} \\
    \cat{Mod}(T; \cat{Set}) & & \cat{Mod}(T; \cat{Set})\rlap{ .}
  }
\end{equation}
By the identification~\eqref{eq:27} of $\cat{Mod}(T, T; \cat{Set})$
with $\cat{Mod}(T \comm T; \cat{Set})$, together with the full
fidelity of the model functor
$\cat{AlgTh}^\op \rightarrow \cat{CAT}/\cat{Set}$, this is equally to
ask that the identity cospan $T \rightarrow T \leftarrow T$ of
$\comp$-monoids factor through the universal commuting cospan
$T \rightarrow T \comm T \leftarrow T$; which is to ask that $T$ be a
commutative $\comp$-monoid in $\F$.
\end{proof}
Note that the hypotheses of Example~\ref{ex:4} apply to $\F$, so that
there is a duoidal adjunction
$(\thg) \cdot \J \dashv \F(\J, \thg) \colon \F \rightarrow \cat{Set}$
inducing, by Proposition~\ref{prop:4}, a monoidal adjunction
$\F\text-\cat{Cat} \leftrightarrows \cat{Cat}$. Restricting this to
one-object categories gives a monoidal adjunction
$(\cat{AlgTh},\odot) \leftrightarrows (\cat{Mon},\times)$ between the category of
algebraic theories and the category of monoids, whose left adjoint
views a monoid as an algebraic theory with only unary operations, and
whose right adjoint sends an algebraic theory to its monoid of unary
operations. The monoidality of this adjunction tells us, in
particular, that a monoid is commutative just when the associated
algebraic theory is commutative; that the monoid of unary
operations of any commutative algebraic theory is commutative; and
that the theory associated to a product monoid $M \times N$ is the
tensor product of $M$ and $N$ \emph{qua} theories.

\subsection{Explicit formulae}
\label{sec:explicit-formulae}
There are well-known explicit formulae which give the tensor product
of two algebraic theories, and which characterise when an algebraic
theory is commutative; see~\cite{Freyd1966Algebra}, for example. We
conclude this section by showing how these formulae can be
reconstructed from our general framework. To do so, we first calculate
the interchange map $\xi \colon (X \comp Y) \conv (Z \comp W) \rightarrow (X \conv Z)
\comp (Y \conv W)$ of the duoidal structure on $\F$. Such a map
classifies a natural family of maps $(X \comp Y)(n) \times (Z \comp W)(m) \rightarrow [(X \conv Z)
\comp (Y \conv W)](nm)$, which, expanding out the definitions, is
equally a family of maps
\begin{equation}
  \begin{aligned}
    Xk \times (Yn)^k \times Z\ell \times (Wm)^\ell & \rightarrow \textstyle\int^k (X \conv Z)k \times (Y \conv W)(nm)^k
  \end{aligned}
\end{equation}
natural in $n$ and $m$ and dinatural in $k$ and $\ell$; which we calculate
to be given by
\begin{equation}\label{eq:45}
  (x, y_1, \dots, y_k, z, w_1, \dots, w_\ell) \mapsto (u(x,z), u(y_1, w_1), \dots, u(y_k, w_k))
\end{equation}
(where  $u \colon Xk \times Z\ell \rightarrow (X \conv
Z)(k\ell)$ and $u \colon Yn \times Wm \rightarrow (Y \conv W)(nm)$ are
the universal maps into the convolution tensor).
It follows that the map $\sigma \colon X \conv Y \rightarrow X \comp Y$
associated to the duoidal structure on $\F$ classifies the natural
family
\begin{equation}
  \begin{aligned}
    Xn \times Ym & \rightarrow (X \comp Y)(nm) = \textstyle\int^k Xk \times Y(nm)^k  \\
    (f, g) & \mapsto (f, Y\alpha_1(g), \dots, Y\alpha_n(g))
  \end{aligned}
\end{equation}
where here $\alpha_j \colon n \rightarrow nm$ is the injection defined
by $\alpha_j(i) = (i,j)$. Dually, the map $\tau \colon X \conv Y \rightarrow Y
\comp X$ classifies the natural family
\begin{equation}
  \begin{aligned}
    Xn \times Ym & \rightarrow (X \comp Y)(nm) = \textstyle\int^k Xk \times Y(nm)^k  \\
    (f, g) & \mapsto (g, X\beta_1(f), \dots, X\beta_m(f))
  \end{aligned}
\end{equation}
where now $\beta_i \colon m \rightarrow nm$ is given by
$\beta_i(j) = (i, j)$. Given these calculations, we now see that, if $T$ is an
algebraic theory, then a cospan $f \colon A \rightarrow T \leftarrow B \colon g$
in $\F$ is commuting just when, for each $\phi = fa \in Tn$ and each
$\psi = gb \in Tm$, we have that
\begin{equation}
  \psi(\phi(\pi_{11}, \dots, \pi_{n1}), \dots, \phi(\pi_{1m}, \dots, \pi_{nm})) =
  \phi(\psi(\pi_{11}, \dots, \pi_{1m}), \dots, \psi(\pi_{n1}, \dots, \pi_{nm}))
\end{equation}
as $nm$-ary operations. In particular, we see that an algebraic theory
$T$ is commutative just when this equality holds for all pairs of
operations $\phi, \psi$ in $T$; furthermore, the commuting tensor
product $S \comm T$ of theories $S$ and $T$ is the quotient of the
coproduct theory $S + T$ which imposes these equalities for all
$\phi \in Sn$ and $\psi \in Tm$.

\section{Example: symmetric operads}
\label{sec:exampl-symm-oper}
Symmetric operads were introduced by May
in~\cite{May1972The-geometry}; while his interest was in the case of
\emph{topological} operads, one can define symmetric operads over any
braided monoidal base $\V$. We consider here the case $\V = \cat{Set}$
of symmetric plain operads; one way of understanding these is as
special kinds of algebraic theory, whose equations between derived
operations are generated by ones involving the same variables on each
side of the equality without omission or repetition.

The category of symmetric plain operads admits a monoidal structure
known as the \emph{Boardman--Vogt} monoidal
structure~\cite{Boardman1973Homotopy}; on viewing symmetric operads as
algebraic theories, their Boardman--Vogt tensor product is precisely
their tensor product as theories described in the previous section.
The purpose of this section is to exhibit this fact as a consequence
of our general framework for commutativity.

\subsection{Species and symmetric operads}
\label{sec:spec-symm-oper}
A \emph{species} is a functor
$X \colon \mathbb P \rightarrow \cat{Set}$, where $\mathbb P$ is the
category of finite cardinals and bijective mappings between them; a
symmetric operad is a monoid with respect to a suitable
\emph{substitution} tensor product on the category
$\esp = [\mathbb P, \cat{Set}]$ of species. 

To describe this, note first that sum and product of finite cardinals
induce convolution monoidal structures on $\esp$, which we denote by
$\oplus$ and $\conv$ respectively; the units of these monoidal
structures are the respective representable functors $O = y_0$ and
$J = y_1$. The sum monoidal structure on $\mathbb P$ in fact exhibits
it as the free symmetric monoidal category on the object $1$; it
follows by~\cite[Theorem~5.1]{Im1986A-universal} that
$(\esp, \oplus, O)$ is the \emph{free symmetric monoidal closed
  cocomplete category} on the object $y_1$. Writing
$\cat{CoctsStrMon}$ for the $2$-category of symmetric monoidal closed
cocomplete categories and cocontinuous symmetric strong monoidal
functors, this is to say that, for any $\V \in \cat{CoctsStrMon}$, the
functor $\cat{CoctsStrMon}(\esp, \V) \rightarrow \V$ which evaluates
at the object $y_1$ is an equivalence of categories. 

In particular, we have an equivalence
$\cat{CoctsStrMon}(\esp, \esp) \simeq \esp$, under which the
composition monoidal structure of the left-hand side transports to
yield the \emph{substitution} tensor product $\comp$ on
$\esp = [\mathbb P, \cat{Set}]$. This has as unit the representable
$I = y_1$, and binary tensor
\begin{equation}\label{eq:31}
  X \comp Y = \textstyle\int^{k \in \mathbb P} Xk \times Y^{\oplus k}
\end{equation}
where here we write $Y^{\oplus k}$ for the $k$-fold tensor product
$Y \oplus \dots \oplus Y$. Observe that to give a natural
transformation $\alpha \colon X \comp Y \rightarrow Z$ in $\esp$ is equally to give
a family of maps
\begin{equation}
  \label{eq:33}
  \tilde \alpha \colon Xk \times Ym_1 \times \dots \times Ym_k \rightarrow Z(\Sigma_i m_i)
\end{equation}
natural in the $m_i$'s and dinatural in $k$.

Now a \emph{symmetric operad} is a $\comp$-monoid $(T, \eta, \mu)$ in
$\esp$. Like in the preceding section, we refer to elements $f \in Tn$
as \emph{$n$-ary operations} of $T$; we write $\id \in T1$ for the
element classified by $\eta \colon y_1 \rightarrow T$; and given
$f \in Tn$ and $g_i \in Tm_i$ (for $i = 1, \dots, n$), we write
$f(g_1, \dots, g_n) \in T(\Sigma_i m_i)$ for their image under
$\tilde \mu$ as in~\eqref{eq:33}.

\subsection{Symmetric operads and theories}
\label{sec:symm-oper-theor}
As indicated above, we wish to identify symmetric operads with certain
kinds of algebraic theory. To this end, let us write
$H \colon \mathbb P \rightarrow \mathbb F$ for the (non-full)
inclusion functor, and write $\Th \colon \esp \rightarrow \F$
for the left Kan extension functor $\Lan_H \colon [\mathbb P,
\cat{Set}] \rightarrow [\mathbb F, \cat{Set}]$. We now have:
\begin{Prop}
  The functor $\Th \colon \esp \rightarrow \F$ is faithful and
  full on isomorphisms, and is strong monoidal as a functor
  \begin{equation}
    \label{eq:34}
    (\esp, \mathord\oplus) \rightarrow (\F, \mathord\times) \qquad \text{and} \qquad 
    (\esp, \mathord\conv) \rightarrow (\F, \mathord\conv) \qquad \text{and} \qquad 
    (\esp, \mathord\comp) \rightarrow (\F, \mathord\comp)\rlap{ ,}
  \end{equation}
  and symmetric in the first two cases. It follows that
  $\Th$ exhibits $\esp$ as equivalent to a (non-full) subcategory of
  $\F$ closed under the $\comp$ and $\conv$ monoidal structures.
\end{Prop}
\begin{proof}
  For fidelity and fullness on isomorphisms, we refer
  to~\cite[Proposition~1]{Joyal1986Foncteurs}
  or~\cite[Proposition~10.9]{Weber2004Generic}. For strong symmetric
  monoidality as a functor
  $(\esp, \mathord\oplus) \rightarrow (\F, \mathord\times)$ and
  $(\esp, \mathord\conv) \rightarrow (\F, \mathord\conv)$, we observe
  that $H \colon \mathbb P \rightarrow \mathbb F$ is strong symmetric
  monoidal with respect to $+$ and $\times$; whence
  $\Th = \Lan_H \colon [\mathbb P, \cat{Set}] \rightarrow [\mathbb F,
  \cat{Set}]$
  is strong symmetric monoidal with respect to the corresponding
  convolution tensor products, which are $\oplus$ and $\conv$ on
  $\esp$, and $\times$ and $\conv$ on $\F$; here we
  use~\cite[\S5]{Day1970On-closed} to see that convolution with
  respect to coproduct on $\mathbb F$ is in fact cartesian product in
  $\F$. For strong monoidality of $\Th$ as a functor
  $(\esp, \comp) \rightarrow (\F, \comp)$, we refer
  to~\cite[Section~2.1(iv)]{Joyal1986Foncteurs}. 

  For the final clause of the proposition, let $\F' \subset \F$ be the
  subcategory whose objects are those isomorphic to ones in the image
  of $\mathrm{Th}$, and whose morphisms $f \colon X \rightarrow Y$ are those
  for which there is an isomorphism $\Th(g) \cong f$ in $\F^\cat{2}$.
  It is easy to show using fidelity and fullness on isomorphisms that
  the corestriction $\esp \to \F'$ of $\Th$ is an equivalence of
  categories, as desired.
\end{proof}
Using this proposition, we can identify symmetric operads with certain
kinds of algebraic theory. In order to identify the Boardman--Vogt
tensor product of symmetric operads with the tensor product of the
associated theories, we will need to prove that $\esp$ is in fact
equivalent to a sub-duoidal category of $\F$. To do so, it will be
convenient to introduce the notion of a \emph{rig category}.
\begin{Defn}
  \label{def:7}
  A \emph{rig category}~\cite{Laplaza1972Coherence}
  $(\V, \oplus, O, \otimes, I)$ is a category $\V$ equipped with a
  symmetric monoidal structure $(\oplus, O)$ and a monoidal structure
  $(\otimes, I)$, together with a lifting of the functor
  $\V \rightarrow [\V, \V]$ sending $X$ to $X \otimes (\thg)$ to a
  functor $\V \rightarrow \cat{OpMon}_\oplus(\V, \V)$ which is
  opmonoidal with respect to the $\oplus$-tensor product on $\V$ and
  the pointwise $\oplus$-tensor product on
  $\cat{OpMon}_\oplus(\V, \V)$.
\end{Defn}
In more elementary terms, $\V$ is a rig category when its two monoidal structures
$\oplus$ and $\otimes$ are related by nullary constraint morphisms
$\alpha_\ell \colon O
\otimes X \rightarrow O$ and $\alpha_r \colon X \otimes O \rightarrow O$ and binary
constraint morphisms
\begin{equation}
  \label{eq:35}
(X \oplus Y) \otimes Z \xrightarrow{\nu_\ell} (X \otimes Z) \oplus (Y \otimes Z) \quad \text{and} \quad 
X \otimes (Y \oplus Z) \xrightarrow{\nu_r} (X \otimes Y) \oplus (X \otimes Z)
\end{equation}
satisfying suitable axioms. In fact, both $\esp$ and $\F$ are rig
categories as a consequence of the following result:
\begin{Prop}
  \label{prop:12}
  Let $\A$ be a small category and $\V$ a cartesian closed cocomplete
  one. If $(\A, \oplus, O, \otimes, I)$ is a rig category, then so is
  $[\A^\op, \V]$ equipped with the corresponding convolution monoidal
  structures.
\end{Prop}
\begin{proof}
  Writing $\oplus, O, \otimes, I$ also for the convolution tensors on
  $[\A^\op, \V]$, we see that to give the constraint maps
  $\alpha_\ell \colon O \otimes X \rightarrow O$ in $[\A^\op, \V]$ is
  equally to give maps
  $\tilde \alpha_\ell \colon XV \rightarrow \V(O \otimes V, O)$ in
  $\V$, natural in $V$; which we obtain as the composites
  \begin{equation}
    \label{eq:37}
    XV \xrightarrow{!} 1 \xrightarrow{(\alpha_\ell)_V} \V(O \otimes V, O)\rlap{ .}
  \end{equation}
  The constraint maps
  $\nu_\ell \colon X \otimes (Y \oplus Z) \rightarrow (X \otimes Y)
  \oplus (X \otimes Z)$
  at $X,Y,Z \in [\A^\op, \V]$ must by the universal property of
  convolution be induced by natural families
  \begin{equation}
    \label{eq:30}
    Xn \times Ym \times Zk \rightarrow [(X \otimes Y) \oplus (X \otimes Z)](n \otimes (m \oplus k))
  \end{equation}
  in $\V$. We obtain these as the composites
  \begin{equation}
    \label{eq:32}
    \begin{aligned}
      Xn \times Ym \times Zk &\xrightarrow{\mathmakebox[7em]{(\pi_1, \pi_2, \pi_1, \pi_3)}} Xn \times Ym \times Xn \times Zk \\
      &\xrightarrow{\mathmakebox[7em]{u \times u}} (X \otimes Y)(n \otimes m) \times (X \otimes Z)(n \otimes k) \\
      &\xrightarrow{\mathmakebox[7em]{u}} [(X \otimes Y) \oplus (X \otimes Z)]((n \otimes m) \oplus (n \otimes k))\\
      & \xrightarrow{\mathmakebox[7em]{[(X \otimes Y) \oplus (X \otimes Z)](\nu_\ell)}} [(X \otimes Y) \oplus (X \otimes Z)](n \otimes (m \oplus k))\rlap{ .}
    \end{aligned}
  \end{equation}
  We define $\alpha_r$ and $\nu_r$ similarly; the coherence axioms for
  a rig now follow from those in $\A$ and the universal property of convolution.
\end{proof}
Since both $\mathbb P^\op$ and $\mathbb F^\op$ are rig categories under
disjoint union and product of finite cardinals, we conclude that
$(\esp, \oplus, O, \conv, J)$ and $(\F, \times, 1, \conv, J)$ are both
rig categories; since $H^\op \colon \mathbb P^\op \rightarrow \mathbb F^\op$ is a
strong morphism of rig categories, so too is $\Th \colon \esp
\rightarrow \F$.

\begin{Prop}
  \label{prop:5}
  $(\esp, \conv, \J, \comp, \I)$ is a normal duoidal category and
  $\Th \colon \esp \rightarrow \F$ a strong duoidal functor;
  whence $\Th$ exhibits $\esp$ as equivalent to a sub-duoidal
  category of~$\F$.
\end{Prop}
\begin{proof}
  As in Proposition~\ref{prop:9}, the units $J$ and $I$ of the $\conv$
  and $\comp$ monoidal structures on $\esp$ are both the representable
  $y_1$, and so the only difficulty lies in defining the interchange
  maps
  $\xi \colon (X \comp Y) \conv (W \comp Z) \rightarrow (X \conv W)
  \comp (Y \conv Z)$.
  Since $\conv$ is cocontinuous in each variable, and $\comp$ is
  cocontinuous in its first variable, it is enough to consider the
  case where $X$ and $W$ are representable; we must thus define maps
  $(y_n \comp Y) \conv (y_m \comp Z) \rightarrow (y_n \conv y_m) \comp
  (Y \conv Z) \cong y_{nm} \comp (Y \conv Z)$
  natural in $n,m \in \mathbb P$. By definition of $\comp$, this is
  equally to give natural maps
  $Y^{\oplus n} \conv Z^{\oplus m} \rightarrow (Y \conv Z)^{\oplus
    nm}$. We obtain these using the rig structure on $\esp$, as the
  composites
  \begin{equation}
    \label{eq:38}
    Y^{\oplus n} \conv Z^{\oplus m} \xrightarrow{\nu_r} (Y^{\oplus n} \conv Z)^{\oplus m} \xrightarrow{(\nu_\ell)^{\oplus m}} ((Y \conv Z)^{\oplus
    n})^{\oplus m} = (Y \conv Z)^{\oplus nm}
  \end{equation}
  where the two non-identity maps are built from repeated applications
  of the opmonoidal constraints $\nu_r$ and $\nu_\ell$ respectively.

  Leaving aside the duoidal coherence axioms in $\esp$ for the moment,
  we next show strong duoidality of $\Th \colon \esp \rightarrow \F$.
  Preservation of unit coherences is straightforward, so it
  suffices to show that each square of the form:
  \begin{equation}
    \label{eq:44}
    \cd{
      {(\Th(X) \comp \Th(Y)) \conv (\Th(W) \comp \Th(Z))} \ar[r]^-{\xi} \ar[d]_{\cong} &
      {(\Th(X) \conv \Th(W)) \comp (\Th(Y) \conv \Th(Z))} \ar@{<-}[d]^{\cong} \\
      {\Th((X \comp Y) \conv (W \comp Z))} \ar[r]^-{\Th(\xi)} &
      {\Th((X \conv W) \comp (Y \conv Z))}
    }
  \end{equation}
  commutes in $\F$. Note that, since $\Th$ is faithful, once we have
  this, we are done, since we may deduce the duoidal coherence axioms
  in $\esp$ from those in $\F$. To show the required commutativity,
  observe that all vertices in the preceding diagram are cocontinuous in the
  variables $X$ and $W$; so it is enough to consider the case where $X
  = y_n$ and $W = y_m$. Using again the definitions~\eqref{eq:29}
  and~\eqref{eq:31} of the $\comp$ tensor products, this means showing
  that each square
  \begin{equation}
    \label{eq:36}
    \cd{
      {\Th(Y)^n \conv \Th(Z)^m} \ar[r]^-{\xi} \ar[d]_{\cong} &
      {(\Th(Y) \conv \Th(Z))^{nm}} \ar@{<-}[d]^{\cong} \\
      {\Th(Y^{\oplus n} \conv Z^{\oplus m})} \ar[r]^-{\Th(\xi)} &
      {\Th((Y \conv Z)^{\oplus nm})}
    }
  \end{equation}
  is commutative. Since $\Th$ is a strong morphism of rig categories,
  the lower composite is the analogue of the map~\eqref{eq:38} for the
  rig $(\F, \times, \conv)$; as such, it is the map whose $(i,j)$th
  projection is $\pi_i \conv \pi_j \colon \Th(Y)^n \conv \Th(Z)^m
  \rightarrow \Th(Y) \conv \Th(Z)$; which, comparing with the
  formula~\eqref{eq:45} for interchange in $\F$, is precisely the upper map.
\end{proof}
\begin{Rk}
  \label{rk:2}
  The normal duoidal structure on $\esp = [\mathbb P, \cat{Set}]$
  described by this proposition can in fact be constructed on
  $[\mathbb P, \V]$ for any cocomplete cartesian closed $\V$; for the
  case of $\V$ being the category of simplicial sets, the interchange
  maps of this duoidal structure were described
  in~\cite[Proposition~1.20]{Dwyer2013The-Boardman-Vogt}. Note that if
  $\V$ is a non-cartesian symmetric monoidal category, then
  $[\mathbb P, \V]$ bears the two monoidal structures $(\comp, \I)$
  and $(\conv, \J)$ but is \emph{not} duoidal; the essential problem
  is that Proposition~\ref{prop:12} fails in this case, which
  obstructs the construction of the interchange maps.
\end{Rk}

\subsection{The Boardman--Vogt tensor product}
\label{sec:boardman-vogt-tensor}
The normal duoidal structure on the locally presentable $\esp$ just
described is $\ast$-symmetric and $\ast$-closed and has a
$\comp$-accessible tensor product, whence by
Propositions~\ref{prop:22} and~\ref{prop:17} as before,
$\esp\text-\cat{Cat}$ admits the commuting monoidal structure
$(\comm, \mathcal I)$. In particular, restricting to the one-object
case, we obtain a tensor product of symmetric operads on $\cat{Set}$,
called the \emph{Boardman--Vogt} tensor product. By transcribing the
explicit calculations of Section~\ref{sec:explicit-formulae} above, we
re-find the well-known formula describing this tensor product: given
operads $\O$ and $\P$, their tensor is obtained from the operadic
coproduct $\O + \P$ by quotienting out by the equalities
\begin{equation}
  \psi(\phi, \dots, \phi) =
  \phi(\psi, \dots, \psi) \cdot \sigma
\end{equation}
for each $\psi \in \O(n) \subset (\O + \P)(n)$ and $\phi \in
\P(m) \subset (\O + \P)(m)$; here $(\thg) \cdot \sigma$
indicates the action on $(\O + \P)(nm)$ of the symmetry $\sigma \colon
nm \rightarrow mn$.

The strong duoidal $\Th \colon \esp \rightarrow \F$ of the preceding
proposition has right adjoint given by restriction along
$H \colon \mathbb P \rightarrow \mathbb F$, and so is the left adjoint of a
duoidal adjunction $\Th \dashv [H, 1] \colon \F \rightarrow \esp$;
whence by Proposition~\ref{prop:4},
$\Th_\ast \colon (\esp\text-\cat{Cat}, \comm) \rightarrow
(\F\text-\cat{Cat}, \comm)$
is a strong monoidal $2$-functor---and so sends the Boardman--Vogt
tensor product of symmetric operads to the tensor product of theories,
as claimed above.

\section{Normalizing duoidal categories}
\label{sec:norm-duoid-categ}
Before giving our remaining examples, we break off briefly in order to
describe a construction which will be useful in producing them. This
construction assigns to any reasonably well-behaved duoidal category
with non-isomorphic units $\I$ and $\J$ a \emph{normalization} in
which $\I$ and $\J$ are forced to coincide. It involves taking
bimodules over the bimonoid $\I$, and is in fact an instance of a more
general bimodule construction, which we now explain.

\subsection{Bimodules over a bimonoid}
\label{sec:bimod-over-bimon}
Suppose that the maps $e \colon \J \rightarrow M \leftarrow M \conv M \colon m$
and $u \colon \I \leftarrow M \rightarrow M \comp M \colon d$ exhibit
$M$ as a bimonoid in the duoidal $\V$. We write $\cat{Bimod}_M$ for
the category of algebras for the monad $M \conv (\thg) \conv M$ on
$\V$, and call its objects \emph{$M$-bimodules}. By exploiting the
$\comp$-comonoid structure on $M$, we may lift the $\comp$-monoidal
structure on $\V$ to one on $\cat{Bimod}_M$, whose unit is $\I$
with the action
\begin{equation}
  M \conv \I \conv M \xrightarrow{u \conv \I \conv u} \I \conv \I \conv \I \xrightarrow{\mu \cdot (\I \conv \mu)} \I
\end{equation}
and whose binary tensor of
$a \colon M \conv A \conv M \rightarrow A$ and $b \colon M \conv B
\conv M \rightarrow B$ is $A \comp B$ equipped with the following
action---wherein we temporarily write $\comp$ as juxtaposition:
\begin{equation}
  M \conv AB \conv M \xrightarrow{d \conv AB \conv d} MM \conv AB \conv MM \xrightarrow{\xi \cdot (\xi \conv \I)} (M \conv A \conv M)(M \conv B \conv M) \xrightarrow{a \conv b}  AB\rlap{ .}
\end{equation}
Suppose now that $\V$ admits reflexive coequalizers which are
preserved by $\conv$ in each variable. Under these circumstances,
$\cat{Bimod}_M$ also admits the tensor product $\conv_{M}$, whose unit
is $M$ itself with the regular action, and whose binary tensor,
constructed by a reflexive coequalizer
$A \conv M \conv B \rightrightarrows A \conv B \twoheadrightarrow A
\conv_M B$,
classifies $M$-bilinear maps. We claim that, in fact,
$(\cat{Bimod}_M, \conv_M, M, \comp, \I)$ is a duoidal category.

The constraint maps in $\cat{Bimod}_M$ are obtained as follows. We
take $\upsilon$ and $\gamma$ to be $u \colon M \rightarrow \I$ and
$d \colon M \rightarrow M \comp M$ respectively; these are easily seen
to be maps of $M$-bimodules. To define $\mu$, we observe that
$\mu \colon \I \conv \I \rightarrow \I$ in $\V$ is $\I$-bilinear, and
so by restriction along the $\conv$-monoid map
$u \colon M \rightarrow \I$ also $M$-bilinear; it thus descends to the
required $\mu \colon \I \conv_M \I \rightarrow \I$ in $\cat{Bimod}_M$.
Finally, we must construct the interchange maps; so let
$A,B,C,D \in \cat{Bimod}_M$, and consider the diagram
\begin{equation}
  \cd[@R-0.3em]{
    {(A \comp B) \conv (C \comp D)} \ar[r]^-{\xi} \ar[d]_{q} &
    {(A \conv C) \comp (B \conv D)} \ar[d]^{q \comp q} \\
    {(A \comp B) \conv_M (C \comp D)} \ar@{-->}[r] &
    {(A \conv_M C) \comp (B \conv_M D)}
  }
\end{equation}
in $\V$, where the $q$'s to the left and right are the universal
$M$-bilinear maps. A straightforward calculation from the duoidal
axioms shows that the upper composite is also a $M$-bilinear map;
whence there is a unique induced map in $\cat{Bimod}_M$ as displayed,
which we take to be the component of $\xi$ for $\cat{Bimod}_M$. Now
each of the duoidal axioms in $\cat{Bimod}_M$ is either exactly the
corresponding axiom in $\V$, or else a direct consequence of it after
descending along the regular epimorphisms $q$.

\subsection{Normalization}
\label{sec:normalization-1}
If we specialize to the case of the preceding construction where $M$ is the unit object $I$ made
into a bimonoid via the structure maps
$\upsilon \colon \J \rightarrow \I \leftarrow \I \conv \I \colon \mu$
and
$1_\I \colon \I \leftarrow \I \rightarrow \I \comp \I \colon
\lambda_\I$,
we see that the induced duoidal structure on $\cat{Bimod}_\I$ has its
structure cell $\upsilon$ an \emph{identity}; as such, it is normal,
and so we can give:

\begin{Defn}
  \label{def:12}
  If $\V$ is duoidal with reflexive coequalizers preserved
  by $\conv$ in each variable, then its \emph{normalization} is
  the normal duoidal category
  $\mathrm{N}(\V) = (\cat{Bimod}_\I, \conv_\I, \I,
  \comp, \I)$.
\end{Defn}

If $\V$ is $\conv$-biclosed with equalizers, then the
$\conv_M$-monoidal structure on $\cat{Bimod}_M$ is also
biclosed (see~\cite{Barr1996The-Chu-construction} for example); in
particular, this means that the $\conv$-biclosedness of a duoidal
structure is preserved under normalization. On the other hand, the
$\conv$-braidedness need \emph{not} be preserved; for this we need to
modify the construction.

Suppose that $\V$ is a $\conv$-braided monoidal category, and that $M$
is a bimonoid in $\V$ with commutative underlying $\conv$-monoid.
There is a full embedding $\cat{Mod}_M \rightarrow \cat{Bimod}_M$ of
the category of left $M$-modules into the category of $M$-bimodules
which sends $\ell \colon M \conv A \rightarrow A$ to the bimodule
$\ell \colon M \conv A \rightarrow A \leftarrow A \conv M \colon \ell
c$. The image of $\cat{Mod}_M$ under this embedding is closed under both
the tensor product $\conv_M$ and also, by~\eqref{eq:2}, under the
lifted $\comp$-monoidal structure; whence $\cat{Mod}_M$ acquires a
duoidal structure $(\conv_M, \comp)$. Moreover, for any $A$ and $B$ in
$\cat{Mod}_M$, the braidings
$c \colon A \conv B \rightarrow B \conv A$ descend to maps
$c \colon A \conv_M B \rightarrow B \conv_M A$ in $\cat{Mod}_M$ which
witness the duoidal structure as $\conv$-braided. As the terminal
$\conv$-monoid $\I$ in $M$ is commutative, we may  define:
\begin{Defn}
  \label{def:13}
  Let $\V$ be $\conv$-braided duoidal with reflexive
  coequalizers preserved by $\conv$ in each variable. The \emph{braided
    normalization} $\mathrm{N}_c(\V)$ is the $\conv$-braided normal
  duoidal category $(\cat{Mod}_\I, \conv_\I, \I, \comp, \I)$.
\end{Defn}

\subsection{Normalization and commutativity}
\label{sec:norm-comm}
With an eye to our applications, it will be convenient to have a
description of commutativity in $\mathrm{N}(\V)$ and $\mathrm{N}_c(\V)$
in terms of data in $\V$. In fact it suffices to consider only
$\mathrm{N}(\V)$ as $\mathrm{N}_c(\V)$ sits inside this as a full sub-duoidal
category.

\begin{Lemma}
  \label{lem:2}
  Let $\V$ be duoidal with normalization
  $\mathrm{N}(\V)$, and let $C$ be a $\comp$-monoid in
  $\mathrm{N}(\V)$. A cospan
  $f \colon A \rightarrow C \leftarrow B \colon g$ in
  $\mathrm{N}(\V)$ is commuting  just when the diagram
  \begin{equation}\label{eq:8}
    \cd[@!C@-0.5em@R-0.5em]{
      & A \comp B \ar[r]^-{f \comp g} & C \comp C \ar[dr]^-{m} \\
      A \conv B \ar[ur]^-{\bar \sigma} \ar[dr]_-{\bar \tau} & & & C\\
      & B \comp A \ar[r]^-{g \comp f} & C \comp  C \ar[ur]_-{m}
    }
  \end{equation}
  commutes in $\V$, where $\bar \sigma$ and $\bar \tau$ are the
  composites
  \begin{equation}
    A \conv B \cong (A \comp \I) \conv (\I \comp B) \xrightarrow{\xi}
    (A \conv \I) \comp (\I \conv B) \xrightarrow{r \comp \ell} 
    A \comp B
  \end{equation}
  \begin{equation}
    A \conv B \cong (\I \comp A) \conv (B \comp \I) \xrightarrow{\xi}
    (\I \conv B) \comp (A \conv \I) \xrightarrow{\ell \comp r} 
    A \comp B
  \end{equation}
  constructed from interchange in $\V$ and the $\I$-bimodule
  structures on $A$ and $B$.
\end{Lemma}
\begin{proof}
  Since the forgetful $\cat{Bimod}_\I \rightarrow \V$ is
  faithful, to ask that~\eqref{eq:3} commutes in
  $\cat{Bimod}_\I$ is equally to ask that its precomposition
  with the epimorphism
  $A \conv B \twoheadrightarrow A \conv_\I B$ commutes in $\V$;
  which is exactly to ask that~\eqref{eq:8} commutes.
\end{proof}
\begin{Cor}
  \label{cor:1}
  Let $\V$ be duoidal with normalization $\mathrm N(\V)$. The
  forgetful $U \colon \mathrm N(\V) \rightarrow \V$ induces an
  isomorphism between the category of commutative $\comp$-monoids in
  $\mathrm N(\V)$ and the category of duoids in $\V$.
\end{Cor}
\begin{proof}
  By Proposition~\ref{prop:11}, it suffices to show that $U$ induces
  an isomorphism between the categories of duoids in $\mathrm N(\V)$
  and in $\V$. To give a $\conv_\I$-monoid in $\mathrm N(\V)$
  is, by a well-known calculation, to give an $\conv$-monoid
  $C$ in $\V$ together with an $\conv$-monoid morphism
  $\I \rightarrow C$. It follows that $\comp$-monoids
  in $\cat{Mon}_{\conv_\I}(\mathrm{N}(\V))$ may be identified
  via $U$ with $\comp$-monoids in $\cat{Mon}_\conv(\V)$.
\end{proof}

\section{Example: strong and commutative monads}
\label{sec:exampl-comm-monads}
Another way of viewing the algebraic theories of
Section~\ref{sec:exampl-algebr-theor} is as finitary monads on
$\cat{Set}$. All aspects of our framework for commutativity may be
re-expressed in this context in a purely monad-theoretic form; this
then allows them to be generalised further to the context of
\emph{strong monads} on any monoidal category $\V$. In this
monad-theoretic setting, the theory of commutativity is due
to~\cite{Kock1970Monads}; in this section, we show how to reconstruct
it from our general framework.

\subsection{The duoidal category of $\kappa$-accessible endofunctors}
\label{sec:duoid-categ-kappa}
In this section, we suppose that $\V$ is a monoidal biclosed category
which is \emph{locally $\kappa$-presentable as a closed
  category}~\cite{Kelly1982Structures}; this means that the category
$\V$ is locally $\kappa$-presentable, with the $\kappa$-presentable
objects being closed under nullary and binary tensor. We will
establish a framework for commutativity for monads on $\V$ which are
\emph{$\kappa$-accessible}---meaning that their underlying endofunctor
preserves $\kappa$-filtered colimits. This covers most cases of
practical interest; in Section~\ref{sec:commutative-monads}, we will
see another approach which can deal with non-accessible monads on
non-locally presentable categories.

We let $\cat{End}_\kappa(\V)$ denote the category of
$\kappa$-accessible endofunctors of $\V$. If $\V_\kappa$ is a small
skeleton of the subcategory of $\kappa$-presentable objects in $\V$,
then restriction and left Kan extension along the inclusion
$\V_\kappa \rightarrow \V$ establishes an equivalence between
$\cat{End}_\kappa(\V)$ and $[\V_\kappa, \V]$. Since $\V_\kappa$ is
closed under the monoidal structure of $\V$, Day convolution gives a
biclosed monoidal structure on $[\V_\kappa, \V]$ and so, by
transporting across the equivalence, a biclosed monoidal structure
$(\conv, \J)$ on $\cat{End}_\kappa(\V)$. Since $\kappa$-accessible
endofunctors compose, $\cat{End}_\kappa(\V)$ also has its composition
monoidal structure $(\comp, \I)$; now generalising
Proposition~\ref{prop:9}, we have:
\begin{Prop}\label{prop:19}
$(\cat{End}_\kappa(\V), \conv, \J, \comp, \I)$ is a duoidal category.
\end{Prop}
\begin{proof}
  If $F, G, H \in \cat{End}_\kappa(\V)$, then any
  natural family $FA \otimes GB \rightarrow H(A \otimes B)$ is
  uniquely determined by its components at $\kappa$-presentable $A$
  and $B$; it follows that the $\conv$-tensor product on
  $\cat{End}_\kappa(\V)$ is in fact convolution in $[\V, \V]$. We now
  use this fact in constructing the duoidal structure maps.

  First, by the universal property of convolution, $\conv$-monoids in
  $\cat{End}_\kappa(\V)$ correspond to lax monoidal functors; so in
  particular, the strict monoidal identity functor is an $\conv$-monoid
  with structure maps $\upsilon \colon \J \rightarrow \I$ and
  $\mu \colon \I \conv \I \rightarrow \I$. Next, to give
  $\gamma \colon \J \rightarrow \J \comp \J$ is equally well, by the
  universal property of $\J$, to give a map $\I \rightarrow JJI$ in
  $\V$; which we take to be the composite
  $Jj \cdot j \colon \I \rightarrow JI \rightarrow JJI$. It remains to
  give the interchange maps
  $\xi \colon (F \comp G) \conv (H \comp K) \rightarrow (F \conv H)
  \comp (G \conv K)$,
  which are equally well specified by giving natural families of maps
  $FG(A) \otimes HK(B) \rightarrow (F \conv H)(G \conv K)(A \otimes B)$.
  We obtain these as the composites
  \begin{equation}
    FG(A) \otimes HK(B) \xrightarrow{\!\!u_{GA, KB}\!\!} 
    (F \conv H)(GA \otimes KB) \xrightarrow{\!\!(F\ast H)(u_{A,B})\!\!} 
    (F \conv H)(G \conv K)(A \otimes B)\rlap{ .}
  \end{equation}
  The axioms for a duoidal category are now easily verified by
  exploiting the universal property of convolution.
\end{proof}

\subsection{Bistrong endofunctors}
\label{sec:bistr-endof}
When $\V = \cat{Set}$ and $\kappa = \omega$, the duoidal structure
induced by the preceding result on $\cat{End}_\omega(\cat{Set})$
corresponds under the equivalence with $[\mathbb F, \cat{Set}]$ to the
duoidal structure of Proposition~\ref{prop:9}; in particular, it is
\emph{normal}. We may explicitly calculate the isomorphisms
$\I \conv F \rightarrow F$ and $F \conv \I \rightarrow F$ as
corresponding to the natural families
\begin{equation}\label{eq:10}
  \begin{aligned}
    t \colon X \times FY & \rightarrow F(X \times Y) &
    \text{and}\ \bar t \colon FX \times Y & \rightarrow F(X \times Y)  \\
    (x,y) & \mapsto F(\lambda z.\, (x,z))(y) & (x,y) & \mapsto
    F(\lambda z.\, (z,y))(x)
  \end{aligned}
\end{equation}
expressing the fact that any finitary endofunctor of $\cat{Set}$ has a
canonical \emph{strength} and \emph{costrength}~\cite{Kock1970Monads}.
The same argument pertains when $\V = \cat{Set}$ and $\kappa$ is any
infinite regular cardinal, and also in the (uninteresting) cases where
$\V = \cat{2}$ or $\V = \cat{1}$; but for any other $\V$, the map
$\upsilon \colon \J \rightarrow \I$ of the preceding proposition is
not an isomorphism, and so we must apply the normalization 
of Section~\ref{sec:norm-duoid-categ}. In this case, passing to the
category of bimodules over the $\conv$-monoid $\I$ means explicitly
equipping $F \in \cat{End}_\kappa(\V)$ with strength and costrength
maps
\begin{equation}
  t \colon A \otimes FB \rightarrow F(A \otimes B) \qquad \text{and} \qquad \bar t \colon FA \otimes B \rightarrow F(A \otimes B)\rlap{ ,}
\end{equation}
natural in $A$ and $B$ and satisfying evident associativity and unit
axioms. If we call such an $F$ \emph{bistrong}, then the normalization
of the duoidal $\cat{End}_\kappa(\V)$ is given by the category
$\cat{BiStr}_\kappa(\V)$ of $\kappa$-accessible bistrong endofunctors
and bistrength-preserving natural transformations, with the two
monoidal structures being composition $(\comp, \I)$ (where the
composition of strengths is the obvious one) and the quotiented
convolution $(\conv_\I, \I)$, whose value $F \conv_\I G$ is obtained by
identifying in $F \conv G$ the costrength of $F$ and the strength of $G$.

Now a $\comp$-monoid $(T, \eta, \mu)$ in $\cat{BiStr}_\kappa(\V)$ is
a bistrong $\kappa$-accessible monad on $\V$; and by the explicit
description of $\xi$ in $\cat{End}_\kappa(\V)$ given above, together
with Lemma~\ref{lem:2}, we see that a cospan $f \colon M \rightarrow T
\leftarrow N \colon g$ in $\cat{BiStr}_\kappa(\V)$ is commuting just
when the diagram
\begin{equation}
  \cd[@!C@-1em@R-0.5em]{
    & \sh{l}{0.5em}M(X \otimes NY) \ar[r]^-{M t} & MN(X \otimes Y) \ar[r]^-{f g_{X \otimes Y}} & \sh{r}{0.7em}TT(X \otimes Y) \ar[dr]^-{\mu} \\
    MX \otimes NY \ar[ur]^-{\bar t} \ar[dr]_-{t} & & & & T(X \otimes Y)\\
    & \sh{l}{0.5em}N(MX \otimes Y) \ar[r]^-{N\bar t} & NM(X \otimes Y) \ar[r]^-{gf_{X \otimes Y}} & \sh{r}{0.7em}TT(X \otimes Y) \ar[ur]_-{\mu}
  }
\end{equation}
commutes for all $X, Y \in \V$, while $T$ itself is
commutative just when
\begin{equation}
  \cd[@!C@-1em@R-0.5em]{
    & \sh{l}{0.5em}T(X \otimes TY) \ar[r]^-{Tt} & \sh{r}{0.5em}TT(X \otimes Y) \ar[dr]^-{\mu} \\
    TX \otimes TY \ar[ur]^-{\bar t} \ar[dr]_-{t} & & & T(X \otimes Y)\\
    & \sh{l}{0.5em}T(TX \otimes Y) \ar[r]^-{T\bar t} & \sh{r}{0.5em}TT(X \otimes Y) \ar[ur]_-{\mu}
  }
\end{equation}
commutes for all $X, Y \in \V$. We have thus reconstructed the
classical notion of \emph{commutative monad}~\cite{Kock1970Monads} on
a monoidal category $\V$. 

If the monoidal $\V$ is braided, then the duoidal
$\cat{End}_\kappa(\V)$ is $\conv$-braided, and so we have the option of
taking the \emph{braided} normalization of $\cat{End}_\kappa(\V)$.
This is the $\conv$-braided normal duoidal category of strong
$\kappa$-accessible endofunctors of $\V$, which
by~\cite[Section~1]{Kock1970Monads}
and~\cite[Remark~7.7]{Kelly1982Structures} is equally the category
$\V\text-\cat{End}_\kappa(\V)$ of $\kappa$-accessible $\V$-enriched
endofunctors of $\V$. The two tensor products giving the duoidal
structure now admit direct description as composition $(\comp, \I)$ and
\emph{$\V$-enriched} convolution $(\conv_\V, \I)$. A $\comp$-monoid is
thus a $\kappa$-accessible $\V$-monad on $\V$, while the notions of
commuting map and commutative $\V$-monad are exactly as above, where
the strength $t$ and costrength $\bar t$ are derived from the
$\V$-enrichment via the formulae of~\cite[Section~1]{Kock1970Monads}.

\subsection{Commuting tensor product}
\label{sec:comm-tens-prod-1}
Since $\cat{End}_\kappa(\V)$ is equivalent to $[\V_\kappa, \V]$, it is
locally presentable; since $\cat{BiStr}_\kappa(\V)$ is the category of
algebras for a cocontinuous monad on $\cat{End}_\kappa(\V)$ it is thus
also locally presentable, and in particular complete. The duoidal
structure thereon is $\ast$-biclosed, while the $\comp$-tensor product
preserves $\kappa$-filtered colimits in each variable; whence
$\cat{End}_\kappa(\V)\text-\cat{Cat}$ is cocomplete and free
$\cat{End}_\kappa(\V)$-categories exist by Proposition~\ref{prop:22}.
So by Proposition~\ref{prop:17}, $\cat{End}_\kappa(\V)\text-\cat{Cat}$
admits the commuting tensor product $\comm$. In particular, we have a
tensor product of $\kappa$-accessible monads on $\V$; the universal
property of this tensor product was first described
in~\cite[Definition~2]{Hyland2007Combining}.

Note that the restriction to $\kappa$-accessible monads is essential
to ensure existence of the tensor product; for \emph{unbounded}
monads---ones which, like the power-set monad on $\cat{Set}$, are not
$\kappa$-accessible for \emph{any} $\kappa$---it may be that their
commuting tensor product fails to exist entirely;
see~\cite[Theorem~22]{Goncharov2011A-counterexample}.

Finally, let us note the force of Proposition~\ref{prop:11} in the
context of this example. Together with Lemma~\ref{lem:2}, it tells us
that commutative $\comp$-monoids in $\cat{BiStr}_\kappa(\V)$ are the
same thing as duoids in $\cat{End}_\kappa(\V)$; while in the
$\conv$-braided case, it tells us that commutative $\comp$-monoids in
$\V\text-\cat{End}_\kappa(\V)$ are the same as $\conv$-commutative
monoids in $\cat{End}_\kappa(\V)$. This latter case reconstructs the
main Theorem~3.2 of~\cite{Kock1970Monads}: that a strong
$\kappa$-accessible monad on a braided monoidal category is
commutative just when it is a monoidal monad, which is then
automatically a braided monoidal monad. The non-braided case tells us
similarly that a bistrong $\kappa$-accessible monad on a monoidal
category is commutative just when it is a monoidal monad.

\section{Example: Freyd-categories}
\label{sec:bistrong-promonads}
Strong and commutative monads play a prominent role in computer
science, in particular in work on computational side-effects growing
out of Moggi's~\cite{Moggi1991Notions}. In this context, a useful
generalisation of strong monads are the \emph{arrows}
of~\cite{Hughes2000Generalising}, which find categorical expression as
the \emph{Freyd}-categories of~\cite{Levy2003Modelling}. The notion of
\emph{Freyd}-category is a slightly delicate one; the purpose of this
section is to show how it fits naturally into our framework.

\subsection{Duoidal categories from monoidales}
\label{sec:duoid-categ-from}
In order to define the duoidal categories giving rise to
Freyd-categories, we will appeal to the following general
construction. Let $(\M, \otimes, \I)$ be a monoidal
bicategory~\cite{Gordon1995Coherence}, and $(A, j, p)$ a monoidale (=
pseudomonoid) in $\M$ whose unit $j \colon \I \rightarrow A$ and
multiplication $p \colon A \otimes A \rightarrow A$ have right
adjoints $j^\conv \colon A \rightarrow \I$ and
$p^\conv \colon A \rightarrow A \otimes A$; since left adjoint
morphisms in a bicategory are often called \emph{maps}, we call $A$ a
\emph{map monoidale} in $\M$.

In this situation, the category $\M(A,A)$ bears two monoidal
structures; the first $(\comp, \I)$ is the composition of the
bicategory $\M$, while the second $(\conv, \J)$ is
convolution with respect to the monoidale $(A, j, p)$ and
comonoidale $(A, j^{\conv}, p^{\conv})$. The unit $\J$
is thus the composite
$jj^\conv \colon A \rightarrow \I \rightarrow A$, the tensor
$f \conv g$ of $f, g \colon A \rightarrow A$ is the composite
\begin{equation}
  A \xrightarrow{p^\conv} A \otimes A \xrightarrow{f \otimes g} A \otimes A \xrightarrow{p} A\rlap{ ,}
\end{equation}
and the coherence constraints are derived from the constraint
$2$-cells of the monoidale $A$ and their mates;
see~\cite[Proposition~4]{Day1997Monoidal}.

\begin{Prop}
  \label{prop:20}
  If $A$ is a map monoidale in the monoidal bicategory $\M$, then
  $(\M(A,A), \conv, \J, \comp, \I)$ is a duoidal category. Moreover,
  if $\M$ is a braided monoidal
  bicategory~\cite[Appendix~A]{McCrudden2000Balanced}, and $A$ a
  braided monoidale~\cite[Section~3]{McCrudden2000Balanced}, then this
  duoidal structure is $\conv$-braided; and if $\M$ is a biclosed
  monoidal bicategory which admits all right liftings and extensions,
  then the duoidal structure is $\conv$-biclosed and $\comp$-biclosed.
\end{Prop}
In the final clause, to say that $\M$ is \emph{biclosed} is to say
that the homomorphisms $A \otimes (\thg)$ and
$(\thg) \otimes A \colon \M \rightarrow \M$ have right biadjoints
$[A, \thg]_\ell$ and $[A, \thg]_r$; while to ask for \emph{right
  liftings} and \emph{right extensions} is to ask that, for each
object $C$ and $1$-cell $f \colon A \rightarrow B$ in $\M$, the
functors $f \comp (\thg) \colon \M(C,A) \rightarrow M(C,B)$ and
$(\thg) \comp f \colon M(B,C) \rightarrow \M(A,C)$ have right
adjoints.
\begin{proof}
  The constraint cells for the duoidal structure on $\M(A,A)$ are
  obtained from the constraints of the monoidal bicategory $\M$ and
  the units and counits of the adjunctions $j \dashv j^{\conv}$ and
  $p \dashv p^{\conv}$:
\begin{gather}
  \mu \colon p(1 \otimes 1)p^\conv \xrightarrow{\cong} pp^\conv \xrightarrow{\epsilon} 1 \qquad 
  \upsilon \colon jj^\conv \xrightarrow{\epsilon} 1\qquad 
  \gamma \colon jj^\conv \xrightarrow{j\eta j^\conv} j{j^\conv}j{j^\conv}\\
\xi \colon p(fh \otimes gk)p^\conv \xrightarrow{\cong} p(f \otimes g)(h \otimes k)p^\conv \xrightarrow{p(f\otimes g)\eta(h \otimes k)p^\conv} p(f \otimes g)p^{\conv}p(h \otimes k)p^\ast\rlap{ .}
\end{gather}
The duoidal axioms are straightforward consequences of the monoidal bicategory
axioms and the triangle identities for an adjunction.

When $\M$ is braided and $A$ is a braided monoidale, the convolution
monoidal structure $\conv$ on $\M(A,A)$ itself acquires a braiding,
and it is now easy to verify the coherence making $\M(A,A)$ into a
$\conv$-braided duoidal category. Finally, suppose that $\M$ is
biclosed with all right liftings and extensions. This immediately
implies $\comp$-biclosedness of $\M(A,A)$. To show $\conv$-biclosedness,
we first claim that the functors
\begin{equation}
  \otimes \colon \M(A,B) \times \M(C,D) \rightarrow \M(A \otimes C, B \otimes D)
\end{equation}
have right adjoints in each variable. To see this, suppose given
$f \in \M(A,B)$; then for each $C, D \in \M$ we have the composite
functor
\begin{equation}
\M(C,D) \xrightarrow{f \otimes (\thg)} \M(A \otimes C, B \otimes D) \xrightarrow{\overline{(\thg)}} \M(C, [A, B \otimes D])\rlap{ .}
\end{equation}
Since these functors are clearly pseudonatural in $C$, they must be
given to within equivalence by composition with a $1$-cell
$D \rightarrow [A, B \otimes D]$, and so---using the right
liftings---must have right adjoints. Since $\overline{(\thg)}$ is an
equivalence, we conclude that each $f \otimes (\thg)$ has a right
adjoint $[f, \thg]_r$, and dually each $(\thg) \otimes g$ has a right
adjoint $[g, \thg]_\ell$. Under these assumptions, it follows that the
duoidal structure on $\M(A,A)$ is $\conv$-biclosed: given
$g, h \colon A \rightarrow A$, the two internal homs are given by
$[g, p^\conv hp]_\ell \colon A \rightarrow A$ and
$[g, p^{\conv}hp]_r \colon A \rightarrow A$ respectively.
\end{proof}

\subsection{Bistrong profunctors}
\label{sec:bistrong-profunctors}
We now apply the construction of the previous section to the
bicategory $\cat{Prof}$ of \emph{profunctors}, whose objects are small
categories and whose $1$- and $2$-cells $\A \tor \B$ are cocontinuous
functors and transformations
$[\A^\op, \cat{Set}] \rightarrow [\B^\op, \cat{Set}]$. In practice, we
prefer the equivalent formulation which views these $1$- and $2$-cells
as functors and transformations
$\B^\op \times \A \rightarrow \cat{Set}$, with composition given by
coend. $\cat{Prof}$ is in fact a monoidal bicategory, with
$\A \otimes \B \defeq \A \times \B$ and with
$(N \colon \A \tor \B) \otimes (M \colon \C \tor \D)$ defined by
$(N \otimes M)\big((b,d),(a,c)\big) = N(b,a) \times M(d,c)$.

There is an identity-on-objects strong monoidal homomorphism
$(\thg)_\conv \colon \cat{Cat}\rightarrow \cat{Prof}$ which sends
$F \colon \A \rightarrow \B$ to the profunctor
$F_\conv \colon \A \tor \B$ with $F_\conv(b,a) = \B(b, Fa)$. Note that
$F_\conv$ has a right adjoint $F^\conv \colon \B \tor \A$ with
$F^\conv(a,b) = \B(Fa,b)$; it follows that each small monoidal
category $\A$ (= monoidale in $\cat{Cat}$) is sent by $(\thg)_\conv$
to a map monoidale in $\cat{Prof}$. Applying the construction of the
preceding section, we conclude that if $(\A, \otimes, i)$ is a small
monoidal category, then the functor category
$\cat{Prof}(\A, \A) = [\A^\op \times \A, \cat{Set}]$ has a duoidal
structure $(\conv, \J,\comp, \I)$ whose tensor products have the
following characterisations:
\begin{itemize}
\item Maps $M \comp N \rightarrow L$ classify families
  $M(c,b) \times N(b,a) \rightarrow L(c,a)$ that are bilinear in the
  sense of Definition~\ref{def:5};
\item Maps $\I \rightarrow M$ classify extranatural families
  $I \rightarrow M(u,u)$;
\item Maps $M \conv N \rightarrow L$ classify natural families
  $M(a,b) \times N(c,d) \rightarrow L(a \otimes c, b \otimes d)$;
\item Maps $\J \rightarrow M$ classify morphisms
  $I \rightarrow M(i,i)$.
\end{itemize}
The exchange maps $\xi$ of this duoidal structure are given as
follows. For $M,N,P,Q \in \cat{Prof}(\A, \A)$, the
universal natural families for $M \conv P$ and $N \conv Q$ yield
an arrow
\begin{equation}
  M(a,b) \times N(b,c) \times P(d,e) \times Q(e,f) \rightarrow (M \conv P)(a \otimes d, b \otimes e) \times (N \conv Q)(b \otimes e, c \otimes f)\rlap{ ;}
\end{equation}
now postcomposing with the universal dinatural family for
$(M \conv P) \comp (N \conv Q)$ yields
\begin{equation}
  M(a,b) \times N(b,c) \times P(d,e) \times Q(e,f) \rightarrow (M \conv P) \comp (N \conv Q)(a \otimes d, c \otimes f)\rlap{ ,}
\end{equation}
a family natural in $a,c,d$ and $f$ and bilinear in $b$ and $e$.
Using the classifying property of $\comp$, this corresponds to a
family
\begin{equation}
  (M \comp N)(a,c) \times (P\comp Q)(d,f) \rightarrow (M \conv P) \comp (N \conv Q)(a \otimes d, c \otimes f)
\end{equation}
natural in all variables; and so, by the classifying property of
$\conv$, to the required morphism $(M
\comp N) \conv (P \comp Q) \rightarrow (M \conv P) \comp (N \conv Q)$.
The other duoidal constraint maps are derived in a similar fashion.

Unless $\A \simeq 1$, the duoidal structure on $\cat{Prof}(\A, \A)$
will not be normal, and so we must apply the normalization of
Section~\ref{sec:norm-duoid-categ}; we now describe the effect this
has. To equip $M \in \cat{Prof}(\A, \A)$ with a left action
$\I \conv M \rightarrow M$ is to give a natural family of maps
$\I(a,b) \times M(c,d) \rightarrow M(a \otimes c, b \otimes d)$, or
equally, by the classifying property of $\I$, a family of maps
\begin{equation}
t_{ucd} \colon M(c,d) \rightarrow M(u \otimes c, u \otimes d)\label{eq:11}
\end{equation}
natural in $c$ and $d$ and dinatural in $u$, and satisfying the
associativity and unit axioms:
\begin{equation}
  \cd[@C-3em]{
    {M(c,d)} \ar[d]_-{t} \ar[rr]^-{t} & &
    {M((u \otimes v) \otimes c, (u \otimes v) \otimes d)} \ar[d]^-{M(\alpha^{-1}, \alpha)} & {\qquad} &
    M(c,d) \ar@{=}[rr] \ar[dr]_-{t} & & M(c,d)\rlap{ .} \\
    {M(v \otimes c, v \otimes d)} \ar[rr]^-{t} & {\qquad}&
    {M(u \otimes (v \otimes c), u \otimes (v \otimes d))} & & &
    {M(i \otimes c, i \otimes d)} \ar[ur]_-{M(\lambda^{-1}, \lambda)}
  }  
\end{equation}
Similarly, a right action $M \conv \I \rightarrow M$ involves a family
of maps
\begin{equation}
  \label{eq:12}
  \bar t_{cdv} \colon M(c,d) \rightarrow M(c \otimes v, d \otimes v)
\end{equation}
satisfying dual associativity and unit axioms; the further axiom
required for~\eqref{eq:11} and~\eqref{eq:12} to comprise an $\I$-bimodule
structure on $M$ is that 
\begin{equation}
  \cd[@R-1.5em]{
    & M(u \otimes c, u \otimes d) \ar[r]^-{\bar t} & 
    M((u \otimes c) \otimes v, (u \otimes d) \otimes v) \ar[dd]^{M(\alpha^{-1}, \alpha)}\\
    M(c,d) \ar[ur]^-{t} \ar[dr]_-{\bar t} \\ &
    M(c \otimes v, d \otimes v) \ar[r]^-{t} & 
    M(u \otimes (c \otimes v), u \otimes (d \otimes v))
  }
\end{equation}
should commute. Note that when $M = F_\conv$ for some $\V$-functor
$F$, giving an $\I$-bimodule structure reduces by the Yoneda lemma to
giving a compatible strength and costrength on $F$, in the sense of
Section~\ref{sec:bistr-endof} above; it thus seems reasonable for a
general $M$ to call this structure a \emph{bistrong profunctor} (note
that in~\cite{Pastro2008Doubles}, the name \emph{Tambara module} was
used for the same structure).

The normalization of the duoidal $\cat{Prof}(\A,\A)$ is thus the
normal duoidal category $\cat{BiStrProf}(\A, \A)$ of bistrong profunctors and
bistrength-preserving transformations, equipped with the two monoidal
structures $(\conv_\I, \I)$ given by the $\I$-bilinear quotient of the
convolution $\conv$, and $(\comp, \I)$ given by composition of
profunctors with the induced bistrength. Note that, by the 
observations of the preceding paragraph, the embedding
$(\thg)_\ast \colon \cat{End}(\A) \rightarrow \cat{Prof}(\A,\A)$ lifts
to an embedding
\begin{equation}
  \label{eq:40}
  (\thg)_\ast \colon \cat{BiStr}(\A,\A) \rightarrow\cat{BiStrProf}(\A,\A)
\end{equation}
which is strong monoidal for the respective composition
monoidal structures.

\subsection{Bistrong promonads}
\label{sec:prem-freyd-categ}
A $\comp$-monoid in the normal duoidal $\cat{BiStrProf}(\A,\A)$ is a
$\comp$-monoid in $\cat{Prof}(\A,\A)$---thus a monad on $\A$ in
$\cat{Prof}$---equipped with a bistrength compatible with the monad
structure. We call this structure a \emph{bistrong promonad}. Note
that, by strong monoidality of~\eqref{eq:40}, any bistrong monad on
$\A$ can be regarded as a bistrong promonad whose underlying
profunctor is representable; in the next section, we consider how our
commutativity notions specialise to this case.

As is well-known, a monad on $\A$ in $\cat{Prof}$ is the same thing
as an identity-on-objects functor $F \colon \A \rightarrow \M$: given
$F$, the corresponding monad is $F^\conv F_\conv \colon \A \tor \A$;
while given a promonad $M \colon \A \tor \A$, the corresponding
$\M$ has the same objects as $\A$ and hom-sets $\M(a,b) = M(a,b)$,
with the monad multiplication giving composition in $\M$ and the monad
unit giving the action of $F \colon \A \rightarrow \M$ on homs. In
these terms, giving a bistrength for $M$
amounts to giving natural families
$\M(a,b) \rightarrow \M(u \otimes a, u\otimes b)$ and
$\M(a,b) \rightarrow \M(a \otimes v, b \otimes v)$. Compatibility with
the monad structure of $M$ says that these functions provide the
actions on homs of functors $u \otimes (\thg)$ and
$(\thg) \otimes v \colon \M \rightarrow \M$ fitting into commuting
diagrams as on the left and right in
\begin{equation}\label{eq:13}
  \cd{
    {\A} \ar[r]^-{u \otimes (\thg)} \ar[d]_{F} &
    {\A} \ar[d]^{F} \\
    {\M} \ar[r]_-{u \otimes (\thg)} &
    {\M}
  } \qquad \qquad 
  \cd{
    {\A} \ar[r]^-{(\thg) \otimes v} \ar[d]_{F} &
    {\A} \ar[d]^{F} \\
    {\M} \ar[r]_-{(\thg) \otimes v} &
    {\M}\rlap{ .}
  }
\end{equation}
The unit axioms for the strength and costrength express that the maps
$F\lambda_a \colon i \otimes a \rightarrow a$ and
$F\rho_a \colon a \otimes i \rightarrow a$ are components of natural
transformations
$i \otimes (\thg) \Rightarrow 1$ and
$(\thg) \otimes i \Rightarrow 1 \colon \M \rightarrow \M$; the
associativity and bimodule axioms express that the maps
$F\alpha_{abc} \colon (a \otimes b) \otimes c \rightarrow a \otimes (b
\otimes c)$
are components of natural transformations
$(a \otimes b) \otimes \thg \Rightarrow a \otimes (b \otimes \thg)$
and
$(a \otimes \thg) \otimes b \Rightarrow a \otimes (\thg \otimes c)$
and
$(\thg \otimes b) \otimes c \Rightarrow \thg \otimes (b \otimes c)
\colon \M \rightarrow \M$.
The naturality of the strength and costrength maps in $a$ and $b$ is
automatic from the structure defined so far, while the extranaturality of
$t$ in $u$ and of $\bar t$ in $v$ express that, for each
$f \colon u \rightarrow v$ in $\A$ and $g \colon c \rightarrow d$ in
$\M$, the following diagrams commute:
\begin{equation}\label{eq:14}
  \cd{
    {u \otimes c} \ar[r]^-{u \otimes g} \ar[d]_{Ff \otimes c} &
    {u \otimes d} \ar[d]^{Ff \otimes d} \\
    {v \otimes c} \ar[r]_-{v \otimes g} &
    {v \otimes d}
  } \qquad \text{and} \qquad 
  \cd{
    {c \otimes u} \ar[r]^-{g \otimes u} \ar[d]_{c \otimes Ff} &
    {d \otimes u} \ar[d]^{d \otimes Ff} \\
    {c \otimes v} \ar[r]_-{g \otimes v} &
    {d \otimes v}\rlap{ .}
  }
\end{equation}

In the case where the monoidal structure on $\A$ is cartesian, and the
symmetry isomorphisms descend to $\M$ in an obvious sense, the above
structure was shown in~\cite[Theorem~6.1]{Jacobs2009Categorical} to be
equivalent to that of a \emph{Freyd}-category in the sense
of~\cite[Definition~4.1]{Levy2003Modelling}. We now explain how to
generalise this equivalence to the non-cartesian situation; we begin
by recalling the necessary definitions.

\begin{Defn}
  \label{def:4}
  A \emph{premonoidal} structure on a category $\M$ is given by the
  following data: (i) a unit object $i \in \M$; (ii) for each
  $u \in \M$ a functor $u \otimes (\thg) \colon \M \rightarrow \M$ and
  for each $v \in \M$ a functor
  $(\thg) \otimes v \colon \M \rightarrow \M$ such that the
  assignation on objects $u,v \mapsto u \otimes v$ on objects is
  unambiguously defined; and (iii), families of maps
  \begin{equation}
    \lambda_a \colon i \otimes a \rightarrow a \qquad 
    \rho_a \colon a \otimes i \rightarrow a \qquad 
    \alpha_{abc} \colon (a \otimes b) \otimes c \rightarrow a \otimes (b \otimes c)
  \end{equation}
  where the $\lambda$'s and $\rho$'s are natural, and the $\alpha$'s
  are natural in each variable separately. The data in (iii) must
  satisfy the usual triangle and pentagon axioms for a monoidal
  category; moreover, we require each map $\lambda_a$, $\rho_a$ and
  $\alpha_{abc}$ to be central. Here, a map
  $f \colon u \rightarrow v$ in $\M$ is called \emph{central} if, for
  every $g \colon c \rightarrow d$ in $\M$, the two squares
\begin{equation}
  \cd{
    {u \otimes c} \ar[r]^-{u \otimes g} \ar[d]_{f \otimes c} &
    {u \otimes d} \ar[d]^{f \otimes d} \\
    {v \otimes c} \ar[r]_-{v \otimes g} &
    {v \otimes d}
  } \qquad \text{and} \qquad 
  \cd{
    {c \otimes u} \ar[r]^-{g \otimes u} \ar[d]_{c \otimes f} &
    {d \otimes u} \ar[d]^{d \otimes f} \\
    {c \otimes v} \ar[r]_-{g \otimes v} &
    {d \otimes v}
  }
\end{equation}
commute. Note that monoidal categories are the same thing as
premonoidal categories in which every map is central.
If $\M$ and $\N$ are premonoidal categories, then a \emph{strict
  premonoidal functor} $\M \rightarrow \N$ is one which preserves the
unit, tensors and constraint morphisms on the nose, and which moreover
sends central maps to central maps.
\end{Defn}
\begin{Defn}
  \label{def:8}
  A generalised \emph{Freyd}-category is a triple $(\A, \M, F)$ where $\A$ is a
  small monoidal category, $\M$ is a
  small premonoidal category, and $F \colon \A \rightarrow \M$ is a bijective-on-objects strict
  premonoidal functor.
\end{Defn}
As indicated above, this generalises the definition of
\emph{Freyd}-category in~\cite{Levy2003Modelling} in two ways: we do
not require $\A$ be \emph{cartesian} monoidal, and do not assume that
$\M$ and $F$ are \emph{symmetric} premonoidal in the obvious sense.
The following result correspondingly
generalises~\cite[Theorem~6.1]{Jacobs2009Categorical}.

\begin{Prop}
  \label{prop:7}
  Let $\A$ be a small monoidal category. To give a bistrong promonad
  on $\A$ is equally to give a
  generalised \textit{Freyd}-category of the form $(\A, \M, F)$.
\end{Prop}
\begin{proof}
  Let $M$ be a promonad on $\A$, and let
  $F \colon \A \rightarrow \M$ be the corresponding
  identity-on-objects functor; we must show that a compatible
  bistrength is precisely what is needed to make $F$ into a generalised
  \emph{Freyd}-category. The premonoidal structure on $\M$ will have
  unit object $i$, with the data of~\eqref{eq:13} providing the
  functors $u \otimes (\thg)$ and $(\thg) \otimes v$ in $\M$, and the
  unit and associativity axioms for the strength and costrength
  yielding the required natural families of constraint maps
  $F\lambda_a$, $F\rho_a$ and $F\alpha_{abc}$. Now commutativity
  in~\eqref{eq:14} is the assertion that $F$ sends central maps (= all
  maps) in $\A$ to central maps in $\M$. In particular, each
  $F\lambda_a$, $F\rho_a$ and $F\alpha_{abc}$ is central in $\M$,
  whence $\M$ is premonoidal; it is now clear from the definitions
  that $F \colon \A \rightarrow \M$ is strict premonoidal.
\end{proof}

Note that every premonoidal $\M$ arises in this way for a suitable
$\A$; for example, we may take $\A$ to be the \emph{centre} $\Z(\M)$
of $\M$, comprising the same objects but only the central maps. The
centre of a premonoidal category is a monoidal category, and the
identity-on-objects inclusion $\Z(\M) \rightarrow \M$ is a strict
premonoidal functor.

We now interpret our basic commutativity notions in light of this
proposition. Let $M$ be a bistrong promonad, corresponding to the
generalised \emph{Freyd}-category $F \colon \A \rightarrow \M$, and let
$h \colon H \rightarrow M \leftarrow K \colon k$ be a cospan in
$\cat{BiStrProf}(\A,\A)$. By Lemma~\ref{lem:2} and the explicit
description of the tensor products in $\cat{Prof}(\A,\A)$, this cospan
is commuting just when the diagram
\begin{equation}
  \cd[@!C@C-6em@R-0.5em]{
    & \sh{l}{2.5em}\M(a,b) \times \M(c,d) \ar[r]^-{(\thg \otimes c) \times (b \otimes \thg)} & 
    \sh{r}{2.7em} \M(a \otimes c,b \otimes c) \times \M(b \otimes c, b \otimes d) \ar[dr]^-{\comp} \\
    \sh{r}{3em}H(a,b) \times K(c,d) \ar[ur]^-{h \times k} \ar[dr]_-{h \times k} & & &  \sh{l}{3em}\M(a \otimes c, b \otimes d)\\
    & \sh{l}{2.5em}\M(a,b) \times \M(c,d) \ar[r]^-{(a \otimes \thg) \times (\thg \otimes d)} & 
    \sh{r}{2.7em}\M(a \otimes c, a \otimes d) \times \M(a \otimes d, b \otimes d) \ar[ur]_-{\comp}
  }
\end{equation}
commutes for all $a,b,c,d \in \A$. This is equally to ask that, for
each $x \in H(a,b)$ and each $y \in K(c,d)$, the square
\begin{equation}
  \label{eq:42}
  \cd{
    a \otimes c \ar[r]^-{h(x) \otimes c} \ar[d]_-{a \otimes k(y)} & b \otimes c \ar[d]^-{b \otimes k(y)} \\
    a \otimes d \ar[r]_-{h(x) \otimes d} & b \otimes d
  }
\end{equation}
should commute in $\M$. In particular, $M$ is commutative just when
$F \colon \A \rightarrow \M$ is a strict monoidal functor between
monoidal categories. Moreover, the generalised \emph{Freyd}-category
$\A \rightarrow \Z(\M)$ given by factorising $F$ through the
centre of $\A$ is precisely the intersection of the left
and right centres (in the sense of Definition~\ref{def:19}) of the
$\comp$-monoid $M$.


Finally, we observe that $\cat{Prof}$ is a symmetric monoidal
bicategory; thus, by Proposition~\ref{prop:20}, if $\A$ is a symmetric
monoidal category, then the resultant duoidal structure on
$\cat{Prof}(\A, \A)$ is naturally $\conv$-braided. We thus have the
option of taking its braided normalization, the category of
\emph{strong profunctors} $\cat{StrProf}(\A, \A)$. In this
circumstance, to give a $\comp$-monoid is to give a
generalised \emph{Freyd}-category which is symmetric in the obvious sense.

\subsection{Commutative monads}
\label{sec:commutative-monads}
As observed in Section~\ref{sec:prem-freyd-categ} above, the category
of bistrong monad on a small monoidal category $\A$ can be identified
with the full subcategory of bistrong promonads on $\A$ whose
underlying endofunctor is representable. It is thus natural to ask
when the bistrong promonad $M = T_\ast$ induced by a bistrong monad
$T$ is commutative. In this circumstance, the profunctor
$M \colon\A\tor\A$ is given by $M(a,b)=\A(a,Tb)$, and so the
generalised $\emph{Freyd}$-category corresponding to $M$ under
Proposition~\ref{prop:7} is of the form $(\A,\A_T,F)$, where
$F\colon\A\to\A_T$ is the free functor into the Kleisli category of
the monad $T$. The commutativity of a cospan of profunctors
$h\colon H\to M\leftarrow K \colon k$, described in the previous
section as the commutativity of~\eqref{eq:42}, now becomes the
commutativity of the diagram:
\begin{equation}
  \label{eq:41}
  \xymatrixcolsep{1.2cm}
  \xymatrix{
    a\otimes c\ar[r]^-{h(x)\otimes c} \ar[d]_{a\otimes k(y)}&
    Tb\otimes c\ar[r]^{\bar t_{bc}}&
    T(b\otimes c)\ar[r]^-{T(b\otimes k(y))}&
    T(b\otimes Td)\ar[d]^-{Tt_{bd}}\\
    a\otimes Td\ar[d]_{t_{ad}}&&&
    T^2(b\otimes d)\ar[d]^{\mu}\\
    T(a\otimes d)\ar[r]^-{T(h(x)\otimes d)}&
    T(Tb\otimes d)\ar[r]^-{T\bar t_{ad}}&
    T^2(b\otimes d)\ar[r]^{\mu}&
    T(b\otimes d)
  }
\end{equation}
for all $x \in H(a,b)$ and $y \in K(c,d)$. When $H=L_\ast$ and
$K=R_\ast$ for endofunctors $L,R$ of $\A$, the morphisms $h$ and $k$
are necessarily induced by natural transformations
$f\colon L\Rightarrow T\Leftarrow R\colon g$, and an application of
Yoneda lemma shows that the commutativity of the cospan $(h,k)$ in
this circumstance reduces to the commutativity of the following
diagrams in $\A$:
\begin{equation}
  \label{eq:43}
  \xymatrixcolsep{1.2cm}
  \xymatrix{
    Lb\otimes Rd\ar[r]^{f_b\otimes Rd} \ar[d]_{Lb\otimes g_d}&
    Tb\otimes Rd\ar[r]^{\bar t_{bRd}}&
    T(b\otimes Rd)\ar[r]^-{T(b\otimes g_d)}&
    T(b\otimes Td)\ar[d]^{Tt_{bd}}\\
    Lb\otimes Td\ar[d]_{t_{Lbd}}&&&
    T^2(b\otimes d)\ar[d]^{\mu}\\
    T(Lb\otimes d)\ar[r]^-{T(f_b\otimes d)}&
    T(Tb\otimes d)\ar[r]^-{T\bar t_{bd}}&
    T^2(b\otimes d)\ar[r]^{\mu}&
    T(b\otimes d)\rlap{ .}
  }
\end{equation}
We have thus proved:
\begin{Prop}\label{prop:23}
  If $T$ is a bistrong monad on $\A$, then a cospan
  $L\Rightarrow T\Leftarrow R$ in $\cat{End}(\A)$ commutes in the
  sense of~\cite{Hyland2007Combining} precisely when the induced
  cospan $L_\ast \Rightarrow T_\ast \Leftarrow R_\ast$ commutes in the
  normal duoidal $\cat{BiStrProf}(\A,\A)$. A bistrong monad $T$ on
  $\A$ is commutative in the sense of~\cite{Kock1970Monads} if and
  only it is so as a bistrong promonad $T_\ast$ on
  $\A$.
\end{Prop}

As in Section~\ref{sec:comm-tens-prod-1}, we may reconstruct from this
the classical result that a bistrong monad $T$ is commutative precisely
when it is a monoidal monad. Indeed, a duoid structure on an object
$T_\ast \in \cat{Prof}(\A, \A)$ is easily seen to be the same as a
monoidal monad structure on $T$; on the other hand, a commutative
monoid structure on $T_\ast$ in the normalization
$\cat{BiStProf}(\A,\A)$ is,  by the proposition
above, a commutative monad on $\A$.

Finally, in this section, we observe that $\cat{Prof}$ is biclosed
(even autonomous) as a monoidal bicategory, and admits all right
extensions and right liftings; so by Proposition~\ref{prop:20}, the
duoidal structure on $\cat{Prof}(\A, \A)$ is $\conv$- and
$\comp$-biclosed. It follows that $\cat{BiStProf}(\A, \A)$ is also
$\conv$ and $\comp$-biclosed; moreover, being the category of algebras
for a cocontinuous monad on
$\cat{Prof}(\A, \A) = [\A^\op \times \A, \cat{Set}]$, it is locally
presentable (indeed, a presheaf category). Applying
Proposition~\ref{prop:22} now shows that the hypotheses of
Proposition~\ref{prop:17} are satisfied by $\cat{BiStProf}(\A, \A)$,
so that we have a good notion of commuting tensor product for bistrong
promonads, or equally, for generalised \emph{Freyd}-categories on
$\A$. In the functional programming literature, strong monads encode
notions of computational side-effect, and an important role is played
by the tensorial combination of two such monads.
\emph{Freyd}-categories were introduced as a generalisation of strong
monads (as implied by Proposition~\ref{prop:23}) and so this tensor
product of generalised \emph{Freyd}-categories is both extremely
natural and of potential interest to computer scientists.

\section{Example: sesquicategories}
\label{sec:exampl-sesq}

Our final example deals with the \emph{sesquicategories}
of~\cite{Street1996Categorical}, which, as explained in the
introduction, are ``$2$-categories without middle-four interchange'';
we will exhibit sesquicategories with fixed underlying category as
$\circ$-monoids in a suitable duoidal category, which is obtained by
normalizing the duoidal category of \emph{derivation schemes}
defined in~\cite[Section~6]{Batanin2012Centers}.

\subsection{Derivation schemes}
\label{sec:derivation-schemes}
Suppose that $(\E, \otimes, I)$ is a monoidal category with pullbacks,
for which the functor $\otimes \colon \E \times \E \rightarrow \E$
preserves pullbacks. In this circumstance, the bicategory
$\cat{Span}(\E)$ of spans in $\E$ becomes a monoidal bicategory, whose
tensor product is that of $\E$ on objects, and on $1$-cells is given
by pointwise tensor product of spans; moreover, the
identity-on-objects inclusion
$(\thg)_\ast \colon \E \rightarrow \cat{Span}(\E)$ sending
$f \colon X \rightarrow Y$ to
$1 \colon X \leftarrow X \rightarrow Y \colon f$ then becomes a strong
monoidal homomorphism. Each morphism $f_\ast$ has a right adjoint
$f^\ast = (f \colon Y \leftarrow X \rightarrow X \colon 1)$ in
$\cat{Span}(\E)$, and so each monoid $(A, j, p)$ in $\E$ is sent by
$(\thg)_\ast$ to a map monoidale in $\cat{Span}(\E)$. It follows by
the construction of Section~\ref{sec:duoid-categ-from} that, for each
monoid $(A,j,p)$ in $\E$, the category
$\cat{Span}(\E)(A, A) = \E / A \times A$ acquires a duoidal structure.

Now fix a set $X_0$; we apply the preceding general considerations to
the monoidal category $\E = \cat{Set}/X_0 \times X_0$ with
(pullback-preserving) tensor product given by composition of spans. To
give a monoid in $\E$ is to give a small category
$\mathbb X = X_1 \rightrightarrows X_0$, and the above construction
now derives from this a duoidal structure on the category
$\cat{Span}(\cat{Set}/X_0 \times X_0)(X_1, X_1) = \cat{Set}/ X_1
\times_{X_0 \times X_0} X_1$.
In the terminology of~\cite{Batanin2012Centers}, this duoidal category
would be called $\cat{Sp}_2(\mathbb X, \cat{Set})$; however, we
follow~\cite{Street1996Categorical,Street2012Monoidal} in referring to
it as the duoidal category $\cat{Ds}(\mathbb X)$ of \emph{derivation
  schemes} on $\mathbb X$. 

Explicitly, an object of $A \in \cat{Ds}(\mathbb X)$ amounts to a
function assigning to each parallel pair of morphisms
$f, g \colon x \rightrightarrows y$ in $\mathbb X$ a set of
``$2$-cells'' $A_{f,g}$, while a morphism
$\alpha \colon A \rightarrow B$ in $\cat{Ds}(\mathbb X)$ is a
collection of functions $\alpha_{f,g} \colon A_{f,g} \to B_{f,g}$. As
for the duoidal structure $(\ast, \J, \comp, \I)$ on
$\cat{Ds}(\mathbb X)$, the unit objects are given by
 \begin{equation}
   \label{eq:39}
   \J_{f,g} =
   \begin{cases}
     1 & \text{ when $f = g = 1_a$;} \\ \emptyset & \text{otherwise,}
   \end{cases}
   \qquad 
   \I_{f,g} =
   \begin{cases}
     1 & \text{ when $f = g$;} \\ \emptyset & \text{otherwise,}
   \end{cases}
 \end{equation}
 while the binary tensors $\comp$ and $\ast$ are characterised by the
 properties that:
 \begin{itemize}
 \item Maps $A \ast B \rightarrow C$ classify
   families of functions
   $A_{f_1, g_1} \times B_{f_2, g_2} \rightarrow C_{f_1f_2, g_1g_2}$
   indexed by all $f_1, f_2 \colon x \rightarrow y$ and
   $g_1, g_2 \colon y \rightarrow z$ in $\mathbb X$; and
\vskip0.25\baselineskip
 \item Maps $A \circ B \rightarrow C$ classify
   families of functions
   $A_{g, h} \times B_{f, g} \rightarrow C_{f,h}$
   indexed by all $f, g, h \colon x \rightarrow y$ in $\mathbb X$.
 \end{itemize}
 It follows that $\conv$-monoid and $\comp$-monoid structures on a
 derivation scheme $A$ endow its $2$-cells with horizontal and
 vertical composition operations respectively, and that a duoid
 structure---involving compatible vertical and horizontal
 composition---is an enrichment of $\mathbb X$ to a $2$-category with
 $2$-cells given by $A$.

 \subsection{Whiskering schemes and sesquicategories}
\label{sec:whisk-schem-sesq}
As is clear from~\eqref{eq:39}, the duoidal category
$\cat{Ds}(\mathbb X)$ is not normal, so that in order for our theory
of commutativity to be applicable we must first pass to its
normalization as in Section~\ref{sec:norm-duoid-categ}. From the above
classification of the tensor products on $\cat{Ds}(\mathbb X)$, we see
that to equip a derivation scheme $A$ with an
$\I \conv (\thg) \conv \I$-algebra structure is to give functions
$h \cdot (\thg) \colon A_{f,g} \rightarrow A_{hf,hg}$ and
$(\thg) \cdot k \colon A_{f,g} \rightarrow A_{fk,gk}$ for all
$k \colon x \rightarrow y$ and $f,g \colon y \rightarrow z$ and
$h \colon z \rightarrow w$ in $\mathbb X$, subject to the evident
associativity and unitality axioms; in other words, to endow the
$2$-cells in $A$ with a notion of whiskering. It seems reasonable to
call this structure a \emph{whiskering scheme}; note that it amounts
to a lifting of
$\cat{Hom}_\mathbb X \colon \mathbb X^\op \times \mathbb X \rightarrow
\cat{Set}$
through the set-of-objects functor $\cat{Gph} \rightarrow \cat{Set}$.

The category $\cat{Ws}(\mathbb X)$ of whiskering schemes is thus the
normalization of the duoidal $\cat{Ds}(\mathbb X)$ when equipped with
the lifted monoidal structure $(\comp, \I)$ and the bilinear quotient
$(\ast_\I, \I)$ of $(\ast, \J)$. To give a $\comp$-monoid structure on
a whiskering scheme is to endow it with a vertical composition of
$2$-cells which is compatible with whiskering, or in other words, to
make it into a sesquicategory; we have thus shown that the category of
$\comp$-monoids in $\cat{Ws}(\mathbb X)$ is the category
$\cat{Sesq}(\mathbb X)$ of sesquicategory structures on $\mathbb X$.

\subsection{Commutativity in sesquicategories}
\label{sec:comm-sesq}
We now interpret the basic notions of our general theory in the
context of the normal duoidal $\cat{Ws}(\mathbb X)$. Let $C$ be a
$\comp$-monoid in $\cat{Ws}(\mathbb X)$, corresponding to the
sesquicategory $\mathbb C$ extending $\mathbb X$, and let
$\alpha \colon A \rightarrow C \leftarrow B \colon \beta$ be a cospan
in $\cat{Ws}(\mathbb X)$. By Lemma~\ref{lem:2} and our explicit
description of the tensor products in $\cat{Ds}(\mathbb X)$, we see
that the cospan $(f,g)$ is commuting just when, for every
$f, g \colon x \rightarrow y$ and $h, k \colon y \rightarrow z$ in
$\mathbb X$, the diagram
\begin{equation}
  \cd[@!C@C-1.7em@R-1.5em]{
    & \sh{l}{1em}C_{h,k} \times C_{f,g} \ar[r]^-{(\thg \cdot g) \times (h \cdot \thg)} & 
    \sh{r}{1em}C_{hg, kg} \times C_{hf, hg} \ar[dr]^-{\comp} \\
    A_{h,k} \times B_{f,g} \ar[ur]^-{\alpha \times \beta} \ar[dr]_-{(\beta \times \alpha).c\ \ \ \ \ } & & & C_{hf, kg}\\
    & \sh{l}{1em}C_{f,g} \times C_{h,k} \ar[r]^-{(k \cdot \thg) \times (\thg \cdot f)} & 
    \sh{r}{1em}C_{kf, kg} \times C_{hf, kf} \ar[ur]_-{\comp}
  }
\end{equation}
commutes; which is equally to ask that, for each $x \in A_{h,k}$ and
$y \in B_{f,g}$, the interchange axiom holds for the composable pair
of $2$-cells 
\begin{equation}
  \cd[@C+2em]{
    x \ar@/^1em/[r]^-{f} \ar@/_1em/[r]_-{g} \dtwocell{r}{\alpha(x)} &
    y \ar@/^1em/[r]^-{h} \ar@/_1em/[r]_-{k} \dtwocell{r}{\beta(y)} &
    z
  }
\end{equation}
in the sesquicategory $\mathbb C$. In particular, we conclude that $C$
is a commutative $\circ$-monoid just when the sesquicategory
$\mathbb C$ is in fact a $2$-category.

\appendix
\section{$\V\text-\cat{Cat}$ is locally presentable when $\V$ is so}
In~\cite[Theorem~4.5]{Kelly2001V-Cat}, Lack and Kelly show that if the
locally presentable category $\V$ bears a monoidal biclosed structure,
then $\V\text-\cat{Cat}$ is locally presentable and moreover free
$\V$-categories exist. In this appendix, we generalise this result by
weakening the assumption of closedness.

\begin{Prop}
  \label{prop:22}
  If $\V$ is a locally presentable category equipped with a monoidal
  structure $(\comp, \I)$ for which each functor $A \comp (\thg)$ and
  $(\thg) \comp A$ is accessible, then $\V\text-\cat{Cat}$ is locally
  presentable and free $\V$-categories exist.
\end{Prop}
\begin{proof}
  Fixing a set $X$, write $\V\text-\cat{Cat}_X$ for the fibre of
  $\ob \colon \V\text-\cat{Cat} \rightarrow \cat{Set}$ over $X$, and
  similarly for $\V\text-\cat{Gph}_X$. We first show that each
  $U_X \colon \V\text-\cat{Cat}_X \rightarrow \V\text-\cat{Gph}_X$ is
  an accessible functor between accessible categories. To this
  end, define $P^0 \in \V\text-\cat{Gph}_X$ by $P^0(x,x) = I$ and
  $P^0(x,y) = 0$ for $x \neq y$, let
  $P^2 \colon (\V\text-\cat{Gph}_X)^2 \rightarrow \V\text-\cat{Gph}_X$
  and
  $P^3 \colon (\V\text-\cat{Gph}_X)^3 \rightarrow \V\text-\cat{Gph}_X$
  be defined by
  \begin{equation}
    \begin{aligned}
    P^2(A,B)(x,y) &= \sum_{z \in X} A(z,y) \comp B(x,z)\\
    P^3(A,B,C)(x,y) &= \sum_{z,w \in X} A(w,y) \comp B(z,w) \comp C(x,z)\rlap{ ,}
    \end{aligned}
  \end{equation}
  and for any $A \in \V\text-\cat{Gph}_X$, let
  \begin{equation}
    \begin{aligned}
    \phi^\ell &\colon P^3(A, A, A) \rightarrow P^2(P^2(A,A),A) & \qquad \phi^r &\colon P^3(A, A, A) \rightarrow P^2(A, P^2(A,A))\\
    \psi^\ell &\colon A \rightarrow P^2(P^0, A) & \psi^r &\colon A \rightarrow P^2(A, P^0)
    \end{aligned}
  \end{equation}
  denote the canonical comparison maps induced by the universal
  property of coproduct. In these terms, to endow $A \in \V\text-\cat{Gph}_X$
  with $\V$-category structure is to give
  maps $e \colon P^0 \rightarrow A$ and $m \colon
  P^2(A,A) \rightarrow A$ rendering commutative the diagrams:
  \begin{equation}
    \cd[@!C@C-4.5em@-0.5em]{
      & \sh{l}{2.5em}P^2(P^2(A,A),A) \ar[r]^-{P^2(m, 1)} &
      \sh{r}{1.5em}P^2(A,A) \ar[dr]^-{m} \\
      P^3(A, A, A) \ar[ur]^-{\phi^\ell} \ar[dr]_-{\phi^r} & & & A \\
      & \sh{l}{2.5em}P^2(A, P^2(A,A)) \ar[r]_-{P^2(1, m)} &
      \sh{r}{1.5em}P^2(A,A) \ar[ur]_-{m}
    } \quad 
    \cd{
      P^2(P^0, A) \ar[d]^-{P^2(e,1)} & A \ar@{=}[d] \ar[l]_-{\psi^\ell} \ar[r]^-{\psi^r} & P^2(A, P^0) \ar[d]_-{P^2(1,e)} \\
      P^2(A, A) \ar[r]_-{m} & A & P^2(A,A)\rlap{ .} \ar[l]^-{m}
    }
  \end{equation}
  Consequently $\V\text-\cat{Cat}_X$ can be constructed from
  $\V\text-\cat{Gph}_X$ using bilimits in $\cat{CAT}$: first, one
  takes the inserter $J \colon \E \rightarrow \V\text-\cat{Gph}_X$ of
  the endofunctors $A \mapsto P^2(A,A)$ and $A \mapsto A$ of
  $\V\text-\cat{Gph}_X$, and then the inserter
  $K \colon \G \rightarrow \E$ of the functors $\Delta P^0$ and
  $J \colon \E \rightarrow \V\text-\cat{Gph}_X$. An object of $\G$
  comprises $A \in \V\text-\cat{Gph}_X$ equipped with maps
  $e \colon P^0 \rightarrow A$ and $m \colon P^2(A,A) \rightarrow A$,
  and now the inclusion $L \colon \V\text-\cat{Cat}_X \rightarrow \G$
  exhibits $\V\text-\cat{Cat}_X$ as the joint inserter of the three
  parallel pairs of $2$-cells in $\cat{CAT}(\G , \V\text-\cat{Gph}_X)$
  corresponding to the three axioms displayed above. Now,
  $\V\text-\cat{Gph}_X = \V^{X \times X}$ is accessible since $\V$ is,
  while $P^0$, $P^2$, $P^3$ are accessible since the tensor product of
  $\V$ is so; since by~\cite[Theorem~5.1.6]{Makkai1989Accessible}, the
  $2$-category $\cat{ACC}$ of accessible categories and accessible
  functors is closed under bilimits in $\cat{CAT}$, we conclude that
  $U_X \colon \V\text-\cat{Cat}_X \rightarrow \V\text-\cat{Gph}_X$ is
  an accessible functor between accessible categories as required.

  Now for any map $f \colon X \rightarrow Y$ in $\cat{Set}$, the reindexing
  functor $\V^{f\times f} \colon \V^{Y \times Y} \rightarrow \V^{X
    \times X}$ is easily seen to lift to a functor
  \begin{equation}
    \cd{
      {\V\text-\cat{Cat}_Y} \ar@{-->}[r]^-{f^\ast} \ar[d]_{U_Y} &
      {\V\text-\cat{Cat}_X} \ar[d]^{U_X} \\
      {\V\text-\cat{Gph}_Y} \ar[r]_-{\V^{f \times f}} &
      {\V\text-\cat{Gph}_X}\rlap{ .}
    }
  \end{equation}
  Since $\V^{f \times f}$ has a right adjoint (given by right Kan
  extension), it preserves all colimits and is in particular
  accessible. On the other hand, composing $f^\ast$ with the
  bilimiting cone that defines $\V\text-\cat{Cat}_X$ over
  $\V\text-\cat{Gph}_X$ gives a cone of accessible categories and
  functors; whence by~\cite[Theorem~5.1.6]{Makkai1989Accessible}
  $f^\ast$ is also accessible. Consequently, the indexed categories
  $\V\text-\cat{Cat}_{(\thg)}$ and
  $\V\text-\cat{Gph}_{(\thg)} \colon \cat{Set}^\op \rightarrow
  \cat{CAT}$
  and the indexed transformation $U$ between them all take values in
  $\cat{ACC}$, so that by~\cite[Theorem~5.4]{Makkai1989Accessible},
  the induced functor
  $U \colon \V\text-\cat{Cat} \rightarrow \V\text-\cat{Gph}$ between
  total categories is an accessible functor between accessible
  categories. Now since $\V\text-\cat{Cat}$ and $\V\text-\cat{Gph}$ are
  also complete, they must be locally presentable; since $U$ is also
  continuous, it must by~\cite[Satz~14.6]{Gabriel1971Lokal} be a right
  adjoint as required.
\end{proof}

\bibliographystyle{acm}
\bibliography{bibdata}

\end{document}